\documentclass[11pt,a4paper]{article}
\usepackage{setspace}
\usepackage[utf8]{inputenc}
\usepackage[T1]{fontenc}
\usepackage{amsmath}
\usepackage{enumerate}
\usepackage{amssymb}
\usepackage{amsthm} 
\usepackage{graphicx}
\usepackage{blkarray}
\usepackage{MnSymbol}
\usepackage{accents}
\usepackage{adjustbox}
\usepackage{isomath}
\usepackage[all]{xy}

\usepackage{tikz-cd}
\usepackage{tkz-graph}
\usetikzlibrary{decorations.markings}
\usepackage{hyperref}
\usepackage{bigints}
\hypersetup{colorlinks, linkcolor=blue}
\usepackage{cleveref}
\usepackage[stable]{footmisc}

\DeclareMathOperator{\sv}{s}

\DeclareMathOperator{\Imm}{Im}

\renewcommand{\Im}{{\Imm}}

\theoremstyle{definition}
\newtheorem{defi}{Definition}[section]
\theoremstyle{definition}
\newtheorem{rmk}[defi]{Remark}
\theoremstyle{definition}
\newtheorem{ex}[defi]{Example}
\theoremstyle{plain}
\newtheorem{thm}{Theorem}[section]
\newtheorem{prop}[thm]{Proposition}
\newtheorem{lemma}[thm]{Lemma}
\newtheorem{cor}[thm]{Corollary}
\newtheorem*{thm*}{Theorem}

\newcommand{\fm}{\mathfrak{m}}
\newcommand{\fdr}{\mathfrak{dr}}

\newcommand{\Q}{\mathbb{Q}}

\newcommand{\R}{\mathbb{R}}

\newcommand{\Pp}{\mathbb{P}}
\newcommand{\E}{\mathcal{E}}

\newcommand{\Sp}{\operatorname{Spec}}

\newcommand{\Hom}{\operatorname{Hom}}

\newcommand{\Ext}{\operatorname{Ext}}

\newcommand{\drY}{\textrm{\sffamily{\textup{Y}}}}
\newcommand{\drX}{\textrm{\sffamily{\textup{X}}}}

\newcommand{\ZZ}{\mathbb{Z}}
\newcommand{\QQ}{\mathbb{Q}}
\newcommand{\CC}{\mathbb{C}}

\newcommand{\dR}{\mathrm{dR}}

\newcommand{\Beta}{\mathrm{B}}

\title{\textbf{The cosmic Galois group, the sunrise Feynman integral, and the relative completion of $\Gamma_1(6)$}}
\author{Matija Tapu\v{s}kovi\'{c}}
\date{}

\begin{document}
\maketitle
\begin{abstract}

 In the first part of this paper we study the coaction dual to the action of the cosmic Galois group on the motivic lift of the sunrise Feynman integral with generic masses and momenta, and we express its conjugates in terms of motivic lifts of Feynman integrals associated to related Feynman graphs. Only one of the conjugates of the motivic lift of the sunrise, other than itself, can be expressed in terms of motivic lifts of Feynman integrals of subquotient graphs. To relate the remaining conjugates to Feynman integrals we introduce a general tool: subdiving edges of a graph. We show that all motivic lifts of Feynman integrals associated to graphs obtained by subdividing edges from a graph $G$ are motivic periods of $G$ itself. This was conjectured by Brown in the case of graphs with no kinematic dependence. We also look at the single-valued periods associated to the functions on the motivic Galois group, i.e. the `de Rham periods', which appear in the coaction on the sunrise, and show that they are generalisations of Brown's non-holomorphic modular forms with two weights. In the second part of the paper we consider the relative completion of the torsor of paths on a modular curve and its periods, the theory of which is due to Brown and Hain. Brown studied the motivic periods of the relative completion of $\mathcal{M}_{1,1}$ with respect to the tangential base-point at infinity, and we generalise this to the case of the torsor of paths on any modular curve. We apply this to reprove the claim that the sunrise Feynman integral in the equal-mass case can be expressed in terms of Eichler integrals, periods of the underlying elliptic curve defined by one of the associated graph hypersurfaces, and powers of $2\pi i$.
    
\end{abstract}
\begin{section}{Introduction}
\subsection{Feynman integrals}
Let $G=(V_G,E_G,E_G^{ext})$ be a connected Feynman graph, where $V_G$ are vertices, $E_G$ are edges, and $E_G^{ext}$ are external half-edges, otherwise known as legs. To each internal edge $e \in E_G$ we assign its particle mass $m_e \in \R$. To each external edge $i \in E_G^{ext}$ we assign a momentum, which is a vector $q_i \in \R^d$ where $d \in \ZZ_{>0}$ is the dimension of space-time. A condition on momenta $\sum_{i \in E_G^{ext}} q_i = 0$, called momentum conservation, is assumed. Associate to each internal edge $e \in E_G$ a variable $\alpha_e$. The first Symanzik polynomial is defined to be
\begin{equation}
\label{first-symanzik}    
\Psi_G = \sum\limits_{T\subset G} \prod\limits_{e \not \in T} \alpha_e
\end{equation}
where the sum is over all spanning trees $T$ of the graph $G$. The second Symanzik polynomial is defined to be:
\begin{equation}
\label{second-Symanzik}
\Xi_G(m,q) = \sum\limits_{T_1 \cup T_2 \subset G} (q^{T_1})^2 \prod\limits_{e \not \in T_1 \cup T_2} \alpha_e + \left(\sum\limits_{e \in E_G} m_e^2\alpha_e\right)\Psi_G.
\end{equation}
where the first sum ranges over all spanning 2-trees $T = T_1 \cup T_2$ of $G$\footnote{A spanning 2-tree of a graph $G$ is a subgraph with 2 connected components, each of which is a tree.}. Here $q^{T_1}= \sum_{i \in E_{T_1}^{ext}} q_i$ is the sum of all incoming momenta entering $T_1$. By momentum conservation $q^{T_1}=-q^{T_2}$. The notation $q^2$ refers to the Euclidean scalar product in $\mathbb{R}^d$. Let $N_G$ be the number of edges of $G$, $F$ the number of non-trivial external momenta, and $h_G$ the number of loops. Let $q=\{q_1,...,q_{F}\}$ and $m=\{m_1,...,m_{N_G}\}$. We consider integrals of the form
\begin{equation}
\label{Feynman integral defi}
I_G(m,q) = \int_{\sigma} \omega_G(m,q) ,
\end{equation}
where 
\begin{equation}
\label{Feynman integrand defi}
\omega_G(m,q) = \frac{1}{\Psi^{d/2}_G}\left(\frac{\Psi_G}{\Xi_G}\right)^{N_G - h_G d/2} \Omega_G,
\end{equation}
and
\begin{equation}
\label{omega-def}
\Omega_G = \sum\limits_{i=1}^{N_G} (-1)^i\alpha_id\alpha_1 \wedge \ldots \wedge \widehat{d\alpha_i}\wedge \ldots \wedge d\alpha_{N_G}.
\end{equation}
The domain of integration is $\sigma = \{ \left[\alpha_1:\ldots : \alpha_{N_G}\right] : \alpha_i \geq 0 \} \subset \Pp^{N_G - 1}(\mathbb{R})$. If the Feynman integral associated to the graph $G$ converges the integral \eqref{Feynman integral defi} is that Feynman integral up to a Gamma value factor. We only treat the convergent case in this paper, leaving regularization for future work.  

There has been a significant amount of interest in the algebraic structure of Feynman integrals evaluating to multiple elliptic polylogarithms and iterated integrals of modular forms in recent years \cite{ABDGM,FH}, with a view towards understanding the structure of multi-loop Feynman integrals. In this paper we will be particularly interested in the case of the sunrise graph -- see figure \ref{sunrise_graph}. It is known that the Feynman integral associated to the sunrise in $d=2$ dimensions of space-time can be expressed in terms of elliptic dilogarithms and integrals of modular forms \cite{BKV2,BV}.

\begin{figure}[h]
\centering
\begin{tikzpicture}
\SetGraphUnit{3}
  
  \SetUpEdge[lw = 1pt,
  color      = black,
  labelcolor = white,
  labelstyle = {sloped,above,yshift=2pt}]
  
  \SetUpVertex[FillColor=black, MinSize=8pt, NoLabel]

  \Vertex[x=2,y=0]{1}
  \Vertex[x=6,y=0]{2}
  \Vertex[x=0,y=0,empty=true]{4}
  \Vertex[x=8,y=0,empty=true]{5}

  \Edge[label=2,style = {double}](1)(2)
  \tikzset{EdgeStyle/.append style = {bend left=60}}
  \Edge[label=1,style = {double}](1)(2)
  \Edge[label=3,style = {double}](2)(1)
  \tikzset{EdgeStyle/.style={postaction=decorate,decoration={markings,mark=at position 0.7 with {\arrow{latex}}}}}
  \Edge[label=$-q_1$](4)(1)
  \Edge[label=$q_1$](5)(2)
\end{tikzpicture}
\caption{The sunrise Feynman graph.}
\label{sunrise_graph}
\end{figure}
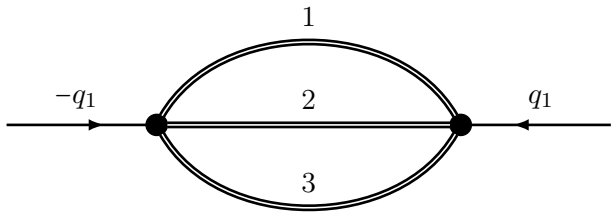

\subsection{Motivic Galois (co)action}
When it converges \eqref{Feynman integral defi} is a family of periods which conjecturally carry an action of the motivic Galois group \cite{KZ,Y}. In order to sidestep difficult conjectures regarding motives we will work in a Tannakian category of realisations $\mathcal{H}(S)$ consisting of triples $\mathcal{V} = (\mathbb{V}_{\Beta},\mathcal{V}_{\dR},c)$, which are typically given by Betti and algebraic de Rham cohomology of a family of algebraic varieties over a base $S$, together with a comparison isomorphism between them (see \cite[\S7.2]{Brown1} for details). Any reasonable category of motives admits a functor to the category of realisations, and any family of periods can be lifted to a Betti-de Rham matrix coefficient\footnote{A matrix coefficient is an element of the ring of functions on the scheme $Isom^{\otimes}_{\mathcal{H}(S)}(\omega_{\dR,Y},\omega_{\Beta,X})$, where $X \subset S(\CC)$ is a simply-connected region of $S(\CC)$, and $Y \subset S(\CC)$ is a region such that $Y \subset U(\CC)$ and $U \subset S$ affine. The functors $\omega_{\dR,Y},\omega_{\Beta,X}$  send a triple $\mathcal{V}$ to the sections of $\mathcal{V}_{\dR}$ over $Y$ and the sections of $\mathbb{V}_{\Beta}$ over $X$ respectively.} in $\mathcal{H}(S)$, i.e. to an equivalence class $\left[\mathcal{V},[\gamma],[\omega]\right]^{\fm}$ where $\mathcal{V}$ is an object of $\mathcal{H}(S)$, $[\gamma]$ is a section of the local system $\mathbb{V}_{\Beta}^{\vee}$ on some region of $S(\CC)$, and $[\omega]$ is a section of the vector bundle $\mathcal{V}_{\dR}$. They span the ring of Betti-de Rham matrix coefficients of $\mathcal{H}(S)$, denoted $\mathcal{P}^{\fm}_{\mathcal{H}(S)}$, and otherwise referred to as `motivic periods'. This ring is equipped with the period homomorphism $\textrm{per}:\mathcal{P}^{\fm}_{\mathcal{H}(S)} \rightarrow M(S(\CC))$ which sends a matrix coefficient to a multi-valued meromorphic function on $S(\CC)$, with the choice of branch corresponding to $[\gamma]$. In the case when $\mathcal{V}$ is  given by the cohomology of a family of varieties the period homomorphism is given by integration. For details see \cite[\S 7.5]{Brown2}.
\begin{ex}
\label{example log lefschetz}
Consider the object $ V = (H^1_{\Beta}(\mathbb{G}_m,\{1,x\}),H^1_{\dR}(\mathbb{G}_m,\{1,x\}),c)$ in $\mathcal{H}(\Sp(\QQ))$, for $x \in \mathbb{G}_m(\QQ)$. We define the motivic logarithm as the matrix coefficient
$$
\log^{\mathfrak{m}}(x) = \left[V,[\sigma_x],\left[\frac{dt}{t}\right]\right]^{\mathfrak{m}} \in \mathcal{P}^{\mathfrak{m}}_{\mathcal{H}(\Sp(\QQ))},
$$
where $\sigma_x$ is a path from $1$ to $x$ not winding around the origin. Its period is 
$$
\textrm{per}(\log^{\mathfrak{m}}(x)) = \int_{\sigma_x}\frac{dt}{t} = \log(x),
$$

To define the motivic lift of $2\pi i$ let $H=(H^1_{\textrm{B}}(\mathbb{G}_m),H^1_{\mathrm{dR}}(\mathbb{G}_m),c)$ be an object in $\mathcal{H}(\Sp(\QQ))$, and let
$$
\mathbb{L}^\mathfrak{m} = \left[H,[\gamma_0],\left[\frac{dt}{t}\right]\right]^{\mathfrak{m}} \in \mathcal{P}^{\mathfrak{m}}_{\mathcal{H}(\Sp(\QQ))},
$$
where $\gamma_0$ is a closed path winding once around the origin.
\end{ex}

A natural variant of motivic periods are de Rham-de Rham matrix coefficients, referred to as `de Rham periods', of the form $[\mathcal{V},[\eta]^{\vee},[\omega]]^{\dR}$, where $[\eta]^{\vee}$ is a section of $\mathcal{V}_{\dR}^{\vee}$, and $[\omega]$ is a section of $\mathcal{V}_{\dR}$ as before. De Rham periods span a Hopf algebra $\mathcal{P}^{\dR}_{\mathcal{H}(S)}$\footnote{These are elements of the Hopf algebra of functions on the group scheme $Aut^{\otimes}_{\mathcal{H}(S)}(\omega_{\dR,X})$.} which coacts on the ring of motivic periods:
\begin{equation} 
\label{general-coaction}
\Delta [\mathcal{V},[\gamma],[\omega]]^{\mathfrak{m}} = \sum\limits_{e_i} [\mathcal{V},[\gamma],[e_i]]^{\mathfrak{m}} \otimes [\mathcal{V},[e_i]^{\vee},[\omega]]^{\mathfrak{dr}}
\end{equation}
where $\{[e_i]\}$ is a basis of sections of $\mathcal{V}_{\dR}$ and $\{[e_i]^{\vee}\}$ is the dual basis. This coaction is dual to the action of the Tannaka group associated to a fiber functor of $\mathcal{H}(S)$ and denoted $G^{\dR}_{\mathcal{H}(S)}$, which plays the role of the motivic Galois group in our context. 
\begin{ex}
\label{example-motivic-dr-logs}
 We define the de Rham counterparts of the logarithm and $2\pi i$ by
$$\log^{\mathfrak{dr}}(x) = \left[V,\left[\frac{dt}{x-1}\right]^{\vee}, \left[\frac{dt}{t}\right]\right]^{\mathfrak{dr}}, \text{ and }
\mathbb{L}^\mathfrak{dr}= \left[H,\left[\frac{dt}{t}\right]^{\vee},\left[\frac{dt}{t}\right]\right]^{\mathfrak{dr}}
$$
respectively\footnote{We can see $\left[\frac{dt}{x-1}\right]^{\vee}$ and $\left[\frac{dt}{t}\right]^{\vee}$ as de Rham versions of paths from $1$ to $x$ and the loop around the origin respectively via a certain natural homomorphism called the de Rham projection -- for details see \cite[\S 6.3]{BD}.}. Applying the general formula for the coaction to the motivic logarithm we obtain
$$
\Delta \log^{\fm}(x) = \log^{\fm}(x) \otimes \mathbb{L}^{\fdr} + 1 \otimes \log^{\fdr}(x).
$$
\end{ex}

\subsubsection{Motivic Feynman amplitudes}
\label{motivic Feynman amplitude definition}
In the setting of $\mathcal{H}(S)$ we may lift \eqref{Feynman integral defi} to motivic periods by defining
\begin{equation}
    I^{\fm}_{G}(m,q) = \left[mot_G,[\sigma_G],[\pi_G^*(\omega_G(m,q))]\right]^{\fm},
\end{equation}
where $S$ is a Zariski open in the space of generic kinematics, $mot_G$ is an object of $\mathcal{H}(S)$ canonically associated to $G$. It is given by the cohomology of the pair $(P^G\setminus X_G, D)$, where $\pi_G: P^G \rightarrow P^{N_G-1}$ is an iterated blow-up along linear subspaces in $\Pp^{N_G-1}$ indexed by certain sub-graphs of $G$, $X_G$ is the strict transform of the union of hypersurfaces defined by $\Xi_G$ and $\Psi_G$, and $D$ is the divisor given by the complement of $X_G$ in the total transform of $\cup_{i=1}^{N_G} V(\alpha_i) \subset \Pp^{N_G-1}$. The section $[\sigma_G]$ is a constant section over a certain region in $S(\CC)$ given by the strict transform of the cycle $\sigma$ in \eqref{Feynman integral defi}. We refer to $mot_G$ as the `graph motive of $G$' (see \cite{Brown2} for details). The motivic period $I^{\fm}_{G}(m,q)$ is referred to as the `motivic Feynman amplitude' and the period homomorphism recovers the Feynman integral \eqref{Feynman integral defi} from it. The subring of $\mathcal{P}^{\fm}_{\mathcal{H}(S)}$ spanned by motivic periods of graph motives carries an action of a subquotient group of $G^{\dR}_{\mathcal{H}(S)}$ referred to as the `cosmic Galois group'. The upshot of studying the corresponding coaction on $I^{\fm}_G(m,q)$ is that we may obtain strong and concrete constraints on Feynman integrals.  In particular, we expect its motivic Galois conjugates of low weight\footnote{The weight filtration comes from the mixed Hodge structure on the cohomology groups associated to Feynman graphs.}  to be motivic periods of motives of subquotient graphs with the number of edges of those graphs controlled by the weight of the conjugates. This way, easy results for small graphs can provide strong constraints on Feynman integrals to all loop orders. This is referred to as \textit{the small graphs principle} \cite[8.4, 9.3]{Brown2}, and it inspired a significant amount of work as well as surprising conjectures \cite{ABDG,ABDGM,FH,PS,S}. Note that some of these results do not compute the motivic Galois coaction directly, but should be compatible with it. It would be interesting to explore these connections further. In this paper we focus on the coaction on $I^{\fm}_{G}(m,q)$ where $G$ is the sunrise graph with generic kinematics.

\subsection{de Rham elliptic integrals and their single-valued periods}
The second Symanzik polynomial of the sunrise graph gives rise to an elliptic curve defined by its vanishing locus, and the associated graph motive can be shown to be an (iterated) extension of the cohomology of this elliptic curve. We will therefore first need to consider the motivic and de Rham periods arising from the cohomology of a family of elliptic curves and the cohomology of a family of elliptic curves relative to points on that curve. In particular, we will introduce de Rham versions of complete and incomplete elliptic integrals. Let $H^1(\mathcal{E})$ denote the object of $\mathcal{H}(S)$ given by the Betti and algebraic de Rham cohomology of a family of elliptic curves $\mathcal{E} \rightarrow S$. Choose a basis of sections of $H^1_{\dR}(\mathcal{E})$, over some open in $S$, and denote it $\{\omega,\eta\}$, where $\omega,\eta$ are differential forms of the first kind and second kind respectively. We define the following de Rham periods
$$
 K_1^{\fdr} = \left[H^1(\mathcal{E}), [\omega]^{\vee} ,[\omega] \right]^{\fdr}, \quad K_{2,\eta}^{\fdr} = \left[H^1(\mathcal{E}), [\eta]^{\vee} ,[\omega] \right]^{\fdr}.
$$
In the above notation we suppress the dependence on parameters unless explicitly required. Note that $K_1^{\fdr}$ is canonically defined for an elliptic curve $\mathcal{E}$ since the the form of the first kind is well defined, but for $K_{2,\eta}^{\fdr}$ we have a choice of the differential form of the second kind, since we can always add to it a multiple of the form of the first kind. De Rham periods are not equipped with a period homomorphism, but we can consider their single-valued periods instead. This is a construction which assigns a single-valued function of the parameters to a de Rham period -- for more on this topic see \cite{BD}. For the construction of the single-valued period homomorphism for families of de Rham periods see Appendix \ref{single-valued-periods definition}. Applying the single-valued period homomorphism, denoted $\mathrm{s}$, to $K_1^{\fdr}, K_{2,\eta}^{\fdr}$ for the universal elliptic curve $\mathcal{E} \rightarrow \mathcal{M}_{1,1}$ depending on $\tau = \frac{\omega_1}{\omega_2}$, where $\omega_i$ are the two periods of the elliptic curve, yields examples of Brown's `non-holomorphic modular forms' -- see \cite{Brown4} and sequels. These functions transform as
$$
f_k(\gamma \tau) = (c\tau + d)^{-k}(c\overline{\tau} + d)^{k}f_k(\tau),
$$
where $\mathrm{s}(K_1^{\fdr}(\tau)) = f_1(\tau)$,
and $\mathrm{s}(K_{2,\eta}^{\fdr}(\tau)) = f_{-1}(\tau)$, for $\gamma = \begin{pmatrix}
a & b \\
c & d
\end{pmatrix} \in \textrm{SL}_2(\mathbb{Z})$.

We extend this to incomplete de Rham elliptic integrals by considering $\mathcal{E}$ with two points $(\mathcal{E},\{P,Q\})\rightarrow S$. We choose a basis of sections of $H^1_{\dR}(\mathcal{E},\{P,Q\})$ given by $\{\omega,\eta,df\}$, where $f$ is a meromorphic function on $\mathcal{E}$ such that $f(P) - f(Q) = 1$. Consider the de Rham periods
$$
F_{P,Q}^{\fdr} = \left[H^1(\mathcal{E},\{P,Q\}), [df]^{\vee} ,[\omega] \right]^{\fdr}, F_{P,Q,\eta}^{\fdr} = \left[H^1(\mathcal{E},\{P,Q\}), [df]^{\vee} ,[\omega] \right]^{\fdr}
$$
Define $f_{3,-1}(\tau,z_1,z_2) = \sv\left(F^{\fdr}_{P,Q}(\tau)\right)$ and $f_{3,1}(\tau,z_1,z_2) = \sv\left(E^{\fdr}_{P,Q,\eta}(\tau)\right)$. Then the following holds.
\begin{prop}
 Let $z_1,z_2$ be the preimages of $P,Q$ under $\CC/\Lambda_{\tau} \cong \mathcal{E}_{\tau}(\CC)$ for $\Lambda_{\tau}$ the lattice spanned by the periods of $\mathcal{E}_{\tau}$, and $\gamma = \begin{pmatrix}
a & b \\
c & d
\end{pmatrix} \in \textrm{SL}_2(\mathbb{Z})$. Then
\begin{equation}
\begin{split}
    & f_{3,k}\left(\frac{a\tau + b}{c\tau + d},\frac{z_1}{c\tau + d},\frac{z_2}{c\tau + d}\right) = (c\tau + d)^{k}f_{3,1}(\tau,z_1,z_2).
\end{split}
\end{equation}
and they have expansions of the form
$$
f(\tau,z_1,z_2) = \sum_{-N \leq k_1,k_2,k_3 \leq N} \sum_{\substack{m_i,n_i \geq 0 \\ 1 \leq i \leq 3}}a_{m,n}^{(k_1,k_2,k_3)}z_1^{k_1}z_2^{k_2}(-2\pi\Imm{\tau})^{k_3}q_{i}^{m_i}\overline{q_i}^{n_i},
$$
where $q_1 = e^{2\pi i \tau},q_2 = e^{2\pi i z_1}, q_3 = e^{2\pi i z_2}$, and $z_1,z_2$ are in a compact set not containing any points of the lattice generated by $1$ and $\tau$.
\end{prop}

\subsection{Subdivision of edges}
Note that the small graphs principle does not imply that we can write all conjugates of a motivic Feynman amplitude as motivic Feynman amplitudes of other graphs. We will see in particular that in the case of the sunrise graph only one of the conjugates, other than the sunrise itself, can be expressed in terms of its subquotient graphs\footnote{As another example see the coaction for the three-edge one-loop graph with vanishing internal masses and contrast this with, for example, the four-edge one-loop graph in \cite{MatijaT}.}. To address this we will introduce a general tool, applicable to any connected graph $G$, in \S\ref{section: subdivison of edges} as follows. Consider a connected Feynman graph $G$, and let $G_{s(e)}$ be the graph $G$ with an edge $e$ replaced by two edges connected by a vertex, where the two new edges have the same mass as the edge $e$ (see figure \ref{fig:subdivided sunrise}). We will show the following.

\begin{figure}[h]
\centering
\begin{tikzpicture}[scale=0.7]
\SetGraphUnit{3}
  
  \SetUpEdge[lw = 1pt,
  color      = black,
  labelcolor = white,
  labelstyle = {sloped,above,yshift=2pt}]
  
  \SetUpVertex[FillColor=black, MinSize=8pt, NoLabel]

  \Vertex[x=2,y=0]{1}
  \Vertex[x=6,y=0]{2}
  \Vertex[x=4,y=1.2]{3}
  \Vertex[x=0,y=0,empty=true]{4}
  \Vertex[x=8,y=0,empty=true]{5}
    \scriptsize
  \Edge[label=3,style = {double}](1)(2)
  
  \tikzset{EdgeStyle/.append style = {bend left=60}}
  \Edge[label=4,style = {double}](2)(1)
  
  \tikzset{EdgeStyle/.style = {bend left=20}}
  \Edge[label=1,style = {double}](1)(3)
  \Edge[label=2,style = {double}](3)(2)
  \tikzset{EdgeStyle/.style={postaction=decorate,decoration={markings,mark=at position 0.7 with {\arrow{latex}}}}}
  \tikzset{LabelStyle/.style = {sloped,above}}
  
  \Edge[label=$-q_1$](4)(1)
  \Edge[label=$q_1$](5)(2)
\end{tikzpicture}
\quad
\raisebox{2px}{
\begin{tikzpicture}[scale=0.7]
\SetGraphUnit{3}
  
  \SetUpEdge[lw = 1pt,
  color      = black,
  labelcolor = white,
  labelstyle = {sloped,above,yshift=2pt}]
  
  \SetUpVertex[FillColor=black, MinSize=8pt, NoLabel]

  \Vertex[x=2,y=0]{1}
  \Vertex[x=6,y=0]{2}
  \Vertex[x=4,y=1.2]{3}
  \Vertex[x=0,y=0,empty=true]{4}
  \Vertex[x=8,y=0,empty=true]{5}
  \Vertex[x=3.3,y=0]{6}
  \Vertex[x=4.6,y=0]{7}
  
  \scriptsize     
  \tikzset{LabelStyle/.style = {yshift=-1pt}}
  \Edge[label=3,style = {double}](1)(6)
  \Edge[label=4,style = {double}](6)(7)
  \Edge[label=5,style = {double}](7)(2)
 \tikzset{LabelStyle/.style = {yshift=1pt}}
  \tikzset{EdgeStyle/.append style = {bend left=60}}
  \Edge[label=6,style = {double}](2)(1)
  \scriptsize
  \tikzset{EdgeStyle/.style = {bend left=20}}
  \Edge[label=1,style = {double}](1)(3)
  \Edge[label=2,style = {double}](3)(2)
  \tikzset{EdgeStyle/.style={postaction=decorate,decoration={markings,mark=at position 0.7 with {\arrow{latex}}}}}
  \tikzset{LabelStyle/.style = {sloped,above}}
  \normalsize
  \Edge[label=$-q_1$](4)(1)
  \Edge[label=$q_1$](5)(2)

\end{tikzpicture}
}
\begin{tikzpicture}[scale=0.7]
\SetGraphUnit{3}
  
  \SetUpEdge[lw = 1pt,
  color      = black,
  labelcolor = white,
  labelstyle = {sloped,above,yshift=2pt}]
  
  \SetUpVertex[FillColor=black, MinSize=8pt, NoLabel]

  \Vertex[x=2,y=0]{1}
  \Vertex[x=6,y=0]{2}
  \Vertex[x=4,y=-1.2]{3}
  \Vertex[x=0,y=0,empty=true]{4}
  \Vertex[x=8,y=0,empty=true]{5}
  \Vertex[x=3.3,y=1.1]{6}
  \Vertex[x=4.6,y=1.1]{7}

  \scriptsize
  \Edge[label=4,style = {double}](1)(2)
  \tikzset{EdgeStyle/.append style = {bend left=10}}
  \Edge[label=1,style = {double}](1)(6)
  \Edge[label=2,style = {double}](6)(7)
  \Edge[label=3,style = {double}](7)(2)
  \tikzset{EdgeStyle/.style = {bend left=20}}
  \tikzset{LabelStyle/.style = {above}}
  \scriptsize
  \Edge[label=6,style = {double}](2)(3)
  \Edge[label=5,style = {double}](3)(1)
  \tikzset{EdgeStyle/.style={postaction=decorate,decoration={markings,mark=at position 0.7 with {\arrow{latex}}}}}
  \tikzset{LabelStyle/.style = {sloped,above}}
  \normalsize
  \Edge[label=$-q_1$](4)(1)
  \Edge[label=$q_1$](5)(2)
\end{tikzpicture}
\quad
\raisebox{2px}{
\begin{tikzpicture}[scale=0.7]
\SetGraphUnit{3}
  
  \SetUpEdge[lw = 1pt,
  color      = black,
  labelcolor = white,
  labelstyle = {sloped,above,yshift=2pt}]
  
  \SetUpVertex[FillColor=black, MinSize=8pt, NoLabel]

  \Vertex[x=2,y=0]{1}
  \Vertex[x=6,y=0]{2}
  \Vertex[x=3.3,y=-1.2]{7}
  \Vertex[x=4.7,y=-1.2]{8}
  \Vertex[x=0,y=0,empty=true]{4}
  \Vertex[x=8,y=0,empty=true]{5}
  \Vertex[x=4,y=0]{6}
  
  \scriptsize
  \Edge[label=2,style = {double}](1)(6)
  \Edge[label=3,style = {double}](6)(2)
  \tikzset{EdgeStyle/.append style = {bend left=60}}
  \Edge[label=1,style = {double}](1)(2)
  \tikzset{EdgeStyle/.style = {bend right=10}}
  \tikzset{LabelStyle/.style = {above,yshift=-1pt}}

  \Edge[label=4,style = {double}](1)(7)
  \Edge[label=5,style = {double}](7)(8)
  \Edge[label=6,style = {double}](8)(2)
  \tikzset{EdgeStyle/.style={postaction=decorate,decoration={markings,mark=at position 0.7 with {\arrow{latex}}}}}
  \tikzset{LabelStyle/.style = {sloped,above}}
  \normalsize
  \Edge[label=$-q_1$](4)(1)
  \Edge[label=$q_1$](5)(2)
\end{tikzpicture}

}

\caption{Graphs obtained from the sunrise by subdividing edges, denoted $G_{s(e_1)}$, $G_{s(e_1,e_2^2)}$, $G_{s(e_1^2,e_3)}$, $G_{s(e_2,e_3^2)}$, top-left to bottom-right respectively.}
\label{fig:subdivided sunrise}
\end{figure}
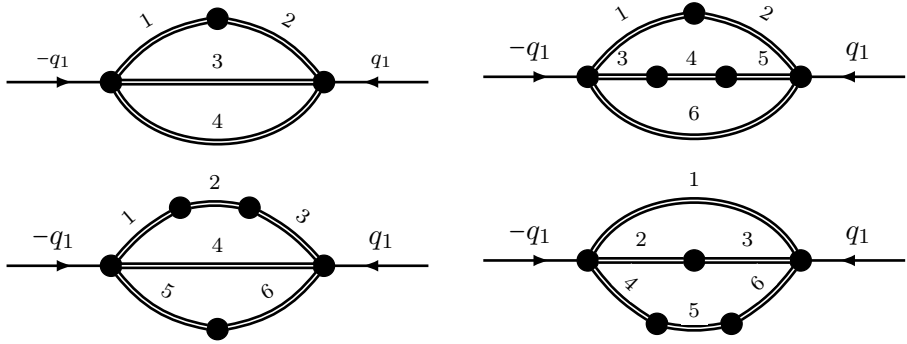

\begin{prop}
Let
\begin{equation}
    \begin{split}
    \rho : \,\, &\mathbb{A}^1 \times \mathbb{P}^{N_G-1} \longrightarrow \mathbb{P}^{N_G}\\
    & (y,\alpha_1,\ldots,\alpha_{N_G}) \mapsto (\alpha_1,\ldots,y\alpha_{N_G},(1-y)\alpha_{N_G}).
    \end{split}
\end{equation}
There exists a morphism $\tilde{\rho}$ such that the following diagram commutes
$$\xymatrix{
            \mathbb{A}^1 \times P^G \ar[d]^{\textrm{id} \times \pi_G} \ar[r]^{\tilde{\rho}}& P^{G_{s(e)}} \ar[d]^{\pi_{G_{s(e)}}}\\
             \mathbb{A}^1 \times \mathbb{P}^{N_G-1} \ar[r]^{\rho} & \mathbb{P}^{N_G}
            } 
$$
The morphism $\tilde{\rho}$ induces a morphism of graph motives
$$
mot_{G_{s(e)}} \xrightarrow{\tilde{\rho}^*} mot_G \otimes \mathbb{Q}(0) .
$$
\end{prop}
We will show that composing morphisms $\tilde{\rho}$ allows us to write the following equivalence of motivic periods.
\begin{cor}
Let $I=\{e_1^{k_1},\ldots,e_{N_{G}}^{k_{N_{G}}}\}$ for each $k_i \geq 0$ an integer, and let $G_{s(I)}$ be the graph obtained by subdividing $k_i$ times the edge $e_i$ for $ 1 \leq i \leq N_{G}$. Let $I_{G_{s(I)}}^{\fm}(m,q)$ be its motivic Feynman amplitude in $d_{G_{s(I)}}$ dimensions, $\omega_G(m,q)$ be the Feynman integrand of $G$ in $d_G$ dimensions, and $K = \sum_{i=1}^{N_G}k_i$. Then 
\begin{equation*}
I_{G_{s(I)}}^{\fm}(m,q) =  \left[mot_G, \left[\sigma_G\right], \left[\pi_{G}^*\left((-1)^K\alpha_{1}^{k_1}\cdots\alpha_{N_G}^{k_{N_G}}\Psi_G^{z_1}\Xi_G^{z_2}\omega_G(m,q))\right)\right] \right]^{\fm},
\end{equation*}
 where $z_1 = K - (h_G/2+1/2)(d_{G_{s(I)}}-d_G)$, and $z_2 = h_G/2(d_{G_{s(I)}}-d_G)-K$.
\end{cor}

It was conjectured by Brown in \cite[\S 9.4]{Brown2} that the morphism 
\begin{equation}
\begin{split}
\left\{\text{Graphs with } m=q=0\right\} & \rightarrow \textrm{Rep}_{\QQ}\left(G^{\dR}_{\mathcal{H}(\Sp(\QQ))}\right) \\
G &\mapsto \mathcal{FP}^{\fm}(G),
\end{split}
\end{equation} 
where $\mathcal{FP}^{\fm}(G) = [mot_G,[\sigma_G],[\omega]]$ and $\omega$ is a section of $(mot_G)_{\dR}$, factors through the quotient given by the relation $G \sim G'$ when $G'$ is obtained from $G$ by subdividing an edge. This is simply a restating of the above corollary with $m=q=0$.

By applying subdivision of edges to the case of the sunrise we will be able to write all conjugates in terms of motivic Feynman amplitudes of subquotient graphs and graphs obtained by subdividing edges. In order for us to be able to write all conjugates of a general motivic Feynman amplitude in terms of motivic Feynman amplitudes of other graphs it is necessary that the de Rham realisations of all graph motives are spanned by globally defined algebraic differential forms on the underlying geometric space. At the time of writing it is not known if this holds.

\subsubsection{Coaction on the sunrise with generic masses}

Let $G$ be the sunrise graph, and let $\mathcal{E}$ be the family of elliptic curves over the space of kinematics given by the vanishing locus of the second Symanzik polynomial, with its rational point $[0:1:0]$. Let $I^{\fm}_{G_{s(I)}}$, for $I \in \{\{e_1\},\{e_1,e_2^2\},\{e_1^2,e_3\},\{e_2,e_3^2\}\}$, be the motivic Feynman amplitudes of the graphs in figure \ref{fig:subdivided sunrise}, where the first one is taken to be in $d=2$ and the rest in $d=4$. We will show that the classes of the integrands of these graphs along with that of the sunrise itself and the class of a differential form given by the image of the Feynman integrand of $G \setminus e_3$ in $d=4$ under the face map $mot_{G \setminus e_3} \rightarrow mot_G$ (see \cite[\S 10.3]{Brown1} for face maps) span a de Rham basis of $mot_G$. Let $K_1^{\fdr},K_{2,\eta}^{\fdr}$ be the de Rham elliptic integrals associated to this elliptic curve, and $F^{\fdr}_{\underline{b}}$ be a $k_S$-linear combination of $F^{\fdr}_{P_1,P_i}$ for $P_1,\ldots,P_6$ the six points $D \cap \mathcal{E}$ (see figure \ref{figure_sunrise_restricted}), and $k_S$ the field of rational functions of masses and momenta. These linear combinations are given explicitly in \S\ref{Arbitrary masses coaction section}. Let $I^{\fdr}_{G,G \setminus e_3}$ be the de Rham period of $mot_G$ associated to the class of the Feynman integrand of the sunrise in $d=2$ dimensions, and the dual\footnote{This is the dual basis with respect to the previously chosen basis of Feynman integrands.} of the class corresponding to $G \setminus e_3$ as described above. Let $I^{\fm}_G$ be the motivic Feynman amplitude associated to the sunrise in $d=2$ space-time dimensions. 
\begin{thm} 
The motivic Galois coaction on the motivic Feynman amplitude associated to the sunrise Feynman graph with respect to the de Rham basis defined in Proposition \ref{basis-elliptic-dr} is
\begin{equation}
\begin{split}
    \Delta(I^{\fm}_G) &= I^{\fm}_G \otimes K_1^{\fdr}\mathbb{L}^{\fdr} + I^{\fm}_{G_{s(e_1)}} \otimes K_{2,\eta}^{\fdr}\mathbb{L}^{\fdr} + I^{\fm}_{G \setminus e_3} \otimes I^{\fdr}_{G,G \setminus e_3} +\\
    & +I^{\fm}_{G_{s(e_1,e_2^2)}} \otimes F^{\fdr}_{\underline{b_1}}\mathbb{L}^{\fdr}  + I^{\fm}_{G_{s(e_1^2,e_3)}} \otimes F^{\fdr}_{\underline{b_2}}\mathbb{L}^{\fdr} + I^{\fm}_{G_{s(e_2,e_3^2)}} \otimes F^{\fdr}_{\underline{b_3}}\mathbb{L}^{\fdr} .
\end{split}
\end{equation}
\end{thm}
By considering appropriate linear combinations of graphs over $k_S$, we may change the basis to simplify the expression on the motivic side of the coaction and obtain
\begin{equation}
\begin{split}
    \Delta(I^{\fm}_G) &= I^{\fm}_G \otimes K_1^{\fdr}\mathbb{L}^{\fdr} + I^{\fm}_{G_{s(e_1)}} \otimes K_{2,\eta}^{\fdr}\mathbb{L}^{\fdr} + I^{\fm}_{G \setminus e_3} \otimes I^{\fdr}_{G,G \setminus e_3} +\\ &+\log^{\fm} \left(\frac{m_3^2}{m_2^2} \right) \otimes F^{\fdr}_{\underline{b'_1}}\mathbb{L}^{\fdr}
    +\log^{\fm} \left(\frac{m_1^2}{m_3^2} \right)\otimes F^{\fdr}_{\underline{b'_2}}\mathbb{L}^{\fdr}
    +\log^{\fm} \left(\frac{m_2^2}{m_1^2} \right) \otimes F^{\fdr}_{\underline{b'_3}}\mathbb{L}^{\fdr}  \, .
\end{split}
\end{equation}
Finally, in the equal-mass case certain terms of the coaction vanish, as is apparent from the last expression involving logs of ratios of masses, and we have:
\begin{thm}
The motivic Galois coaction for the equal-mass sunrise can be expressed as
\begin{equation}
\begin{split}
    \Delta(I^{\fm}_G) &= I^{\fm}_G \otimes K_1^{\fdr}\mathbb{L}^{\fdr} + I^{\fm}_{G_{s(e_1)}} \otimes K_{2,\eta}^{\fdr}\mathbb{L}^{\fdr} + I^{\fm}_{G \setminus e_3} \otimes I^{\fdr}_{G,G \setminus e_3}.
\end{split}
\end{equation}
\end{thm}

This concludes the first part of the paper.

\subsection{Motivic periods of the relative completion of the torsor of paths on a modular curve and the sunrise}
The theory of the relative completion of path torsors of modular curves $\pi_1^{\textrm{rel}}(X_{\Gamma},x,y)$, where $X_{\Gamma}$ is a modular curve for $\Gamma$ a congruence subgroup and $x,y\in X_{\Gamma}(k)$ was studied by Hain \cite{Hain1,Hain2} and Brown \cite{Brown5}. The periods of the affine ring of functions $\mathcal{O}(\pi_1^{\textrm{rel}}(X_{\Gamma},x,y))$ are given by products of iterated Eichler integrals, powers of periods of $\mathcal{E}$ and powers of $2\pi i$. In the second part of this paper we describe the motivic lifts of those periods generalising \cite[\S 15]{Brown5} where $\pi_1^{\textrm{rel}}(\mathcal{M}_{1,1},\vec{1}_{\infty})$ is considered. We then show the following.
\begin{thm}
    Let $f: \mathcal{E} \rightarrow X_{\Gamma}$ be the universal elliptic curve for the modular curve $X_{\Gamma}$. Let $\mathcal{V} = (\mathbb{V}_{\Beta},\mathcal{V}_{\dR},c) \in \mathcal{H}(X_{\Gamma})$ be such that the local system $\mathbb{V}_{\Beta}$ has a filtration
    $$
    0 \subset \mathbb{V}_{\Beta}^1 \subset \mathbb{V}_{\Beta}^2 \subset \cdots \subset \mathbb{V}_{\Beta}^k = \mathbb{V}_{\Beta}, 
    $$
    such that 
    $$
    \mathbb{V}_{\Beta}^i/\mathbb{V}_{\Beta}^{i-1} \in Sym^k(H^1_{\Beta}(\mathcal{E}/X_{\Gamma})(r),
    $$
    where $k,r \in \ZZ$ for all $0<i\leq k$. Furthermore we assume that $\mathcal{V}_{\dR}$ has an analogous filtration such that $\mathcal{V}_{\dR}^i/\mathcal{V}_{\dR}^{i-1} \in Sym^k(H^1_{\dR}(\mathcal{E}/X_{\Gamma}))$. Let $\mathcal{P}^{\fm}_{\mathcal{V}_x}$ be the span over $k$ of $[\mathcal{V}_x,\sigma,\omega]^{\fm}$ where $\sigma \in \mathbb{V}_{\Beta,x}^{\vee}$ and $\omega \in \mathcal{V}_{\dR,x}$, i.e. the periods of the `fiber at $x$'. Then
    $$
        \mathcal{P}^{\fm}_{\mathcal{V}_y} \subset \mathcal{P}^{\fm}_{\mathcal{V}_x} \otimes \mathcal{O}(\pi_1^{\textrm{rel}}(X_{\Gamma},x,y)).
    $$
\end{thm}

This implies that periods of such an object $\mathcal{V}$ at $y$ can be expressed as products of iterated Eichler integrals from $x$ to $y$ and periods of $\mathcal{V}$ at $x$. Turning back to the sunrise, it is known from \cite{BV} that in the equal-mass case the elliptic curve associated to the sunrise is the universal family $\mathcal{E} \rightarrow X_1(6)$.
\begin{cor}
 The motivic Feynman amplitude of the sunrise in the equal-mass case is equivalent to a motivic period of an object $E \in \mathcal{H}(X_1(6))$ satisfying the conditions of the previous theorem. The sunrise Feynman integral can therefore be written as a $k$-linear combination of products of iterated Eichler integrals, powers of periods of $\mathcal{E}$ and powers of $2\pi i$.
\end{cor}

Since we know the mixed Hodge structure on $\mathcal{O}(\pi_1^{\textrm{rel}}(X_1(6),x,y)$ by \cite{Hain2} the above theorem implies that one can write down an antsaz for the sunrise motivic Feynman amplitude. This gives an alternative strategy to proving the results of \cite{BV,AW}, which also applies whenever the motivic Feynman amplitude can be expressed as a motivic period of an object satisfying the assumptions of the previous theorem.

\begin{subsection}*{Acknowledgements}
This work was supported by the Engineering and Physical Sciences Research Council [grant number EP/W020793/1]; and the European Research Council(ERC) under the European Union's Horizon 2020 research and innovation programme (grant agreement No. 724638). The author owes special thanks to Francis Brown and Erik Panzer for discussions. Thanks are also owed to Tiago Fonseca and Federico Zerbini.
\end{subsection}

\end{section}

\part{The motivic coaction on the sunrise}
\section{Motivic and de Rham periods of $H^1(\mathcal{E},\{P,Q\})$}

Let $k \subset \CC$ be a field, and let $\mathcal{E} \subset \mathbb{P}^2$ be a smooth elliptic curve defined over $k$. Let $P$ and $Q$ be two $k$-points on $\mathcal{E}$, and
\begin{equation}
\label{de Rham and Betti elliptic two points}
H_{\dR} = H^1_{\dR}(\mathcal{E},\{P,Q\}) \text{ and } H_{\Beta,{\varphi}} = H^1_{B}(\mathcal{E}_{\varphi},\{P,Q\}; \mathbb{Q})
\end{equation}
be algebraic de Rham and Betti cohomology respectively, where we define $\mathcal{E}_{\varphi} = \mathcal{E} \times_{k,\varphi} \CC$ for every embedding $\varphi: k \hookrightarrow \CC$\footnote{Note that the embedding data is necessary to define the comparison isomorphism since $H_{\dR}$ is a $k$-vector space which we must complexify, which depends on the embedding. Similarly, in order to speak of homology cycles we must take the complex points of $\mathcal{E}$, which depends on the embedding. This subtlety does not arise when the variety is defined over $\QQ$ since then we have a unique embedding. For its role in the study of Artin motives and Grothendieck's version of classical Galois theory see \cite[\S 5.1]{Brown1}.}. Note that $H^1(\mathcal{E},\{P,Q\})$ sits in a long exact sequence  
\begin{equation}
\label{ses}
0 \rightarrow \widetilde{H}^0(\{P,Q\}) \rightarrow H^1(\mathcal{E},\{P,Q\}) \rightarrow H^1(\mathcal{E}) \rightarrow 0
\end{equation}
where on the left we have reduced cohomology.

\subsection{de Rham realization}
\label{standard elliptic de rham realization}
Let $\mathcal{E}$ be given by the Weierstrass equation
$$
y^2z=x^3+axz^2+bz^3 , 
$$
for some $a,b \in k$. Then we may consider the following differential forms
$$
\omega = \frac{dx}{2y}, \eta = \frac{xdx}{2y}, \text{ and } \xi_{P,Q} =df.
$$ 
Here $\omega$ is a differential form of the first kind, $\eta$ is a differential form of the second kind with a double pole at $[0:1:0]$ and vanishing residue, and $f$ is a meromorphic function on $\mathcal{E}$ such that $f(P)-f(Q)=1$. By abuse of notation we denote by $[\omega]$ and $[\eta]$ the corresponding classes of $H_{\rm{dR}}^1(\mathcal{E},\{P,Q\})$ (see remark \ref{triple-complex remark} below for details). The class of the form $\xi_{P,Q}$ generates $\widetilde{H}^0_{\rm{dR}}(\{P,Q\})$, and we denote by $[\xi_{P,Q}]$ its image under the injective map $\widetilde{H}^0_{\rm{dR}}(\{P,Q\}) \rightarrow H_{\rm{dR}}^1(\mathcal{E},\{P,Q\})$. 

\begin{rmk}
\label{triple-complex remark}
Consider a covering $\mathcal{U}$ of $\mathbb{P}^2$ given by three affine pieces: $U_1$ given by $z \not = 0$, $U_2$ given by $y \not = 0$ and $U_3$ given by $x \not = 0$. In fact it will be enough to consider $U_1 \cup U_2$ since its complement in $\mathbb{P}^2$ is $(1:0:0)$ and this does not lie on our curve. Let the coordinates on $U_1$ be denoted by $(x,y,1)$, and coordinates on $U_2$ by $(u,1,v)$. The total complex of the relative algebraic \v Cech-de Rham triple complex below computes $H^1_{\rm{dR}}(\mathcal{E},\{P,Q\})$.

\begin{center}
\begin{tikzcd}[row sep=scriptsize, column sep=scriptsize,scale=0.7]
& 0 \arrow[rr] \arrow[from=dd] & & 0  \\
\Omega^1(U_1) \oplus \Omega^1(U_2) \arrow[ur] \arrow[rr, crossing over] & & \Omega^1(U_1 \cap U_2) \arrow[ur] \\
& \bigoplus_{i=1,2}\mathcal{O}(P \cap U_i) \oplus \mathcal{O}(Q \cap U_i) \arrow[rr]  & & 0 \arrow[uu] \\
\mathcal{O}(U_1) \oplus \mathcal{O}(U_2) \arrow[rr] \arrow[uu] \arrow[ur] & & \mathcal{O}(U_1 \cap U_2) \arrow[ur] \arrow[uu, crossing over]\\
\end{tikzcd}
\end{center}
In the diagram the vertical arrows are given by differentiation, the horizontal by taking the difference, and the diagonal by evaluating at $P,Q$. An element of $H^1_{dR}(\mathcal{E},\{P,Q\})$ is a class represented by a tuple in 
\begin{equation}
\label{touples}
\Omega^1(U_1) \oplus \Omega^1(U_2) \oplus \mathcal{O}(U_1 \cap U_2) \oplus \bigoplus_{i=1,2}\mathcal{O}(P \cap U_i) \oplus \mathcal{O}(Q \cap U_i)
\end{equation}
which satisfies the cocycle condition in the total complex. A basis of de Rham cohomology is given by the classes of the following cocycles
\begin{equation}
    \label{basis-triple-complex}
    \begin{split}
    & \left[\left(\frac{dx}{2y},\frac{du}{1-2auv-3bv^2},0,\ldots,0\right)\right],\\
    & \left[\left(\frac{xdx}{2y},\frac{udu}{v(1-2auv-3bv^2)} - d\left(\frac{1}{u^2}\right),\frac{1}{u^2},0\ldots,0\right)\right], \\
    & \left[(df,0,f|_{U_1\cap U_2},0,\ldots,0)\right],
    \end{split}
\end{equation}
which we denote by $[\omega], [\eta],$ and $[\xi_{P,Q}]$ respectively, where in the last tuple $f \in \mathcal{O}(U_1)$ such that $f(P)-f(Q)=1$. For example we can take $f = \frac{x-x(Q)}{x(P)-x(Q)}$ if the points $P,Q$ have different $x$ coordinates. The cocycle $\left[\xi_{P,Q}\right]$ is cohomologous to $(0,0,0,-1,0,0,0)$ where the $-1$ is at the position corresponding to $\mathcal{O}(P \cap U_1)$. 
\end{rmk}

\begin{rmk}
\label{choice-of-basis-remark}
A few things should be noted about our choice of basis of the de Rham realization. The homomorphism $\phi_1 : H^1_{\rm{dR}}(\mathcal{E}, \{P,Q\} ) \rightarrow H^1_{\rm{dR}}(\mathcal{E})$ is given by projection of the above triple complex to the double complex with arrows given by differentiation and taking the difference. The isomorphism $ \phi_2: H_{dR}^1(\mathcal{E}) \rightarrow H_{dR}^1(\mathcal{E} \setminus [0:1:0])$ is induced by sending a tuple in the resulting double complex to its first entry. Now, the space of global 1-forms, i.e. differentials of the first kind, $H^0(\mathcal{E},\Omega^1_{\mathcal{E}/k})$, is 1-dimensional over $k$ and injects into $ H^1_{\rm{dR}}(\mathcal{E})$. For each differential of the first kind $\omega$ the lift of $[\omega] \in H^1_{\rm{dR}}(\mathcal{E})$ is not well defined in $H^1_{\rm{dR}}(\mathcal{E}, \{P,Q\})$ since we have that $\phi_1([\omega + \lambda \xi_{P,Q}]) = [\omega]$ for any $\lambda \in k$. Similarly we have to make a choice when
lifting $[\eta] \in H^1_{\rm{dR}}(\mathcal{E})$. These choices are made in \eqref{basis-triple-complex}. Furthermore, given a form of the second kind adding to it a $k$-multiple of a form of the first kind results in another form of the second kind. The choice of a form of the second kind corresponds to choosing a splitting of the Hodge filtration
$$
0 \subset F^1H^1_{\rm{dR}}(\mathcal{E}) \subset F^0H^1_{\rm{dR}}(\mathcal{E}) = H^1_{\rm{dR}}(\mathcal{E}) .
$$
There is no canonical splitting of the Hodge filtration for (families of) elliptic curves, and we will make a choice for each elliptic curve where we wish to work with a basis of de Rham cohomology explicitly.
\end{rmk}

\subsection{Betti realization}
A set of generators of the dual of the Betti realization $H_{B,{\varphi}}$, i.e. the relative Betti homology $H_1^{\Beta}(\mathcal{E},\{P,Q\};\QQ)$, is given by the classes 
$$
[\alpha], [\beta], \text{ and } [\gamma_{P,Q}],
$$
where $\alpha$ and $\beta$ are the two standard closed loops on the elliptic curve $\mathcal{E}$, and $\gamma_{P,Q}$ is a path from $P$ to $Q$ on the elliptic curve. Note that $[\gamma_{P,Q}]$ is a chosen lift of a generator of $\tilde{H}_0^{\Beta}(\{P,Q\};\QQ)$ in
$$
0 \rightarrow H_1^{\Beta}(\mathcal{E};\QQ) \rightarrow H_1^{\Beta}(\mathcal{E},\{P,Q\};\QQ) \xrightarrow{\partial} \tilde{H}_0^{\Beta}(\{P,Q\};\QQ) \rightarrow 0 .
$$
As such it is unique up to adding rational multiples of $\alpha$ and $\beta$.

\subsection{Periods}
Consider the function
\begin{equation*}
\begin{split}
{\rm{exp}}_{\mathcal{E}} : \mathbb{C} & \longrightarrow \mathcal{E} \subset \mathbb{P}^2(\mathbb{C}) \\
z & \longmapsto [\wp(z):\wp'(z):1],
\end{split}
\end{equation*}
 Note that ${\rm{exp}}_{\mathcal{E}}^*\left(\frac{dx}{2y}\right) = dz$ and ${\rm{exp}}_{\mathcal{E}}^*\left(\frac{xdx}{2y}\right) = -\wp(z)dz = d\zeta(z)$, where $\zeta(z)$ is the Weierstrass zeta function. The comparison isomorphism $$c_{\sigma} : H^1_{\dR}(\mathcal{E},\{P,Q\}) \cong H^1_{\Beta}(\mathcal{E},\{P,Q\};\QQ)$$ is given by integration, which in our case produces the period matrix
\begin{equation}
\label{period-matrix-1}
P_{\mathcal{E},P,Q} = \begin{blockarray}{cccc}
[\omega] & [\eta] & [\xi_{P,Q}] \\
\begin{block}{(ccc)c}
  \omega_1 & \eta_1 & 0 & [\alpha] \\
  \omega_2 & \eta_2 & 0 & [\beta] \\
  z_1-z_2 & \zeta(z_1) - \zeta(z_2) & 1 & [\gamma_{P,Q}], \\
\end{block}
\end{blockarray}
\end{equation}
where $\omega_1$ and $\omega_2$ are the numbers classically known as the periods of the elliptic curve $\mathcal{E}$, while $\eta_1$ and $\eta_2$ are usually called the quasi-periods of $\mathcal{E}$. The morphsim ${\rm{exp}}_{\mathcal{E}}$ restricts to a complex analytic isomorphism on $\mathbb{C}/\Lambda$, where $\Lambda$ is the lattice spanned the periods. We define $z_1,z_2$ to be the points in the preimage of $P$ and $Q$ respectively under this isomorphism.

\begin{subsection}{Elliptic motivic and de Rham periods}
We are interested in the following object
$$
H^1(\mathcal{E},\{P,Q\}) = ((H_{\Beta,{\varphi}})_{\varphi}, H_{\dR}, (c_{\varphi})_{\varphi}) \in \textrm{Ob}(\mathcal{H}_k) ,
$$
where $H_{\Beta,{\varphi}}$ and $H_{\dR}$ are defined in \eqref{de Rham and Betti elliptic two points}, and $\varphi: k \hookrightarrow \CC$ are embeddings of $k$ and $\mathcal{H}_k = \mathcal{H}(\Sp(k))$.

\begin{defi}
\label{elliptic-motivic-periods}
Let $[\omega],[\eta] \in H_{\dR}$ be the classes of a form of first kind and the second kind respectively, and let $\{[\alpha],[\beta]\}$ be a basis of the Betti homology $H^{\Beta}_1(\mathcal{E};\QQ)$. We define the following motivic periods:
\begin{equation}
\begin{split}
\label{motivic-periods-elliptic-first-kind}
  K_1^{\fm} = \left[H^1(\mathcal{E}), [\alpha] ,[\omega] \right]^{\fm}, \quad K_2^{\fm} = \left[H^1(\mathcal{E}), [\beta] ,[\omega] \right]^{\fm} \, ,
\end{split}
\end{equation}
as well as
\begin{equation}
\begin{split}
\label{motivic-periods-elliptic-second-kind}
 E_{1,\eta}^{\fm} = \left[H^1(\mathcal{E}), [\alpha] ,[\eta] \right]^{\fm}, \quad  E_{2,\eta}^{\fm} = \left[H^1(\mathcal{E}), [\beta] ,[\eta] \right]^{\fm}.
\end{split}
\end{equation}
Note that $K_1^{\fm}, K_2^{\fm}$ are associated to $\mathcal{E}$ canonically, up to a multiplicative constant in $k$, since this is true of the differential form of the first kind $\omega$. On the other hand, $E_{1,\eta}^{\fm}$ and $E_{2,\eta}^{\fm}$ depend on the choice of a form of the second kind $\eta$, as these are defined up to adding a multiple of the form of the first kind.

Furthermore, given two $k$-points $P,Q$, and a path $\gamma_{P,Q}$ between then, we define the following two motivic periods of $H^1(\mathcal{E},\{P,Q\})$:
\begin{equation}
\begin{split}
\label{motivic-periods-elliptic-incomplete}
F_{\gamma_{P,Q}}^{\fm} = \left[H^1(\mathcal{E},\{P,Q\}), [\gamma_{P,Q}] ,[\omega] \right]^{\fm}, \\
E_{\gamma_{P,Q},\eta}^{\fm} = \left[H^1(\mathcal{E},\{P,Q\}), [\gamma_{P,Q}] ,[\eta] \right]^{\fm}.
\end{split}
\end{equation}
We refer to $K_1^{\fm},K_{2}^{\fm}$ as \textit{motivic complete elliptic integrals of the first kind}, $E_{1,\eta}^{\fm},E_{2,\eta}^{\fm}$ as \textit{motivic complete elliptic integrals of the second kind}, and $F_{\gamma_{P,Q}}^{\fm}$ and $E_{\gamma_{P,Q},\eta}^{\fm}$ as \textit{motivic incomplete elliptic integrals of the first and second kind} respectively.
\end{defi}

\begin{defi}
\label{elliptic-dr-periods-defi}
For the basis of $H_{\dR}$ as in the previous definition we have the dual basis $\{[\omega]^{\vee}, [\eta]^{\vee}\}$. Define
\begin{equation}
\begin{split}
\label{dr-periods-elliptic-first-kind}
  & K_1^{\fdr} = \left[H^1(\mathcal{E}), [\omega]^{\vee} ,[\omega] \right]^{\fdr}, \quad K_{2,\eta}^{\fdr} = \left[H^1(\mathcal{E}), [\eta]^{\vee} ,[\omega] \right]^{\fdr} \\
  & E_{1,\eta}^{\fdr} = \left[H^1(\mathcal{E}), [\omega]^{\vee} ,[\eta] \right]^{\fdr}, \quad E_{2,\eta}^{\fdr} = \left[H^1(\mathcal{E}), [\eta]^{\vee} ,[\eta] \right]^{\fdr} .
\end{split}
\end{equation}
Note that $K_1^{\fdr}$ is associated to $\mathcal{E}$ canonically, up to a constant in $k$, but $K_{2,\eta}^{\fdr}, E_{1,\eta}^{\fdr},E_{2,\eta}^{\fdr}$ depend on the choice of a form of the second kind $\eta$.

Given two $k$-points $P,Q$ on $\mathcal{E}$, we can complete the basis of $H^1_{\rm{dR}}(\mathcal{E},\{P,Q\})$ by adding $[df]$ to $\{[\omega],[\eta]\}$, for $f$ a meromorphic function on $\mathcal{E}$ such that $f(P)-f(Q)=1$. We denote the dual basis by $\{[\omega]^{\vee},[\eta]^{\vee},[df]^{\vee}\}$. Then define
\begin{equation}
\begin{split}
\label{dr-periods-elliptic-incomplete}
F_{P,Q}^{\fdr} = \left[H^1(\mathcal{E},\{P,Q\}), [df]^{\vee} ,[\omega] \right]^{\fdr}, \\
E_{P,Q,\eta}^{\fdr} = \left[H^1(\mathcal{E},\{P,Q\}), [df]^{\vee} ,[\eta] \right]^{\fdr}.
\end{split}
\end{equation}
We refer to $K_1^{\fdr},K_{2,\eta}^{\fdr}$ as \textit{de Rham complete elliptic integrals of the first kind}, $E_{1,\eta}^{\fdr},E_{2,\eta}^{\fdr}$ as \textit{de Rham complete elliptic integrals of the second kind}, and $F_{P,Q}^{\fdr}$ and $E_{P,Q,\eta}^{\fdr}$ as \textit{de Rham incomplete elliptic integrals of the first and second kind} respectively. 
\end{defi}

\begin{rmk}
The above definitions are stated for an elliptic curve over $k$, i.e. $\mathcal{E} \rightarrow \Sp(k)$, and two $k$-points on it. They naturally generalize to families of elliptic curves $\pi: \mathcal{E} \rightarrow S$ over a more general base scheme $S$, along with a family of divisors $\{P,Q\}$. When we have such a dependence on parameters we consider the object $H^1(\mathcal{E})_{/S} = (H^1_{\Beta}(\mathcal{E})_{/S},H^1_{\dR}(\mathcal{E})_{/S},c)$ of $\mathcal{H}(S)$, where $H^1_{\Beta}(\mathcal{E})_{/S} = R^1\pi_*\mathbb{Q}$ is the local system, under appropriate constraints on $\pi$, whose fiber at $t = (t_1,\ldots,t_n)$ is the Betti cohomology of $\mathcal{E}_t$ with rational coefficients, and $H^1_{\dR}(\mathcal{E})_{/S}$ is the vector bundle $R^1\pi_*\Omega^{\bullet}_{X/S}$ whose fiber is the algebraic de Rham cohomology of $\mathcal{E}_t$. We define
\begin{equation}
\begin{split}
  K_1^{\fm}(t) = \left[H^1(\mathcal{E})_{/S}, [\alpha] ,[\omega] \right]^{\fm}, \\
  K_2^{\fm}(t) = \left[H^1(\mathcal{E})_{/S}, [\beta] ,[\omega] \right]^{\fm}.
  \end{split}
\end{equation}
  We make analogous definitions for the remaining motivic and de Rham periods defined in \ref{elliptic-motivic-periods} and \ref{elliptic-dr-periods-defi}. The elliptic curve, or the family of elliptic curves, we associate motivic and de Rham periods to will be suppressed from the notation when it is evident from context to avoid cluttering.
\end{rmk}

We can apply the general coaction formula to the previously defined motivic periods, e.g.
\begin{equation}
\begin{split}
& \Delta(K_1^{\fm}) = K_1^{\fm} \otimes K_1^{\fdr} + E_{1,\eta}^{\fm} \otimes K_{2,\eta}^{\fdr}. \\
& \Delta(F_{\gamma_{P,Q}}^{\fm}) = F_{\gamma_{P,Q}}^{\fm} \otimes K_1^{\fdr} + E_{\gamma_{P,Q},\eta}^{\fm} \otimes K_{2,\eta}^{\fdr} + 1 \otimes F_{P,Q}^{\fdr}.
\end{split}
\end{equation}
Alternatively, we can write the coaction in matrix form:
\begin{equation*}
\Delta \, 
\begin{pmatrix}
  K_1^{\fm} & E_{1,\eta}^{\fm} & 0 \\
  K_2^{\fm} & E_{2,\eta}^{\fm} & 0\\
  F_{\gamma_{P,Q}}^{\fm} & E_{\gamma_{P,Q},\eta}^{\fm} & 1 & \\
\end{pmatrix} = 
\begin{pmatrix}
  K_1^{\fm} & E_{1,\eta}^{\fm} & 0 \\
  K_2^{\fm} & E_{2,\eta}^{\fm} & 0\\
  F_{\gamma_{P,Q}}^{\fm} & E_{\gamma_{P,Q},\eta}^{\fm} & 1 & \\
\end{pmatrix} \otimes
\begin{pmatrix}
  K_1^{\fdr} & E_{1,\eta}^{\fdr} & 0 \\
  K_{2,\eta}^{\fdr} & E_{2,\eta}^{\fdr} & 0\\
  F_{P,Q}^{\fdr} & E_{P,Q,\eta}^{\fdr} & 1 & \\
\end{pmatrix}.
\end{equation*}

\end{subsection}
\begin{subsection}{Single-valued periods}
Let us consider the category $\mathcal{H}(\QQ)$ and extend its definition to include for each object $V = (V_{\Beta},V_{\dR},c)$ a linear involution $F_{\infty} : V_{\Beta} \xrightarrow{\sim} V_{\Beta}$ such that
$$
\xymatrix{
            V_{\dR} \otimes \CC \ar[d]^{c_{\dR}} \ar[r]^{c}& V_{\Beta} \otimes \CC \ar[d]^{F_{\infty}\otimes c_{\Beta}}\\
            V_{\dR} \otimes \CC \ar[r]^{c} & V_{\Beta} \otimes \CC
            },
$$
where $c_{\dR}$ and $c_{\Beta}$ are defined by $x \otimes \lambda \mapsto x \otimes \overline{\lambda}$, commutes.
We can then consider, in addition to the comparison isomorphism $c$, the isomorphism given by the composition
$$
c_{\textrm{s}} : V_{\dR} \otimes \CC \xrightarrow{c} V_{\Beta} \otimes \CC \xrightarrow{F_{\infty} \otimes id} V_{\Beta} \otimes \CC \xrightarrow{c^{-1}} V_{\dR} \otimes \CC .
$$
We comparison $c$ yields the period homomorphism, and similarly $c_{\textrm{s}}$ yields the \textit{single-valued period homomorphism}:
$\mathrm{s} : \mathcal{P}^{\fdr}_{\mathcal{H}(S)} \rightarrow \mathbb{R}$ sending $\left[V,f,\omega\right]$ to $f(c_{\textrm{s}}(\omega))$. The single-valued period matrix, i.e. the matrix corresponding to the isomorphism $c_{\textrm{s}}$, is given by multiplying the period matrix with its conjugate inverse on the left. The single-valued period homomorphism for families of de Rham periods over a base scheme $S$ is defined in appendix \ref{single-valued-periods definition}. It is a homomorphism 
$$
\mathrm{s} : \mathcal{P}^{\fdr}_{\mathcal{H}(S)} \rightarrow M(S(\CC)) \otimes \overline{M}(S(\CC)), 
$$ 
for $M_{X,Y}(S(\CC))$ the ring of multivalued meromorphic functions on $S(\CC)$ and $\overline{M}(S(\CC))$ the ring of quotients of antiholomorphic functions on $S(\CC)$.

Let $\mathcal{E} \rightarrow \mathcal{M}_{1,1}$ be the universal elliptic curve, and $P_{\mathcal{E}}$ be the period matrix of $H^1(\mathcal{E})$ with respect to the basis $\{[\omega],[\eta]\}$ of $H^1_{\dR}(\mathcal{E})$ and $\{[\alpha],[\beta]\}$ of $H^{\Beta}_1(\mathcal{E};\QQ)$ defined in \ref{standard elliptic de rham realization}. Let $\tau = \frac{\omega_2}{\omega_1}$, $\lambda = \frac{\omega_1}{2\pi i}$, and let $
\mathbb{G}^*_2(\tau) = \mathbb{G}_2(\tau) + \frac{1}{8\pi\Imm{\tau}}
$, where 
$$
\mathbb{G}_2(\tau) = -\frac{1}{24}+\sum_{n \geq 1}\sigma_1(n)q^n.
$$
The function $\mathbb{G}^*_2(\tau)$ is the modified real analytic version of the Eisenstein series of weight 2, which transforms like a modular form of weight 2. Here $\sigma_1$ is the divisor sum function. As pointed out in \cite[\S 6.2]{BD} we obtain the single-valued period matrix of $H^1(\mathcal{E})$
$$
 \overline{P_{\mathcal{E}}}^{-1}P_{\mathcal{E}} = 
\begin{pmatrix}
f_{1,1}(\tau) & f_{1,2}(\tau)
 \\
f_{2,1}(\tau) & f_{2,2}(\tau)
\end{pmatrix},
$$
where the entries of the matrix are given by
\begin{equation}
\label{f_is}
\begin{split}
&f_{1,1}(\tau) = -\lambda\overline{\lambda}^{-1}8\pi\Imm(\tau)\overline{\mathbb{G}^*_2(\tau)} \\ &f_{1,2}(\tau) = \lambda^{-1}\overline{\lambda}^{-1}(4\pi\Imm{\tau})^{-1}((-8\pi\Imm(\tau))^2\mathbb{G}^*_2(\tau)\overline{\mathbb{G}^*_2(\tau)}-1) \\
& f_{2,1}(\tau) =-\lambda\overline{\lambda}4\pi\Imm{\tau}, \quad f_{2,2}(\tau) = -\lambda^{-1}\overline{\lambda}-8\pi\Imm(\tau)\mathbb{G}^*_2(\tau).
\end{split}
\end{equation}
This follows from the relations shown by Fricke and Legendre respectively
$$
\mathbb{G}_2(\tau) = -\frac{1}{2}\frac{\omega_1\eta_1}{(2\pi i )^2} \quad \text{ and } \quad \omega_1\eta_2 - \eta_1\omega_2 = 2 \pi i .
$$
By the definition of the single-valued map we have
\begin{equation}
\begin{split}
&\mathrm{s}(K_1^{\fdr}(\tau))=f_{1,1}(\tau), \quad \mathrm{s}(K_{2,\eta}^{\fdr}(\tau))= f_{1,2}(\tau)\\
& \mathrm{s}(E_{1,\eta}^{\fdr}(\tau)) = f_{2,1}(\tau) , \quad\mathrm{s}(E_{2,\eta}^{\fdr}(\tau)) = f_{2,2}(\tau) .
\end{split}
\end{equation}
Let and $r,s$ be the powers of the prefactors $\lambda$ and $\overline{\lambda}$ in \eqref{f_is} respectively. Then these functions transform as
$$
f(\gamma \tau) = (c\tau + d)^{-r}(c\overline{\tau} + d)^{-s}f(\tau),
$$
for each $\gamma = \begin{pmatrix}
a & b \\
c & d
\end{pmatrix} \in \textrm{SL}_2(\mathbb{Z})$.
\begin{rmk}
\label{two-weights-fund-group-remark}
The above single-valued functions are examples of Brown's real analytic modular forms with two weights which are generally defined by replacing the cohomology of $\mathcal{E}$ above with its unipotent fundamental group  --- see \cite{Brown4} and sequels. These are functions on the upper half plane which satisfy a modular transformation property such as the one in the previous equation, for arbitrary $(r,s)$, and which have an expansion of the form
$$
f(\tau) = \sum_{k=-N}^{N} \sum_{m,n \geq 0}a_{m,n}^{(k)}(-2\pi\Imm{\tau})^kq_{\tau}^m\overline{q_{\tau}}^n,
$$
where $q_{\tau} = e^{2\pi i \tau}$.  They are generalized by the single-valued functions we define next.
\end{rmk}

Now let us consider the object $H^1(\mathcal{E},\{P,Q\})$, and its period matrix $P_{\mathcal{E},P,Q}$. We have two de Rham periods: $F^{\fdr}_{P,Q}(\tau)$ and $E^{\fdr}_{P,Q,\eta}(\tau)$, which do not appear as de Rham periods of $H^1(\mathcal{E})$. We summarise the properties of their single-valued periods.
\begin{prop}
\label{s-v-jacobi-forms}
The single-valued periods associated to $F^{\fdr}_{P,Q}(\tau)$ and $E^{\fdr}_{P,Q,\eta}(\tau)$ are
\begin{equation}
\begin{split}
    & \sv\left(F^{\fdr}_{P,Q}(\tau)\right)=\overline{(\zeta(z_1)-\zeta(z_2))}f_{2,1}(\tau) + (\overline{z_1}-\overline{z_2})f_{1,1}(\tau) + 2\pi i(z_1 - z_2) \\
    & \sv\left(E^{\fdr}_{P,Q,\eta}(\tau)\right)=\overline{(\zeta(z_1)-\zeta(z_2))}f_{2,2}(\tau) + (\overline{z_1}-\overline{z_2})f_{1,2}(\tau) + 2\pi i(\zeta(z_1) - \zeta(z_2)) .
\end{split}
\end{equation}
Denote the two functions by $f_{3,1}(\tau,z_1,z_2)$ and $f_{3,2}(\tau,z_1,z_2)$ respectively. Let $\gamma = \begin{pmatrix}
a & b \\
c & d
\end{pmatrix} \in \textrm{SL}_2(\mathbb{Z})$. Then these functions transform as
\begin{equation}
\begin{split}
    & f_{3,1}\left(\frac{a\tau + b}{c\tau + d},\frac{z_1}{c\tau + d},\frac{z_2}{c\tau + d}\right) = (c\tau + d)^{-1}f_{3,1}(\tau,z_1,z_2) \\
    & f_{3,2}\left(\frac{a\tau + b}{c\tau + d},\frac{z_1}{c\tau + d},\frac{z_2}{c\tau + d}\right) = (c\tau + d)f_{3,2}(\tau,z_1,z_2).
\end{split}
\end{equation}
They have expansions of the form
$$
f(\tau,z_1,z_2) = \sum_{-N \leq k_1,k_2,k_3 \leq N} \sum_{\substack{m_i,n_i \geq 0 \\ 1 \leq i \leq 3}}a_{m,n}^{(k_1,k_2,k_3)}z_1^{k_1}z_2^{k_2}(-2\pi\Imm{\tau})^{k_3}q_{i}^{m_i}\overline{q_i}^{n_i},
$$
where $q_1 = e^{2\pi i \tau},q_2 = e^{2\pi i z_1}, q_3 = e^{2\pi i z_2}$, and $z_1,z_2$ are in a compact set not containing any points of the lattice generated by $1$ and $\tau$.
\end{prop}
\begin{proof}
This follows from \eqref{single-valued-matrix definition} by computing the matrix $(\overline{P_{\mathcal{E},P,Q}})^{-1}P_{\mathcal{E},P,Q}$, where $P_{\mathcal{E},P,Q}$ is defined in \eqref{period-matrix-1}, and noting that the Weierstrass zeta function transforms as 
$$
\zeta\left(\frac{a\tau + b}{c\tau + d},\frac{z}{c\tau + d}\right) = (c\tau + d)\zeta(\tau,z).
$$
We have assumed so far that the Weierstrass zeta function is defined with respect to the lattice generated by $(\tau,1)$, and here we have written the dependence explicitly. Moreover, $z_1,z_2$ are also defined with respect to the lattice generated by $(\tau,1)$, and correspond to $\frac{z_1}{c\tau+d},\frac{z_2}{c\tau+d}$ with respect to the generators of the lattice given by $\gamma(\tau,1)^T$. The expansion follows from the previous computation for $f_{1,1},f_{1,2},f_{2,1}$ and $f_{2,2}$ as well as the well known expansion of the Weierestrass zeta function (see \cite[Ch. XVIII \S 3]{LangEF})
$$
\zeta(\tau,z) = \eta_2z+\pi i \frac{q_z+1}{q_z-1}+2\pi i\sum_{n \geq 1}\frac{q_{\tau}^n/q_z}{1-q_{\tau}q_z} - \frac{q_{\tau}^nq_z}{1-q_{\tau}^nq_z},
$$
where $q_z = e^{2\pi i z}$, where we then expand the terms $\frac{-1}{1 - q_z}, \frac{1}{1-q_{\tau}q_z}$, and $\frac{1}{1-q_{\tau}^nq_z}$.
\end{proof}
\end{subsection}

\section{Subdivision of edges}
\label{section: subdivison of edges}

\begin{defi}
Given a Feynman graph $G$ and an internal edge $e \in E_G$ connected to vertices $v_{e,1},v_{e,2} \in V_G$, we denote by $G_{s(e)}$ the graph obtained by adding a vertex $v$ and two edges $e_1,e_2$ to $G \setminus e$ such that $e_1$ is connected to $v_{e,1}$ and $v$, and $e_2$ is connected to $v$ and $v_{e,2}$. If $e$ has a non-zero mass $m_e$ in $G$ then both edges $e_1$ and $e_2$ have mass $m_e$ in $G_{s(e)}$, otherwise $e_1$ and $e_2$ have zero mass. We say that $G_{s(e)}$ is obtained by subdividing $e$ in $G$. Define $G_{s(I)}$, where $I=\{e_1^{k_1},\ldots,e_{N_{G}}^{k_{N_{G}}}\}$ for each $k_i \geq 0$ an integer, to be the graph obtained by subdividing $k_i$ times the edge $e_i$ for $ 1 \leq i \leq N_{G}$.
\end{defi}

We will order the edges of $G$ so that the parameter $\alpha_{N_G}$ corresponds to the edge $e$ in $E_G$, and that $\alpha_{N_G},\alpha_{N_G+1}$ correspond to the edges $e_1,e_2$ in $E_{G_{s(e)}}$ obtained by subdividing $e$. For a given graph $G$ we will consider the following morphism
\begin{equation}
\label{rho definiton}
    \begin{split}
    \rho : \,\, &\mathbb{A}^1 \times \mathbb{P}^{N_G-1} \longrightarrow \mathbb{P}^{N_G}\\
    & (y,\alpha_1,\ldots,\alpha_{N_G}) \mapsto (\alpha_1,\ldots,y\alpha_{N_G},(1-y)\alpha_{N_G}).
    \end{split}
\end{equation}

\begin{lemma}
\label{pullback of dotted no blowup}
Let $X_{\Xi_G}=V(\Xi_G)$ and $X_{\Psi_G}=V(\Psi_G)$ be the graph hypersurfaces associated to $G$. Then 
$$
\rho^*(X_{\Xi_{G_{s(e)}}}) = X_{\Xi_G}\quad \text{and } \quad \rho^*(X_{\Psi_{G_{s(e)}}}) = X_{\Psi_G}.
$$
Moreover, we have
$$
\rho^*\left(\omega_{G_{s(e)}}\right) = - \alpha_{N_G}\psi_G^{z_1}\Xi_G^{z_2}\omega_G\wedge dy,
$$
where $z_1 = 1-(h_G + 1)d/2$ and $ z_2 = dh_G/2-1$, for $d=d_{G_{s(e)}}-d_G$ the difference between the space-time dimension for $G_{s(e)}$ and $G$.
\end{lemma}
\begin{proof}
The first statement follows simply from the fact that we obtain the polynomials $\Xi_{G_{s(e)}}$ and $\Psi_{G_{s(e)}}$ from the polynomials $\Xi_{G}$ and $\Psi_{G}$ by replacing $\alpha_{N_G}$ by $\alpha_{N_G}+\alpha_{N_G+1}$ in the latter. This follows from their definitions \eqref{first-symanzik} and \eqref{second-Symanzik} in terms of spanning trees of $G$ and $G_{s(e)}$. For each spanning tree $T$ of $G$ which includes the edge $e$ we obtain a spanning tree of $G_{s(e)}$ spanned by $E_T \cup \{e_1,e_2\}$, where $E_T$ are the edges of $T$. For each spanning tree $T$ of $G$ which didn't include $e$ we get two spanning trees of $G_{s(e)}$, one spanned by $E_T \cup e_1$ and the other $E_T \cup e_2$, as we must attach the new vertex to the rest of the spanning tree. This accounts for all spanning trees of $G_{s(e)}$. Each spanning 2-tree, which are used to define $\Xi_{G_{s(e)}}$, contains either one or both of the new edges $e_1,e_2$ and is obtained from spanning 2-trees of $G$ analogously to the previous case of spanning trees. There is no term in $\Xi_{G_{s(e)}}$ corresponding to a spanning 2-tree that contains only the new vertex in $G_{s(e)}$ as one of its connected components as there is no incoming momentum to the new vertex.

We will now show that $\rho^*(\Omega_{G_{s(e)}}) = -\alpha_{N_G}\Omega_G \wedge dy$, where $\Omega_G$ is defined in \eqref{omega-def}. The extra factor $\psi_G^{z_1}\Xi_G^{z_2}$ comes from $G_{s(e)}$ having one more edge than $G$ and the same number of loops. This, combined with the first statement and the definition of the Feynman integrand $\omega_G$ \eqref{Feynman integrand defi} implies the second statement of the lemma.

For $i\not \in \{N_G,N_G+1\}$ we have
\begin{equation}
\begin{split}
&\rho^*((-1)^i\alpha_id\alpha_1 \wedge \ldots \wedge \widehat{d\alpha_i}\wedge \ldots \wedge d\alpha_{N_G}\wedge d\alpha_{N_G+1}) \\
& = (-1)^i\alpha_id\alpha_1 \wedge \ldots \wedge \widehat{d\alpha_i}\wedge \ldots \wedge d(y\alpha_{N_G})\wedge d((1-y)\alpha_{N_G}) \\
& = (-1)^{i+1}\alpha_{N_G}\alpha_id\alpha_1 \wedge \ldots \wedge \widehat{d\alpha_i}\wedge \ldots \wedge d\alpha_{N_G}\wedge dy.
\end{split}
\end{equation}
For $i = N_G$ we have 
\begin{equation}
\begin{split}
&\rho^*((-1)^{i}\alpha_{i}d\alpha_1 \wedge \ldots \wedge \widehat{d\alpha_{i}} \wedge d\alpha_{i+1}) = (-1)^{i}y\alpha_{i}d\alpha_1 \wedge \ldots \wedge d((1-y)\alpha_{i}) \\
& = (-1)^{i+1}y\alpha_{i}^2d\alpha_1 \wedge \ldots \wedge d\alpha_{i-1}\wedge dy + (-1)^{i}y\alpha_{i}d\alpha_1 \wedge \ldots \wedge d\alpha_{i} + \\
&+ (-1)^{i+1}y^2\alpha_{i}d\alpha_1 \wedge \ldots \wedge d\alpha_{i}.
\end{split}
\end{equation}
Finally for $i=N_G+1$ we have
\begin{equation}
\begin{split}
&\rho^*((-1)^{i+1}\alpha_{i+1}d\alpha_1 \wedge \ldots \wedge d\alpha_{i}) = (-1)^{i+1}(1-y)\alpha_{i}d\alpha_1 \wedge \ldots \wedge d(y\alpha_{i}) \\
& = (-1)^{i+1}y\alpha_{i}d\alpha_1 \wedge \ldots d\alpha_{i} + (-1)^{i+1}\alpha_{i}^2d\alpha_1 \wedge \ldots \wedge d\alpha_{i-1}\wedge dy + \\
&+(-1)^{i+2}y^2\alpha_{i}d\alpha_1 \wedge \ldots \wedge d\alpha_{i} +(-1)^{i+2}y\alpha_{i}^2d\alpha_1 \wedge \ldots \wedge d\alpha_{i-1} \wedge dy.
\end{split}
\end{equation}
Adding the last two equations together we obtain the claim.
\end{proof}

Let $P^G \rightarrow \mathbb{P}^{N_G-1}$ and $P^{G_{s(e)}}\rightarrow \mathbb{P}^{N_G}$ be the blow-ups associated to the graphs $G$ and $G_{s(e)}$. These blow-ups are indexed by motic subgraphs, see \cite[\S 5]{Brown2} and \cite[Definition 6.3]{Brown2} for the definition in full generality. We will show that the morphism $\rho$ lifts to a morphism of the blow-ups. We will need the following lemma.
\begin{lemma}
For each motic subgraph of $G$ containing the edge $e$ there is exactly one motic subgraph of $G_{s(e)}$ containing both $e_1$ and $e_2$. For each motic subgraph of $G$ which does not contain $e$ there is exactly one motic subgraph of $G_{s(e)}$ containing neither $e_1$ nor $e_2$.
\end{lemma}
\begin{proof}
A subgraph $\gamma$ of $G$ can be motic because it is mass-momentum spanning and each of its subgraphs is not mass-momentum spanning. If such $\gamma$ contains $e$ then the corresponding subgraph in $G_{s(e)}$ must contain both $e_1,e_2$ because they both have the same mass as $e$ in $G$. Alternatively, a subgraph $\gamma$ of $G$ can be motic because all of its subgraphs have a lower loop order than itself. If such $\gamma$ contains $e$ in $G$ then the corresponding subgraph in $G_{s(e)}$ must contain both $e_1,e_2$, otherwise it would have a lower loop order than $\gamma$. Finally we cannot have an additional motic subgraph $G_{s(e)}$ containing only one of $e_1,e_2$ because adding one of them to a subgraph cannot make it mass-momentum spanning (since the other has the same mass), or increase the loop order of a subgraph.
\end{proof}

\begin{prop}
\label{subdivide morphism}
There exists a morphism $\tilde{\rho}$ such that the following diagram commutes
$$\xymatrix{
            \mathbb{A}^1 \times P^G \ar[d]^{\textrm{id} \times \pi_G} \ar[r]^{\tilde{\rho}}& P^{G_{s(e)}} \ar[d]^{\pi_{G_{s(e)}}}\\
             \mathbb{A}^1 \times \mathbb{P}^{N_G-1} \ar[r]^{\rho} & \mathbb{P}^{N_G}
            } 
$$
The morphism $\tilde{\rho}$ induces a morphism of graph motives
$$
mot_{G_{s(e)}} \xrightarrow{\tilde{\rho}^*} H^{N_G-1}(\mathbb{A}^1 \times P^G, \{0,1\} \times D_{G} \setminus Y_{G} \cap D_{G}) \cong mot_G \otimes \mathbb{Q}(0) .
$$
\end{prop}
\begin{proof}

In order to show that the morphism $\tilde{\rho}$ is well defined we will consider compatible coverings of the spaces $P^G$ and $P^{G_{s(e)}}$. For the definition of the covering of $P^G$ for any graph $G$ we follow \cite[\S 5.3]{Brown2}. Let $S_G = \{1,\ldots ,N_G \}$. Denote by $B_G$ the set of subsets of $S_G$ where each element $I \in B_G$ corresponds to a linear subspace $L_I = \cap_{i \in I}V(\alpha_i) \subset \mathbb{P}^{N_G-1}$ which we blow up to define $P^G$. Note that each $I$ corresponds to a motic subgraph of $G$.

For each flag 
$$
\mathcal{F} : \emptyset = I_0 \subsetneq I_1 \ldots I_k \subsetneq I_{k+1} = S_G
$$
and a choice of elements
$$
c : j_n \in I_n \setminus I_{n-1} \, \text{, for} \, 1 \leq n \leq k+1
$$
which cannot be made larger (i.e., there is no $I$ and $1\leq i \leq k+1$ such that $I_i \subsetneq I \subsetneq I_{i+1}$ and $j_{i+1} \not \in I$)  
we define an affine open $\mathbb{A}^{\mathcal{F},c} = \Sp(\mathbb{Z}[\beta_1,\ldots,\widehat{\beta_{j_{k+1}}},\ldots,\beta_{N_G}])$, and a projection morphism $\pi_{\mathcal{F},c} : \mathbb{A}^{\mathcal{F},c} \rightarrow \Sp(\mathbb{Z}[\alpha_1,\ldots,\widehat{\alpha_{j_{k+1}}},\ldots,\alpha_{N_G}])$. Note that the target of the morphism is the open in $\mathbb{P}^{N_G-1}$ where $\alpha_{j_{k+1}} \not = 0$. It is defined by
\begin{equation}
\label{projection}
\pi^*_{\mathcal{F},c}(\alpha_i) = \begin{cases}
		    \beta_i\beta_{j_n}\ldots\beta_{j_{k+1}}, & \text{for } i \in I_n \setminus I_{n-1} \cup \{j_n\} \text { for some } 1 \leq n \leq k+1\\
            \beta_{j_n}\ldots\beta_{j_{k+1}}, & \text{for } i = j_n \text { for some } 1 \leq n \leq k+1
		 \end{cases} .
\end{equation}
Its inverse is defined on the subset of $\mathbb{P}^{N_G-1}$ where $\alpha_{j_n} \not = 0$ for $1 \leq n \leq k+1$ by 
\begin{equation}
\label{inverse projection}
(\pi_{\mathcal{F},c}^{-1})^*(\beta_i) = \begin{cases}
		    \frac{\alpha_i}{\alpha_{j_n}}, &  \text{for } i \in I_n \setminus I_{n-1} \cup \{j_n\} \text { for some } 1 \leq n \leq k+1\\
            \frac{\alpha_{j_n}}{\alpha_{j_{n+1}}}, & \text{for } i = j_n \text { for some } 1 \leq n \leq k \\
            \alpha_{j_{k+1}}, & i=j_{k+1} \\
		 \end{cases} .
\end{equation}
The affine opens $\mathbb{A}^{\mathcal{F},c}$ glue together to form $P^G$ (see \cite[\S 5.3]{Brown2}).

Recall $N_G,N_G+1$ are the indices corresponding to the edges obtained by subdividing of $e$ in $G$ according to our chosen ordering. It follows from the previous lemma each element of $B_{G_{s(e)}}$ contains either both $N_G$ and $N_G+1$ or neither. Therefore for each flag $\mathcal{F}$ of elements of $B_G$ there is a unique flag 
$$
\mathcal{F'} : \emptyset = I_0' \subsetneq I_1' \ldots I_k' \subsetneq I_{k+1}' = S_{G_{s(e)}}
$$
where $I_n \in B_{G_{s(e)}}$ are defined by
$$
I'_n = \begin{cases}
    I_n \cup \{N_G+1\}, & \text{if } N_G \in I_n \\
    I_n, & \text{otherwise}
\end{cases} .
$$
If a choice $c$ of elements in $B_G$ associated to $\mathcal{F}$ does not contain $j_n = N_G$ for any $1 \leq n \leq k+1$ then we can consider the affine open in $P^{G_{s(e)}}$ defined by the pair $\mathcal{F}',c$. In this case, it follows from \eqref{projection}, \eqref{inverse projection} and the defintion of the morphism $\rho$ \eqref{rho definiton} that
$$
\left(\pi_{\mathcal{F},c}^{-1} \circ \rho|_{\mathbb{A}^{\mathcal{F},c}} \circ \pi_{\mathcal{F}',c} \right)^*(\beta_i) = \begin{cases}
\beta_i, & \text{ if } i \not \in \{ N_G, N_G+1 \} \\
y\beta_{N_G}, & \text{ if } i = N_G \\
(1-y)\beta_{N_G}, & \text{ if } i = N_G+1
\end{cases} .
$$
We can see that this is well defined on the whole open $\mathbb{A^1} \times \mathbb{A}^{\mathcal{F},c}$.

If a choice $c$ of elements in $B_G$ associated to $\mathcal{F}$ contains $j_r = N_G$ for some $1 \leq r \leq k+1$ then we can consider two affine opens in $P^{G_{s(e)}}$: one defined by the pair $\mathcal{F}',c$ and another $\mathcal{F}',c'$ where $c'$ is the same as $c$ except for $j_r=N_G+1$. We have
$$
\left(\pi_{\mathcal{F},c}^{-1} \circ \rho|_{\mathbb{A}^{\mathcal{F},c}} \circ \pi_{\mathcal{F}',c} \right)^*(\beta_i) = \begin{cases}
\beta_i, & \text{ if } i \in I_n \setminus I_{n-1} \text{ and } n \not = r \\
\frac{\beta_i}{y}, & \text{ if } i \in I_r \setminus I_{r-1} \text{ and } i \not = N_G+1 \\
\frac{1-y}{y}, & \text{ if } i = N_G+1 \\
y\beta_{N_G}, & \text{ if } i = N_G
\end{cases} ,
$$
and we can see that this is well defined on $\{y\not = 0\} \times \mathbb{A}^{\mathcal{F},c}$. Finally for $\mathbb{A}^{\mathcal{F}',c'}$ we have that 
$\left(\pi_{\mathcal{F},c}^{-1} \circ \rho|_{\mathbb{A}^{\mathcal{F},c}} \circ \pi_{\mathcal{F}',c'} \right)^*$ acts in the same way as the previous case where we swap $N_G$ and $N_G+1$, as well as $y$ and $(1-y)$. It is defined on $\{y\not = 1\} \times \mathbb{A}^{\mathcal{F},c}$. Along with the previous case we have covered $\mathbb{A}^1 \times \mathbb{A}^{\mathcal{F},c}$. Therefore we have shown that $\rho$ extends to the blow-up, and we denote the extension $\tilde{\rho}$.

To show that it induces a morphism of motives, the second claim of the proposition, we have to check that 
$$\tilde{\rho}\left( \{0,1\} \times D_{G} \setminus Y_{G} \cap D_{G}\right) \subset D_{G_{s(e)}} \setminus Y_{G_{s(e)}} \cap D_{G_{s(e)}}.
$$ 
By the commutativity of the diagram we have just shown it is enough to check $$\rho\left( \{0,1\} \times \Delta_{G} \setminus X_{G} \cap \Delta_{G}\right) \subset \Delta_{G_{s(e)}} \setminus X_{G_{s(e)}} \cap \Delta_{G_{s(e)}}.
$$
In fact it is clear from the definition of $\rho$ that $\rho\left( \{0,1\} \times \Delta_{G} \right) \subset \Delta_{G_{s(e)}}$ and we have already shown that it send graph hypersurfaces of $G$ to graph hypersurfaces of $G_{s(e)}$.
\end{proof}

\begin{cor}
\label{subdivided motivic period}
    Let $I=\{e_1^{k_1},\ldots,e_{N_{G}}^{k_{N_{G}}}\}$ for each $k_i \geq 0$ an integer. Let $I_{G_{s(I)}}^{\fm}(m,q)$ be its motivic Feynman amplitude in $d_{G_{s(I)}}$ dimensions, $\omega_G(m,q)$ be the Feynman integrand of $G$ in $d_G$ dimensions, $K = \sum_{i=1}^{N_G}k_i$, and $d = d_{G_{s(I)}}-d_G$. Then 
\begin{equation*}
I_{G_{s(I)}}^{\fm}(m,q) =  \left[mot_G, \left[\sigma_G\right], \left[\pi_{G}^*\left((-1)^K\alpha_{1}^{k_1}\cdots\alpha_{N_G}^{k_{N_G}}\Psi_G^{z_1}\Xi_G^{z_2}\omega_G(m,q))\right)\right] \right]^{\fm},
\end{equation*}
 where $z_1 = K - (h_G+1)d/2$, and $z_2 = dh_G/2-K$.
\end{cor}
\begin{proof}
It is enough to show the claim for $I = \{e_{N_G}\}$, the edge corresponding to the parameter $N_G$. The result then follows by induction. By the commutativity of the diagram in the previous proposition and lemma \ref{pullback of dotted no blowup} we have 
\begin{equation*}
\begin{split}
\tilde{\rho}^*\left(\left[\pi_{G_{s(e)}}^*\left(\omega_{G_{s(e)}}\right)\right]\right) & = \left[\pi_G^*\left( \frac{\alpha_{N_G}}{\Psi_G^{d_{G_{s(e)}}}}\left(\frac{\Psi_G}{\Xi_G}\right)^{N_G+1-h_Gd_{G_{s(e)}}/2}\Omega_G\right)\wedge (-dy)\right] \\
& = \left[\pi_G^*\left(\alpha_{N_G}\psi_G^{1-(h_G + 1)d/2}\Xi_G^{dh_G/2-1}\omega_G\right)\wedge (-dy)\right].
\end{split}
\end{equation*}
Given a morphism $\phi : M_1 \rightarrow M_2$ in $\mathcal{H}(S)$ we get equivalences of motivic periods
$$
[M_1,\sigma_1,\phi^*(\omega_2)] = [M_2, \phi_*(\sigma_1),\omega_2].
$$
Let
$$
I_G^{\fm}(\omega) = \left[mot_G, \left[\sigma_G\right], \left[\pi_G^*\left(\omega\right)\right] \right]^{\fm}.
$$
It follows from $\tilde{\rho}_*\left(\left[\gamma_{0,1}\times \sigma_{G_{s(e)}}\right]\right) = \left[\sigma_G\right]$, where $\gamma_{0,1}$ is a path from 0 to 1 in $\mathbb{A}^1$, that we have
\begin{equation*}
I_{G_{s(e)}}^{\fm} = I_G^{\fm}\left(\alpha_{N_G}\psi_G^{1-(h_G + 1)d/2}\Xi_G^{dh_G/2-1}\omega_G\right)  \otimes \left[H^1(\mathbb{A}^1,\{0,1\}),[\gamma_{0,1}],[-dy])\right]^{\fm}.
\end{equation*}
We may apply the period homomorphism to the right hand side 
$$
\textrm{per}\left(\left[H^1(\mathbb{A}^1,\{0,1\},[\gamma_{0,1}],[-dy])\right]^{\fm}\right) = -1.
$$
The conclusion follows.
\end{proof}

\begin{section}{The sunrise}
We denote the sunrise graph (figure \ref{sunrise_graph}) by $G$. Its first and second Symanzik polynomials, respectively, are
\begin{equation}
\begin{split}
\Psi_G &= \alpha_1\alpha_2 + \alpha_2\alpha_3 + \alpha_1\alpha_3, \text{ and} \\
\Xi_G(m,q) &= q_1^2\alpha_1\alpha_2\alpha_3 + (m_1^2\alpha_1+m_2^2\alpha_2+m_3^2\alpha_3)\Psi_G .
\end{split}
\end{equation}
We note that the zero locus of the second Symanzik polynomial, along with a choice of a rational point, defines an elliptic curve $\mathcal{E}$ over a field $k$ for fixed values of $m_1,m_2,m_3,q_1^2 \in k \subset \mathbb{C}$. We will allow the masses and momenta to vary, and work with the family of elliptic curves $\mathcal{E} \rightarrow S$, where $S$ is an open subset of the space of generic kinematics (see \cite[\S 1.7]{Brown2}). In the equal-mass case this family of elliptic curves will be the universal family for the modular curve $X_1(6)$ (see \cite[\S 4]{BV}).

\begin{subsection}{The graph motive} 

We define the graph motive of the sunrise graph following Brown in \cite[\S 6]{Brown2}\footnote{We will actually work with what Brown calls a 'restricted' motive of a graph, due to the fact that in $d=2$ only the second Symanzik polynomial appears in the denominator of the integrand.}. All the following schemes are viewed over the base $S$. Note that we have the points $[1:0:0],[0:1:0],[0:0:1]$ on $\mathcal{E}$ for any value of $m_1,m_2,m_3,q_1^2 \in S(\CC)$. We blow-up $\mathbb{P}^2$ at these three points\footnote{These points are correspond to \textit{motic subgraphs}, defined in \cite[\S 3]{Brown2}, of $G$. They are the three subgraphs spanned by three distinct pairs of edges}, and denote the blow-up by $\pi_G: P^G \rightarrow \mathbb{P}^2$. Let, by abuse of notation, $\mathcal{E}$ denote the strict transform of the vanishing locus of the second Symanzik polynomial $\Xi_G$, and denote the total transform of $\Delta_G = \bigcup_{e \in E_G} V(\alpha_e)$ in $\mathbb{P}^2$ by $D$. The divisor $D$ has 6 irreducible components, and we denote the exceptional divisors by $D_{-1},D_{-2},D_{-3}$, their union by $D_E$, and the strict transforms of the faces $V(\alpha_e)$ of the coordinate simplex in $\mathbb{P}^2$ by $D_1,D_2,D_3$. Geometrically we have the following picture.

\begin{figure}[h]
\centering
    \includegraphics[width=10cm]{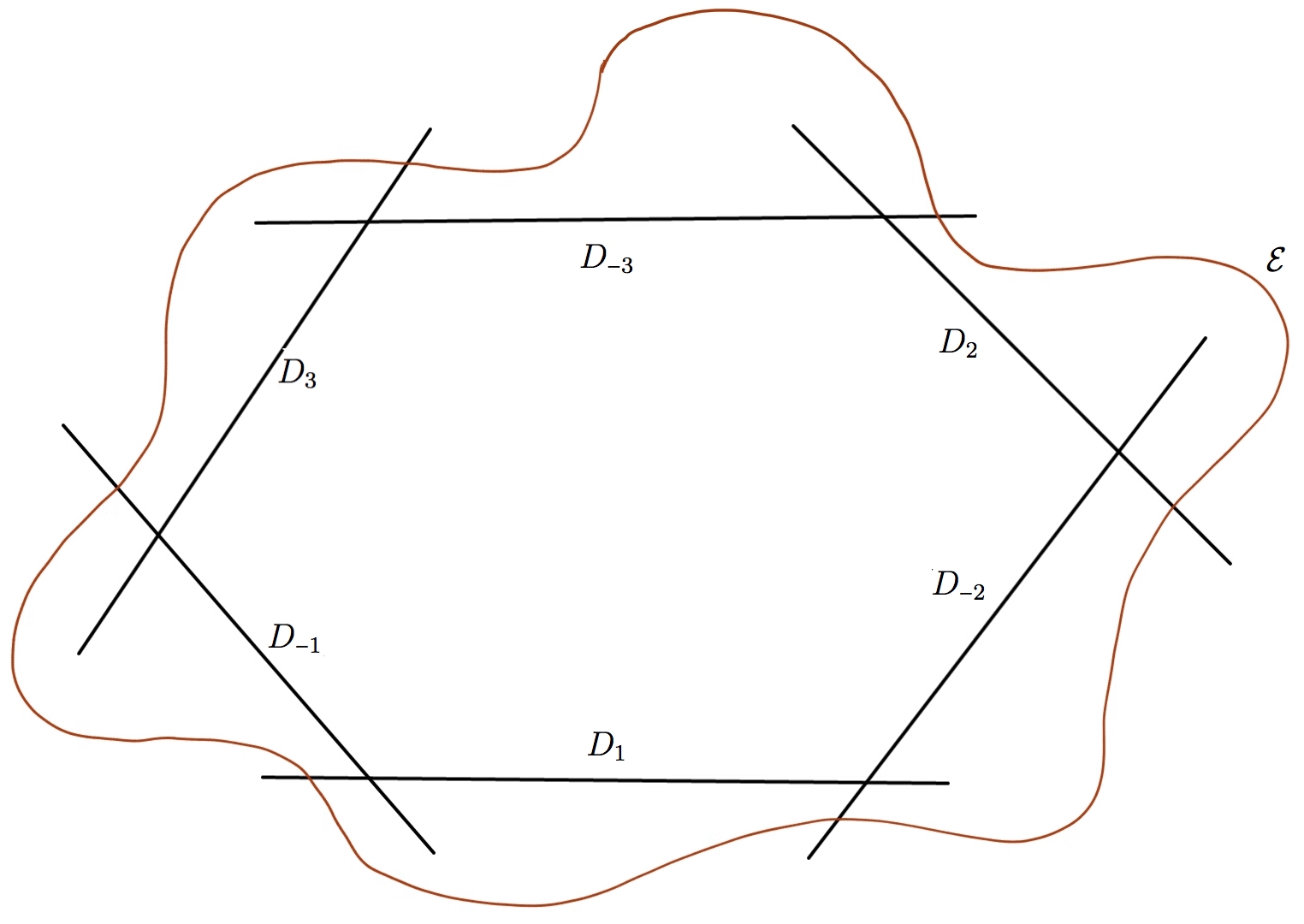}
\caption{The geometry underlying the sunrise graph after blowing up.}
\label{figure_sunrise_restricted}
\end{figure}

The Betti realization of the graph motive is given by $R^n\pi_*j_!\Q,$
where $j : D \setminus \mathcal{E} \cap D \setminus D \hookrightarrow P^G \setminus \mathcal{E}$. We define $S$ so that the Betti realisation is a local system on its complex points\footnote{$S$ is the complemenet in the space of generic kinematics of a closed subscheme typically referred to as the Landau variety.}. We denote it by $H^n_{\Beta}(P^G \setminus \mathcal{E}, D \setminus \mathcal{E} \cap D)_{/S}$ as its fiber at $t \in S(k)$ is the Betti cohomology of the complement of $\mathcal{E}_t$ in $P^G$ relative to the divisor $D \setminus \mathcal{E}_t \cap D$. 

Let $D_J = \cap_{j \in J} D_j$, and let $\Omega_{D_J/S}^{\bullet}$ denote the sheaf on $P^G \setminus \mathcal{E}$ which is the direct image of the corresponding sheaves of Kähler differentials on $D_J$, and which vanishes outside of $D_J$. Let 
 $I = \{\pm1,\pm2,\pm3\}$. Consider the double complex of sheaves on $P^G \setminus \mathcal{E}$
\begin{equation}
\label{relative-double-complex}
\Omega_{D_{\bullet}/S}^{\bullet} : \quad \Omega_{P^G \setminus \mathcal{E}/S}^{\bullet} \rightarrow \bigoplus\limits_{j\in I} \Omega_{D_j/S}^{\bullet} \rightarrow \bigoplus\limits_{\substack{J \subset \in I, \\|J|=2}} \Omega_{D_J/S}^{\bullet}
\end{equation}
where the horizontal maps are pullbacks along inclusions $D_{j} \hookrightarrow P^G \setminus \mathcal{E}$ with alternating signs. Then define the de Rham realization of the graph motive to be
\begin{equation}
\label{relative-de-rham-definition}
H^2_{\mathrm{dR}}(P^G \setminus \mathcal{E},D \setminus \mathcal{E} \cap D)_{/S} =\mathbb{R}^2\pi_*(\textrm{Tot}^{\bullet}(\Omega_{D_{\bullet}/S}^{\bullet})),
\end{equation}
where $\textrm{Tot}^{\bullet}$ denotes the total complex. It has a flat connection
\[
\nabla : H^2_{\mathrm{dR}}(X,D)_{/S} \rightarrow H^2_{\mathrm{dR}}(X,D)_{/S} \otimes \Omega^1_{S/k} 
\]
by a relative version of \cite{KO}. Its fiber at $t$ is the algebraic de Rham cohomology of the complement of $\mathcal{E}_t$ in $P^G$ relative to the divisor $D \setminus \mathcal{E}_t \cap D$. 

We put everything together to define the 'motive' of the sunrise graph, i.e. the object in the category $\mathcal{H}(S)$ given by the triple
\begin{equation}
(H^2_{\Beta}(P^G \setminus \mathcal{E}, D \setminus \mathcal{E} \cap D)_{/S},H^2_{\dR}(P^G \setminus \mathcal{E}, D \setminus \mathcal{E} \cap D)_{/S}, c),
\end{equation}
where the comparison $c$ is given by \cite[Proposition 2.28]{Deligne4}. It follows from Saito's theory of mixed Hodge modules \cite{Saito88,Saito90} that the Betti realization is a (admissible, graded-polarizable) variation of mixed Hodge structures. We will drop the notation $/S$ henceforth to avoid cluttering, and state explicitly when we are working with a fiber over $S$.

\begin{prop}
The weight-graded (semi-simple) object associated to each fiber $M$ of $mot_G$ is 
$$
gr^W M = H^1(\mathcal{E})(-1) \bigoplus \mathbb{Q}(-1)^{\oplus 3} \bigoplus \mathbb{Q}(0)
$$
\end{prop}
\begin{proof}
Since $\mathcal{E}$ is smooth the residue/Gysin long exact sequence applies to 
$$
M = H^2(P^G \setminus \mathcal{E}, D \setminus \mathcal{E} \cap D)
$$ so that we have 
\begin{equation}
\label{residue_seq_sunrise}
0 \rightarrow H^2(P^G,D) \rightarrow M \xrightarrow{res} H^1(\mathcal{E},D\cap\mathcal{E})(-1) \xrightarrow{Gysin} H^3(P^G,D) .
\end{equation}
We can apply the relative cohomology long exact sequence to $H^2(P^G,D)$, part of which is
$$
H^1(P^G) \rightarrow H^1(D) \rightarrow H^2(P^G,D) \rightarrow H^2(P^G) \rightarrow H^2(D) . 
$$
On the right we have $H^2(D) = \mathbb{Q}(-1)^{\oplus 6}$ since $D$ is a union of 6 projective lines --- the three exceptional divisors and the three strict transforms of the coordinate hyperplanes. We have $H^2(P^G) = \mathbb{Q}(-1)^{\oplus 4}$, where one of the classes corresponds to $H^2(\mathbb{P}^2)$ and is represented by a general projective line in $\mathbb{P}^2$, while the remaining three correspond to the three exceptional divisors. The rightmost arrow is therefore injective. Furthermore, we have $H^1(P^G) = 0$, and $H^1(D) = \mathbb{Q}(0)$ since $D$ is a union of 6 projective lines such that each two meet in a single point. It follows that $H^2(P^G,D) = \mathbb{Q}(0)$.

We can apply the relative cohomology long exact sequence to $H^1(\mathcal{E},D \cap \mathcal{E})(-1)$ to get 
$$
0 \rightarrow \widetilde{H}^0(D \cap \mathcal{E})(-1) \rightarrow H^1(\mathcal{E},D \cap \mathcal{E})(-1) \rightarrow H^1(\mathcal{E})(-1) \rightarrow 0
$$
where on the left hand side we have reduced cohomology. There are six points of intersection $\mathcal{E}\cap D$ so the rank of $\widetilde{H}^0(D \cap \mathcal{E})(-1)$ is 5. 

Finally, we need to know the image of the residue map in \eqref{residue_seq_sunrise}, or equivalently the kernel of the Gysin morphism. From the relative cohomology long exact sequence we have that
$$
H^2(P^G) \rightarrow H^2(D) \rightarrow H^3(P^G,D) \rightarrow H^3(P^G).
$$
Here $H^3(P^G)=0$, while $H^2(P^G)$ and $H^2(D)$ have ranks 4 and 6 respectively. It follows that $H^3(P^G,D)$ has rank 2. One can see that the Gysin map in \eqref{residue_seq_sunrise} is surjective, for example by showing that $H^3(P^G \setminus \mathcal{E}, D \setminus \mathcal{E} \cap D)=0$ using the relative cohomology long exact sequence. Therefore we coclude that 
$$
gr^W \textrm{Im}(res) = H^1(\mathcal{E})(-1) \bigoplus \mathbb{Q}(-1)^{\oplus 3}  .
$$
The result follows.
\end{proof}
\end{subsection}

\begin{subsection}{Basis of de Rham cohomology}
\label{basis of de Rham cohomology section}
\begin{subsubsection}{The Feynman integrand and the holomorphic form on $\mathcal{E}$}
\label{sunrise Feynman integrand}
We may check by direct computation\footnote{See \cite[\S 6.6]{Brown2} for a more general argument for all Feynman graphs.} that $\pi_G^*(\omega_G(m,q))$, where $\omega_G(m,q) = \frac{\Omega_G}{\Xi_G}$ is the Feynman integrand for the sunrise graph in $d=2$ space-time dimensions, does not acquire any poles along the exceptional divisors. Therefore it defines a section of $(mot_G)_{\dR}$ over $S$. We also note that we have a commutative diagram 
\begin{equation}
\label{blow-up residue commute diagram}
\xymatrix{
            H^2(\mathbb{P}^2 \setminus \mathcal{E}) \ar[d]^{res} \ar[r]& H^2(P^G \setminus \mathcal{E}) \ar[d]^{res}\\
            H^1(\mathcal{E})(-1) \ar[r]^{=} & H^1(\mathcal{E})(-1).
            }
\end{equation}
The map on the top row is induced by $i \circ p^{-1}$
$$
P^G \setminus \mathcal{E} \overset{i}{\hookleftarrow} P^G \setminus (\mathcal{E} \cup D_{E}) \overset{p}{\underset{\cong}{\longrightarrow}} \mathbb{P}^2 \setminus \mathcal{E},
$$
where on the right we have the restriction of the projection $\pi_G$ to the complement of the exceptional divisors where it is an isomorphism. A well known theorem of Griffiths \cite[8.6]{G} tells us that, for $X \subset \mathbb{P}^n$ a smooth projective hypersurface, we have an isomorphism induced by the residue map
$$
A^n_k(X)/dA^{n-1}_{k-1}(X) \overset{res}{\underset{\cong}{\rightarrow}} F^{n-k}H^{n-1}_{\Beta}(X;\mathbb{C})_{prim}
$$
where on the left hand side we have $A^n_k(X)$ rational $n$-forms on $\mathbb{P}^n$ with poles of order $k$ along $X$, and on the right hand side we have the primitive part of the Betti cohomology of $X$\footnote{For $X \subset \Pp^n$ a smooth projective variety, and $\xi \in H^2(X;\CC)$ a hyperplane section the primitive cohomology of $X$ is defined as $H^{n-k}(X;\CC)_{prim}=\ker\left(\bigcup \xi^{k+1} : H^{n-k}(X;\CC) \rightarrow H^{n+k+2}(X;\CC)\right)$.}. When $n$ is even, as is our case, all cohomology of $X$ is primitive \cite[Lemma 8.20]{G}. Using this result in the case $n=2,k=1$ and $X = \mathcal{E}$, along with the commutativity of the diagram \eqref{blow-up residue commute diagram}, we obtain that $res((\pi_G^*(\omega_G(m,q)))$ is the holomorphic differential on $\mathcal{E}$, or, in other words, $$res((\pi_G^*(\omega_G(m,q))) \in F^1H^1_{\dR}(\mathcal{E}) = F^2H^1_{\dR}(\mathcal{E})(-1).
$$ 

We note that $[\pi_G^*(\omega_G(m,q))]$ is of weight 3. This will be used later when we complete the de Rham basis, as well as when we show that the Feynman integral associated to the sunrise graph can be expressed in terms of integrals of modular forms.

\end{subsubsection}

\subsubsection{Weight 0}

That there are relations between periods of different Feynman graphs arising from the recursive structure of graph motives \cite[\S7.2]{Brown2}. In particular for any subgraph $\gamma \subset G$, where we consider $\gamma$ to have no dependence on kinematic parameters unless it contains all the edges of $G$ with non-trivial masses and momenta, we have morphisms in the category $\mathcal{H}(S)$ of graph motives
\begin{equation}
\label{face_map}
mot_{\gamma} \otimes mot_{G/\gamma} \rightarrow mot_G,
\end{equation}
where $G/\gamma$ is the graph obtained by contracting $\gamma$ \cite[Theorem 7.8]{Brown2}. As morphisms in $\mathcal{H}(S)$ these morphism have a de Rham and a Betti component, i.e. a morphism $\phi: (\mathcal{V}_{\dR,1},\mathbb{V}_{\Beta,1},c_1) \rightarrow (\mathcal{V}_{\dR,2},\mathbb{V}_{\Beta,2},c_2)$ consists of morphisms $\phi_{\dR}: \mathcal{V}_{\dR,1} \rightarrow \mathcal{V}_{\dR,2}$ and $\phi_{\Beta}: \mathbb{V}_{\Beta,1} \rightarrow \mathbb{V}_{\Beta,2}$ which commute with comparison isomorphisms $c_1,c_2$. In our case, we use the de Rham component of such a morphism, where we take $\gamma = G \setminus e_3$\footnote{We can equally take $G\setminus e_1$ or $G\setminus e_2$, their graph motives are the same.} to be the two-edge one loop `bubble' graph \eqref{general-bubble} with trivial masses and momenta. Its graph motive is $$mot_{G \setminus e_3} = H^1(\mathbb{P}^1\setminus [1:-1],\{[1:0],[0:1]\}).$$

\label{weight-0-sunrise}
\begin{lemma}
  Let $\nu_0$ be the integrand of the bubble graph $\gamma = G \setminus e_3$ with no kinematic dependence in $d=4$ dimensions. Consider the morphism
  $$
  \phi: mot_{G \setminus e_3} \rightarrow mot_G
  $$
  in $\mathcal{H}(S)$. Then $\phi_{\rm{dR}}([\nu_0])$ generates $W_0(mot_G)_{\rm{dR}}$.
\end{lemma}
\begin{proof}
Without loss of generality we delete the edge $e_3$ of $G$. In $d=4$ the Feynman integrand of $G \setminus e_3$ is
$$
\nu_0 = \frac{\alpha_1 d\alpha_2-\alpha_2 d\alpha_1}{(\alpha_1+\alpha_2)^2}.
$$
Let $\mathbb{A}_{12,2} = \Sp(k[u,v])$ be an affine open in $P^G$ where we have $(\pi_G|_{\mathbb{A}_{12,2}})^*(\alpha_1)=u$ and $(\pi_G|_{\mathbb{A}_{12,2}})^*(\alpha_2)=uv$, and let us consider the form
$$
\omega_0 = \frac{m_3^4\alpha_3^4(\alpha_1d\alpha_2-\alpha_2d\alpha_1)}{\Xi_G(m,q)^2} .
$$
Note that 
$$
\pi_G^*(\omega_0)|_{\mathbb{A}_{12,2}} = \frac{m_3^4dv}{(uv+(m_1^2u+m_2^2uv+m_3^2)(uv+v+1))^2},
$$
When this is restricted the exceptional divisor in $\mathbb{A}_{12,2}$, which is given by $u=0$ in this chart, we get
$$
\pi_G^*(\omega_0)|_{\mathbb{A}_{12,2}\cap\{u=0\}} = \frac{dv}{(1+v)^2}.
$$
We can see then that we have
$$
\pi_G^*(\omega_0)|_{\mathbb{A}_{12,2}\cap\{u=0\}} = \phi_{\rm{dR}}|_{\alpha_1 \not = 0}(\nu_0|_{\alpha_1 \not = 0}).
$$
Analogously, if we replace $\mathbb{A}_{12,2}$ by $\mathbb{A}_{12,1}$ with $(\pi_G|_{\mathbb{A}_{12,1}})^*(\alpha_1)=uv$ and $(\pi_G|_{\mathbb{A}_{12,1}})^*(\alpha_2)=v$, and $\alpha_1 \not = 0$ by $\alpha_2 \not = 0$, we get that $[\pi_G^*(\omega_0)|_{D_{-3}}] = [\phi_{\rm{dR}}(\nu_0)]$. The fact that the restriction of $\pi_G^*(\omega_0)$ to the exceptional divisor is not exact implies that its differential defines a non-trivial class in relative cohomology $$
[\phi_{\rm{dR}}(\nu_0)] = [\pi_G^*(\omega_0)|_{D_{-3}}] = [d(\pi_G^*(\omega_0))] \in (mot_G)_{\rm{dR}}
$$
This class spans the weight zero subspace $W_0 (mot_G)_{\rm{dR}}$.
\end{proof}

\begin{figure}[h]
\centering
\begin{tikzpicture}

\SetGraphUnit{3}
  
  \SetUpEdge[lw = 1pt,
  color      = black,
  labelcolor = white,
  labelstyle = {sloped,above,yshift=2pt}]
  
  \SetUpVertex[FillColor=black, MinSize=8pt, NoLabel]

  \Vertex[x=2,y=0]{1}
  \Vertex[x=6,y=0]{2}
  \Vertex[x=0,y=0,empty=true]{4}
  \Vertex[x=8,y=0,empty=true]{5}

  \tikzset{EdgeStyle/.append style = {bend left}}
  \Edge[label=1](1)(2)
  \Edge[label=2](2)(1)
  \tikzset{EdgeStyle/.style={postaction=decorate,decoration={markings,mark=at position 0.7 with {\arrow{latex}}}}}
  \Edge(4)(1)
  \Edge(5)(2)
\end{tikzpicture}
\caption{The bubble graph with no kinematic dependence.}
\label{general-bubble}
\end{figure}
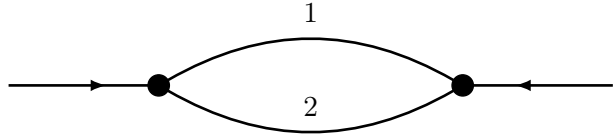

\begin{subsubsection}{Completing the de Rham basis}
\label{completing the de rham basis}
Only a 1-dimensional subspace of the de Rham realization of the sunrise graph motive is spanned by the image under \eqref{face_map} of the de Rham class of the Feynman integrand of a subquotient graph with a strictly smaller number of edges. This follows from the relative cohomology long exact sequence
$$
0 \rightarrow H^1(D \setminus \mathcal{E} \cap D) \rightarrow mot_G \rightarrow H^2(P^G \setminus \mathcal{E}) \rightarrow 0, 
$$
where on the left hand side we have $H^1(D \setminus \mathcal{E} \cap D) = \mathbb{Q}(0)$. This is similar to the triangle graph with all three vanishing masses (see \cite[\S 5.2]{MatijaT}). It contrasts with, for example, the case of the one-loop four-edge graph with all three non-vanishing masses (see \cite[\S 4]{MatijaT}), where the images of the de Rham classes of Feynman integrands of subquotient graphs with strictly fewer edges, along with the Feynman integrand of the one-loop four-edge graph itself, form a basis of the de Rham realization of its graph motive.

We will use the Feynman integrands of graphs obtained from the sunrise by iterated subdivision of edges to complete $\{\phi_{\dR}(\nu_0),[\pi_G^*(\omega_G)]\}$ to a basis of $(mot_G)_{\dR}$ via the de Rham component of morphisms of graph motives defined in \S\ref{section: subdivison of edges}. Consider the following diagram
$$
\xymatrix@C=1.5em{
            0 \ar[r] & H^1(D \setminus \mathcal{E} \cap D) \ar[r] & H^2(P^G \setminus \mathcal{E}, D \setminus \mathcal{E} \cap D) \ar[d]^{res} \ar[r] & H^2(P^G \setminus \mathcal{E}) \ar[r] & 0\\
            0 \ar[r] & \widetilde{H}^0(\mathcal{E} \cap D)(-1) \ar[r]^{p} & H^1(\mathcal{E}, \mathcal{E} \cap D)(-1) \ar[r] & H^1(\mathcal{E})(-1) \ar[r] & 0
            }
$$
The object on the top left is isomorphic to $\mathbb{Q}(0)$, and the image of its de Rham realization is spanned by $\phi_{\dR}([\nu_0])$ \ref{weight-0-sunrise}. The form $\pi_G^*(\omega_G)$ lifts the holomorphic differential on $\mathcal{E}$. We now want to choose a lift of a differential of the second kind on $\mathcal{E}$ and the three classes in $\Im(p) \cap \Im(res)$ via the residue morphism. Let us order the points $\mathcal{E} \cap D$ and denote $P_{2i-1} = D_{-i} \cap \mathcal{E}$ and $P_{2i} = D_{i} \cap \mathcal{E}$ for $1 \leq i \leq 3$. As discussed in \S\ref{standard elliptic de rham realization} a basis of $H_{\dR}^1(\mathcal{E},D \cap \mathcal{E})$ is given by a differential form of the first kind on $\mathcal{E}$, a differential form of the second kind on $\mathcal{E}$ and images of classes in $\tilde{H}^0(D \cap \mathcal{E})$ under the map
$$
\tilde{H}_{\dR}^0(D \cap \mathcal{E}) \rightarrow H^1_{\dR}(\mathcal{E},D \cap \mathcal{E}),
$$
which are given by the classes $[df_i]$ where $f_i$ is a meromorphic function on $\mathcal{E}$ which takes the value $1$ at the point $P_i$ and $0$ at $P_j$ for $j \not = i$, for $1 \leq i \leq 6$. They satisfy the relation $[df_1]+\ldots+[df_6] = 0$.

\begin{prop}
\label{basis-elliptic-dr}

Define $\eta_G$ to be the form such that $\eta_G \wedge dy$ is the pullback via $\tilde{\rho}$ of the Feynman integrand associated to the graph $G_{s(e_1)}$ in $d=2$ space-time dimensions. In addition, define $\nu_1,\nu_2,\nu_3$ to be the forms such that $\nu_i \wedge \wedge_{j=1}^3 dy_j$, for $1 \leq i \leq 3$ are the pullbacks via $\tilde{\rho}$ of the Feynman integrands associated to the graphs 
$G_{s(e_1,e_2^2)},G_{s(e_1^2,e_3)},G_{s(e_2,e_3^2)}$ in $d=4$ respectively (see figure \ref{fig:subdivided sunrise}). Then the form $\eta_G$ is sent to a form of the second kind on $\mathcal{E}$ by the residue morphism. Moreover, the images of $\omega_G,\eta_G,\nu_1,\nu_2,\nu_3$ under the residue morphism are linearly independent over $k(m,q)$. Hence their classes, along with the class of $\phi_{\dR}(\nu_0)$, form a basis of $(mot_G)_{\dR}$.
\end{prop}
\begin{proof}
By \S\ref{section: subdivison of edges} we have that $\eta_G = \pi_G^*\left(\frac{-\alpha_1\Psi_G}{\Xi_G}\omega_G\right)$. Since the numerator $\alpha_1\Psi_G$ of $\eta_G$ is not the Jacobian ideal of $\Xi_G$ (i.e. the ideal generated by its partial derivatives with respect to $\alpha_i, 1\leq i \leq 3$) it follows that we cannot reduce the order of the pole along $\mathcal{E}$ by \cite[Proposition 4.6]{G}, so the class of its residue in $H^1_{\dR}(\mathcal{E})$ is not a $k(m,q)$-multiple of the class of the global form on $\mathcal{E}$.

We will consider the basis 
$$
\{[df_2],\ldots,[df_6],[res(\eta_G)],[res(\omega_G)]\} \text{ of } H^1_{\dR}(\mathcal{E},D \cap \mathcal{E}),
$$ 
and write $[res(\nu_i)]$, for $1 \leq i \leq 3$ as a linear combination
$$
[res(\nu_i)] = \sum_{j=1}^5 a_{i,j}[df_{j+1}] + a_{i,6}[res(\eta_G)] + a_{i,7}[res(\omega_G)].
$$
Note that by \S\ref{section: subdivison of edges} we have 
$$
\nu_1 = \pi_G^*\left(\frac{-\alpha_1\alpha_2^2}{\Xi_G}\omega_G\right), \quad \nu_2 = \pi_G^*\left(\frac{-\alpha_1^2\alpha_3}{\Xi_G}\omega_G\right), \quad \nu_3 = \pi_G^*\left(\frac{-\alpha_2\alpha_3^2}{\Xi_G}\omega_G\right) .
$$
We can determine the $a_{i,j}$'s by considering the affine covering we used in the proof of proposition \ref{subdivide morphism}. In the case of the sunrise it has six affine open sets, each of which contains two of the points $P_k$, $1 \leq k \leq 6$. Denote the affine open containing the points $P_k,P_l$ by $\mathbb{A}_{k,l}$. We can compose the residue with the pullback along $r_{k,l} : \mathbb{A}_{k,l} \cap \mathcal{E} \rightarrow \mathcal{E}$ to obtain the commutative diagram
$$\xymatrix{
            \tilde{H}_{\dR}^0(D \cap \mathcal{E}) \ar[d] \ar[r]& H_{\dR}^1(\mathcal{E},D \cap \mathcal{E}) \ar[d]^{r_{k,l}^*}\\
            \tilde{H}_{\dR}^0(\{P_k,P_l\}) \ar[r] & H_{\dR}^1(\mathbb{A}_{k,l} \cap \mathcal{E},\{P_k,P_l\})
            } 
$$
Computing explicitly on the bottom row we can write
$$
r_{k,l}^*(res(\nu_i)) = a_{i,6}r_{k,l}^*(res(\omega_G)) + a_{i,7}r_{k,l}^*(res(\eta_G)) + dG_{i;k,l}
$$
for a polynomial $G_{i;k,l}$. To obtain it we reduce the pole order of $r_{k,l}^*(res(\nu_i))$. This can always be done as $\mathcal{E}$ is smooth and so the partial derivatives of $\Xi_G$ and $\Xi_G$ itself don't have any common zeroes, so by weak Nullstellensatz the ideal generated by them will contain any numerator, in particular that of  $r_{k,l}^*(res(\nu_i))$ (see \cite[Proposition 4.6]{G}). We obtain equations $$a_{i,k} - a_{i,l} = G_{i;k,l}(P_k) - G_{i;k,l}(P_l).$$
We can repeat the process for $1\leq k \leq 5, l=k+1$ and solve the resulting system of equations to obtain all $a_{i,j}$'s. They can be computed algorithmically, and are given explicitly in appendix \ref{Residues of integrands}. This allows us to check that they are linearly independent.
\end{proof}
\end{subsubsection}
\end{subsection}

\begin{subsection}{The motivic Galois coaction}
\begin{subsubsection}{Arbitrary masses}
\label{Arbitrary masses coaction section}
We will write the coaction for the sunrise graph in terms of the de Rham basis in proposition \ref{basis-elliptic-dr}. Note that the expression for the coaction depends on the choice of basis, and we show an alternative expression in \S\ref{coaction in second basis}. In the following statement we drop the notation $(m,q)$ for the dependence on masses and momenta for brevity. We remind the reader of the definitions of de Rham periods $K_1^{\fdr},K_{2,\eta}^{\fdr}, F^{\fdr}_{P_i,P_j}$ in \S\ref{elliptic-dr-periods-defi}, where these are applied with respect to the elliptic curve $\mathcal{E}$ and the basis of $H^1_{\dR}(\mathcal{E})$ given by $\{[res(\omega_G)], [res(\eta_G)]\}$. We further define 
\begin{equation}
    \label{de-rham-elliptic-combi}
        F^{\fdr}_{\underline{b}} = \sum_{\substack{1 \leq k \leq 5}}b_{i}F^{\fdr}_{P_1,P_{i+1}}
\end{equation}
where $\underline{b}=(b_1,\ldots,b_5)$ is a vector with coefficients in $k_S$, i.e. rational functions of masses and momenta. Finally we define a de Rham Feynman period
$$
I^{\fdr}_{G,G \setminus e_3} = \left[ mot_G,\left[\phi_{\dR}(\nu_0)\right]^{\vee},[\pi_G^*(\omega_G)]\right]^{\fdr}
\footnote{Here we have defined a de Rham Feynman amplitude in relation to a chosen weight 0 basis element, in our case coming from the bubble graph obtained by deleting the edge $e_3$ from the sunrise. The reader could compare this to the case of one-loop graphs, when the graph motives are mixed Tate, and this can done without a choice of a weight 0 basis element since the weight filtration is split by the Hodge filtration in the de Rham realization -- see \cite[\S 2.5.2]{MatijaT}.}
$$
\begin{thm} 
\label{sunrise-cocation-generic-thm}
The motivic Galois coaction on the motivic Feynman amplitude associated to the sunrise Feynman graph with respect to the de Rham basis defined in Proposition \ref{basis-elliptic-dr} is
\begin{equation}
\begin{split}
    \Delta(I^{\fm}_G) &= I^{\fm}_G \otimes K_1^{\fdr}\mathbb{L}^{\fdr} + I^{\fm}_{G_{s(e_1)}} \otimes K_{2,\eta}^{\fdr}\mathbb{L}^{\fdr} + I^{\fm}_{G \setminus e_3} \otimes I^{\fdr}_{G,G \setminus e_3} +\\
    & +I^{\fm}_{G_{s(e_1,e_2^2)}} \otimes F^{\fdr}_{\underline{b_1}}\mathbb{L}^{\fdr}  + I^{\fm}_{G_{s(e_1^2,e_3)}} \otimes F^{\fdr}_{\underline{b_2}}\mathbb{L}^{\fdr} + I^{\fm}_{G_{s(e_2,e_3^2)}} \otimes F^{\fdr}_{\underline{b_3}}\mathbb{L}^{\fdr}  \, ,
\end{split}
\end{equation}
where $\underline{b_i} = (b_{i,1},\ldots,b_{i,5})$ for $1 \leq i \leq 3$, and $b_{i,j}$ are defined in \eqref{bij definition}.
\end{thm}

\begin{proof}
The general formula for the coaction \eqref{general-coaction} applied with respect to the basis of $(mot'_G)_{\dR}$ from proposition \ref{basis-elliptic-dr} gives
\begin{equation}
\begin{split}
\Delta(I^{\fm}_G) & = \left[mot_G,\left[\sigma_G\right],\left[\pi_G^*(\omega_G)\right] \right]^{\fm} \otimes \left[mot_G,\left[\pi_G^*(\omega_G)\right]^{\vee},\left[\pi_G^*(\omega_G)\right] \right]^{\fdr}  \\
& + \left[mot_G,\left[\sigma_G\right],\left[\eta_G\right] \right]^{\fm} \otimes \left[mot_G,\left[\eta_G\right]^{\vee},\left[\pi_G^*(\omega_G)\right] \right]^{\fdr}  \\
& + \sum_{i=1}^3 \left[mot_G,\left[\sigma_G\right],\left[\nu_i\right] \right]^{\fm} \otimes \left[mot_G,\left[\nu_i\right]^{\vee},\left[\pi_G^*(\omega_G)\right] \right]^{\fdr}  \\
+& \left[mot_G,\left[\sigma_G\right],\left[\phi_{\dR}(\nu_0)\right] \right]^{\fm} \otimes \left[mot_G,\left[\phi_{\dR}(\nu_0)\right]^{\vee},\left[\pi_G^*(\omega_G)\right] \right]^{\fdr}.
\end{split}
\end{equation}

By definition $\left[mot_G,\left[\sigma_G\right],\left[\pi_G^*(\omega_G)\right] \right]^{\fm} = I^{\fm}_G$ , and via the subdivision morphsim $\tilde{\rho}$ in \eqref{subdivided motivic period} we have
$
\left[mot_G,\left[\sigma_G\right],\left[\eta_G\right] \right]^{\fm} = I^{\fm}_{G_{s(e_1)}}.
$
Similarly, the remaining motivic periods are identified with motivic Feynman amplitudes of graphs obtained by subdividing edges, except for
$\left[mot_G,\left[\sigma_G\right],\left[\phi_{\dR}(\nu_0)\right] \right]^{\fm}$. The latter is equal to $I^{\fm}_{G \setminus e_3}$ via the face map $\phi$ (see \S\ref{weight-0-sunrise}).

It remains to identify the de Rham periods. As in \S\ref{completing the de rham basis} consider the basis $\{[df_2],\ldots,[df_6],[res(\eta_G)],[res(\omega_G)]\}$ of $H^1_{\dR}(\mathcal{E}, D \cap \mathcal{E})$. Let $\{e_i:1 \leq i \leq 6\}$ denote the basis $\{\left[\phi_{\dR}(\nu_0)\right],[\nu_1],[\nu_2],[\nu_3],[\eta_G],[\pi_G^*(\omega_G)]\}$ of $(mot_G)_{\dR}$ in that order, and by $\{e_i^{\vee}:1 \leq i \leq 6\}$ the associated dual basis. We can use the residue morphism
$$
mot_G=H^2(P^G \setminus \mathcal{E}, D \setminus \mathcal{E} \cap D) \xrightarrow{res} H^1(\mathcal{E}, D \cap \mathcal{E})(-1)
$$
to realise an equivalence of de Rham periods
$$
\left[mot_G, e_i^{\vee}, \left[\pi_G^*(\omega_G)\right]\right]^{\fdr} = \left[H^1(\E,D\cap \E)(-1), \xi_i^{\vee} ,[res(\omega_G)(-1)]\right]^{\fdr},
$$ for $2 \leq i \leq 6$, where $\xi_i^{\vee} \in H^1(\mathcal{E}, D \cap \mathcal{E})(-1)^{\vee}$ is a functional such that $res^{\vee}\left(\xi_i^{\vee}\right) = e_i^{\vee}$. To determine $\xi_i^{\vee}$ we write it as a linear combination in terms of the dual basis to our chosen basis of $H^1(\mathcal{E}, D \cap \mathcal{E})(-1)$
$$
\xi_i^{\vee} = \sum_{k=2}^6b_{i,k-1}\left[df_k(-1)\right]^{\vee} + b_{i,6}\left[res(\eta_G)(-1)\right]^{\vee} + b_{i,7}\left[res(\omega_G)(-1)\right]^{\vee}.
$$
Now note that we must have 
$$
res^{\vee}\left(\xi_i^{\vee}\right)(e_j) = \xi_i^{\vee}(res(e_j)) = \delta_{ij}.
$$
It follows that each coefficient $b_{i,k}$ is the $(k,i)$ entry of the matrix
\begin{equation}
\label{bij definition}
\begin{pmatrix}
a_{1,1} & \ldots & a_{1,7} \\
\vdots & \ddots & \\
a_{5,1} & \ldots & a_{5,7}
\end{pmatrix}^{-1},
\end{equation}
where the first three rows are the coefficients of $res(\nu_i)$, $1 \leq i \leq 3$, in Appendix \ref{Residues of integrands}, and the last two rows are $(0,\ldots0,1,0),(0,0,0,\ldots,0,1)$. Note that $b_{i,6}=b_{i,7}=0$ for $1 \leq i \leq 3$. Hence, for $1 \leq i \leq 3$, we have
\begin{equation}
\label{b prime lin comb de rham}
\begin{split}
\left[mot_G, e_i^{\vee}, \left[\pi_G^*(\omega_G)\right]\right]^{\fdr} &= \sum_{k=1}^5b_{i,k}\left[H^1(\E,D\cap \E), \left[df_{k+1}\right]^{\vee} ,[res(\omega_G)]\right]^{\fdr}\mathbb{L}^{\fdr},
\end{split}
\end{equation}
where $\mathbb{L}^{\fdr}$ accounts for the Tate twist. Furthermore we have
\begin{equation*}
\left[mot_G,\left[\eta_G\right]^{\vee},\left[\pi_G^*(\omega_G)\right] \right]^{\fdr} = \left[H^1(\E,D\cap \E), \left[res(\eta_G)\right]^{\vee} ,[res(\omega_G)]\right]^{\fdr} \mathbb{L}^{\fdr}.
\end{equation*}         
The latter is equal to $K_{2,\eta}^{\fdr}\mathbb{L}^{\fdr}$, by definition.
Similarly we have,
$$
         \left[mot_G,\left[\pi_G^*(\omega_G)\right]^{\vee},\left[\pi_G^*(\omega_G)\right] \right]^{\fdr} = \left[H^1(\E,D\cap \E), \left[res(\omega_G)\right]^{\vee} ,[res(\omega_G)]\right]^{\fdr}\mathbb{L}^{\fdr} .
$$
The latter is equal to $K_1^{\fdr}\mathbb{L}^{\fdr}$.

Now, let $\iota$ be the inclusion of pairs $(\E,\{P_1,P_i\}) \rightarrow (\E,\{P_1,\ldots,P_6\})$. It induces a morphism of motives
$$
H^1(\E,\{P_1,\ldots,P_6\}) \xrightarrow{\iota^*} H^1(\E,\{P_1,P_i\})
$$
and an equivalence of de Rham periods
$$
\left[H^1(\E,\{P_1,\ldots,P_6\}),\left[df_i\right]^{\vee},[\omega]\right]^{\fdr} = \left[H^1(\E,\{P_1,P_i\}),\left[d\tilde{f}_i\right]^{\vee},[\omega]\right]^{\fdr} = F^{\fdr}_{P_i,P_1},
$$
where $\tilde{f}_i$ is a meromorphic function on $\E$ which is $1$ at $P_i$ and $0$ at $P_1$, and $[\omega] \in H^1_{\dR}(\mathcal{E})$. The equivalence holds because 
$$
(\iota^*)^T\left(\left[d\tilde{f}_i\right]^{\vee}\right)\left(\left[df_k\right]\right) = \left[d\tilde{f}_i\right]^{\vee}(\iota^*(\left[df_k\right])) = \begin{cases}
1, & k = i \\
0, & \text{otherwise.}
\end{cases}
$$
The remaining de Rham period is $\left[mot_G,\left[\nu_i\right]^{\vee},\left[\pi_G^*(\omega_G)\right] \right]^{\fdr} = I^{\fdr}_{G,G \setminus e_3}$ by definition.

\end{proof}
\end{subsubsection}

\begin{subsubsection}{Changing the de Rham basis}
\label{coaction in second basis}
We once again consider the differential forms $\pi_G^*(\omega_G)$, $\eta_G$, and $\phi_{\dR}(\nu_0)$, as in the previous theorem, and replace the forms $\nu_1,\nu_2,\nu_3$ with
\begin{equation}
\begin{split}
    &\mu_1 = \frac{m_1^2x(\partial_y\Xi_G - \partial_z\Xi_G)}{\Xi_G^2}\Omega_G, \\
    &\mu_2 = \frac{m_2^2y(\partial_z\Xi_G - \partial_x\Xi_G)}{\Xi_G^2}\Omega_G, \\
    &\mu_3 = \frac{m_3^2z(\partial_x\Xi_G - \partial_y\Xi_G)}{\Xi_G^2}\Omega_G.
\end{split}
\end{equation}
First note that these are global forms which are Feynman integrands associated to $k_S$-linear combinations of Feynman graphs obtained by subdivision of edges from the sunrise graph \S\ref{section: subdivison of edges}. We can show that the motivic periods associated to $\mu_i$ and the domain of integration $\sigma_G$ are motivic logarithms:
\begin{equation}
    \begin{split}
    &\left[mot_G,[\sigma_G],[\mu_1]\right]^{\fm} = \log^{\fm} \left(\frac{m_3^2}{m_2^2} \right) \\
    &\left[mot_G,[\sigma_G],[\mu_2]\right]^{\fm} = \log^{\fm} \left(\frac{m_1^2}{m_3^2} \right) \\
    &\left[mot_G,[\sigma_G],[\mu_3]\right]^{\fm} = \log^{\fm} \left(\frac{m_2^2}{m_1^2} \right).
    \end{split}
\end{equation}
This follows, for example, from the fact that motivic logarithms are determined by their periods, so a computation of the periods is enough to show the equality above.

 We can repeat the process from the previous theorem and proposition \ref{basis-elliptic-dr} to compute the coefficients in the expression
$$
\left[mot_G,[\mu_i]^{\vee},[\pi_G^*(\omega_G)]\right]^{\fdr} = \sum_{k=1}^5b_{i,k}'\left[H^1(\E,D\cap \E), \left[df_{k+1}\right]^{\vee} ,[res(\omega_G)]\right]^{\fdr}\mathbb{L}^{\fdr}.
$$
In this case they are rational. Each $b_{i,k}'$ corresponds to the entry $(k,i)$ of the matrix
\begin{equation}
\begin{pmatrix}
\frac{1}{6} & \frac{7}{24} & -\frac{5}{24} 
\\[6pt]
 \frac{1}{6} & -\frac{11}{24} & \frac{1}{24}
\\[6pt]
 0 & \frac{1}{4} & \frac{1}{4}
\\[6pt]
 \frac{1}{6} & \frac{1}{24} & -\frac{11}{24} 
\\[6pt]
 \frac{1}{6} & -\frac{5}{24} & \frac{7}{24}
\end{pmatrix}.
\end{equation}
Hence we get the following expression for the coaction on the sunrise
\begin{equation}
\label{sunrise coaction w logs}
\begin{split}
    \Delta(I^{\fm}_G) &= I^{\fm}_G \otimes K_1^{\fdr}\mathbb{L}^{\fdr} + I^{\fm}_{G_{s(e_1)}} \otimes K_{2,\eta}^{\fdr}\mathbb{L}^{\fdr} + I^{\fm}_{G \setminus e_3} \otimes I^{\fdr}_{G,G \setminus e_3} +\\ &+\log^{\fm} \left(\frac{m_3^2}{m_2^2} \right) \otimes F^{\fdr}_{\underline{b'_1}}\mathbb{L}^{\fdr}
    +\log^{\fm} \left(\frac{m_1^2}{m_3^2} \right)\otimes F^{\fdr}_{\underline{b'_2}}\mathbb{L}^{\fdr}
    +\log^{\fm} \left(\frac{m_2^2}{m_1^2} \right) \otimes F^{\fdr}_{\underline{b'_3}}\mathbb{L}^{\fdr}  \, ,
\end{split}
\end{equation}
where, as before, we set $\underline{b'_i}=(b_{i,1}',\ldots,b_{i,5}')$ and $F^{\fdr}_{\underline{b}}$ is defined in \ref{de-rham-elliptic-combi}.

\end{subsubsection}

\begin{subsubsection}{Equal masses}
\label{equal masses}
The coaction simplifies in the case when the masses are equal. We can observe from \eqref{sunrise coaction w logs} that motivic logarithms vanish in this case since they are determined by their periods. In terms of mixed Hodge structures this is reflected by the splitting of the extension of $\QQ$-mixed Hodge structures
$$
0 \rightarrow H^2(P)/[\mathcal{E}] \rightarrow H^2(P \setminus \mathcal{E}) \rightarrow H^1(\mathcal{E})(-1) \rightarrow 0 .
$$
This is shown in \cite[Lemma 6.16]{BV}, and is a consequence of the fact that the points of intersection $\mathcal{E} \cap D$ are torsion points of the elliptic curve when the masses are equal. The previous extension fits into the following diagram
\begin{equation}
\label{diagram-equal-masses}
\xymatrix{
            & & & 0 \ar[d] & \\
            & & & H^2(P)/[\mathcal{E}] \ar[d] & \\
            0 \ar[r] & H^1(D \setminus \mathcal{E} \cap D) \ar[r] & H^2(P^G \setminus \mathcal{E}, D \setminus \mathcal{E} \cap D) \ar[r] & H^2(P^G \setminus \mathcal{E}) \ar[r] \ar[d]^{res} & 0 \\
            & & & H^1(\mathcal{E})(-1) \ar[d] & \\
            & & & 0 & \\
            }
\end{equation}
where we write $H^2(P)/[\mathcal{E}]$ for $H^2(P)/\Imm(Gysin)$ and $Gysin$ is the Gysin morphism $H^0(\mathcal{E})(-1) \rightarrow H^2(P)$. Since the right-hand column splits we have a morphism $i : H^1(\mathcal{E})(-1) \rightarrow H^2(P^G \setminus \mathcal{E})$. We can pull back the short exact sequence in the middle row above via $i$ to a short exact sequence
\begin{equation}
0 \rightarrow \QQ(0) \rightarrow E \rightarrow H^1(\mathcal{E})(-1) \rightarrow 0,
\end{equation}
where we also use the isomorphism $H^1(D \setminus \mathcal{E} \cap D) \cong \QQ(0)$. Therefore we get an object $E$ of $\mathcal{H}(S')$ or rank 3, where $S'$ is the subspace of $S$ where the masses are equal. A basis for its de Rham realization is $$\{[\phi_{\dR}(\nu_0)],[res(\eta_G)],[res(\omega_G)]\},$$ and the motivic Feynman amplitude associated to the sunrise $I_G^{\fm}$ is equivalent to a motivic period of $E$. Hence we know that the coaction formula has 3 terms as it corresponds to the rank of $E$. The coaction formula can be written as
\begin{equation}
\label{sunrise coaction equal mass}
\begin{split}
    \Delta(I^{\fm}_G) &= I^{\fm}_G \otimes K_1^{\fdr}\mathbb{L}^{\fdr} + I^{\fm}_{G_{s(e_1)}} \otimes K_{2,\eta}^{\fdr}\mathbb{L}^{\fdr} + I^{\fm}_{G \setminus e_3} \otimes I^{\fdr}_{G,G \setminus e_3}.
\end{split}
\end{equation}
\end{subsubsection}
\end{subsection}
\end{section}

\part{The relative completion of modular groups and the sunrise}

\section{Motivic periods of the relative completion of the torsor of paths on a modular curve}
The theory of the relative completion of fundamental groups and path torsors, and their periods in the case of modular curves, is due to Hain \cite{Hain1,Hain2} and Brown \cite{Brown5}. We describe here the motivic lifts of those periods.

\begin{subsection}{Tannakian definition of the relative completion}
\label{tannakian definition}
Let $\mathcal{T}$ be a Tannakian category with two fiber functors $\omega_x,\omega_y$, and let $\mathcal{S} \hookrightarrow \mathcal{T}$ be a full semisimple Tannakian subcategory with fiber functors obtained by restriction of $\omega_x,\omega_y$. With respect to these two categories define a third category $\mathcal{W}(\mathcal{T},\mathcal{S}) \hookrightarrow \mathcal{T}$ such that its objects are objects $V \in \mathcal{T}$ equipped with a filtration
$$
0 = V_0 \subset V_1 \subset \ldots \subset V_n = V
$$
such that its graded pieces $V_i/V_{i-1}$ are objects of $\mathcal{S}$. The category $\mathcal{W}(\mathcal{T},\mathcal{S})$ is Tannakian with fiber functors $\omega_x,\omega_y$ restricted from $\mathcal{T}$ to the subcategory. Define 
$$
 \pi_1(\mathcal{T},\mathcal{S},\omega_x) = \textrm{Aut}^{\otimes}_{\mathcal{W}(\mathcal{T},\mathcal{S})}(\omega_x), \text { and }  \pi_1(\mathcal{T},\mathcal{S},\omega_x,\omega_y) = \textrm{Isom}^{\otimes}_{\mathcal{W}(\mathcal{T},\mathcal{S})}(\omega_x,\omega_y).
$$
Let $S_x = \textrm{Aut}^{\otimes}_{\mathcal{S}}(\omega_x)$, noting the different category to the one in the definition of $\pi_1(\mathcal{T},\mathcal{S},\omega_x)$. Then $S_x$ is a pro-reductive affine group scheme because the category $\mathcal{S}$ is semi-simple \cite[Proposition 2.23]{DeligneMilne}, and we get a morphism of affine group schemes $\pi_1(\mathcal{T},\mathcal{S},\omega_x) \rightarrow S_x$ induced by the inclusion of categories $\mathcal{S} \hookrightarrow \mathcal{W}(\mathcal{T},\mathcal{S})$. We get an exact sequence
$$
0 \rightarrow \mathcal{U}_x \rightarrow \pi_1(\mathcal{T},\mathcal{S},\omega_x) \rightarrow S_x \rightarrow 0 \,,
$$
where $\mathcal{U}_x$ is a pro-unipotent group. We also write $S_{x,y} = \textrm{Isom}^{\otimes}_{\mathcal{S}}(\omega_x,\omega_y)$, and similarly we have a morphism $\pi_1(\mathcal{T},\mathcal{S},\omega_x,\omega_y) \rightarrow S_{x,y}$.
\end{subsection}
\begin{subsection}{Betti and de Rham relative completion of $\pi_1$}
\label{B-dR-completion}
Let $X$ be a smooth geometrically connected scheme over a field $k \subset \mathbb{C}$, and let $x,y \in X(k)$ be two rational points. 
\begin{itemize}
    \item \textbf{Betti:} Let $\mathcal{L}_X$ be the category of local systems of finite-dimensional $k$-vector spaces on $X$, and $\omega_x^B$ the fiber functor sending a local system to its fiber over the point $x \in X(k)$ for some field $k$. Let $\mathcal{S}^B$ be a full semi-simple Tannakian subcategory of $\mathcal{L}_X$. Then define the \textit{relative Betti fundamental group of X} by 
    $$
    \pi_1^{B,\mathcal{S}}(X,x) = \pi_1(\mathcal{L}_X,\mathcal{S}^B,\omega_x^B).
    $$

    The relative completion can be defined for any group with respect to a morphism with Zariski dense image into the $k$-points of a reductive algebraic group by a universal property (see  \cite[\S 1]{Hain1}). If we consider the topological fundamental group $\pi_1(X,x)$, and $S^B_x = \textrm{Aut}^{\otimes}_{\mathcal{S}^B}(\omega_x^B)$, and we note that each local system on $X$ is equivalent to a $\pi_1(X,x)$-representation, we get a Zariski dense morphism 
    \begin{equation}
    \label{Zariski dense}
    \pi_1(X,x) \rightarrow S^B_x(k).
    \end{equation}
    Then $\pi_1^{B,\mathcal{S}}(X,x)$ is the relative completion of the group $\pi_1(X,x)$ with respect to the morphism \eqref{Zariski dense}. We also define the \textit{relative Betti fundamental groupoid}:
    $$
    \pi_1^{B,\mathcal{S}}(X,x,y) = \pi_1(\mathcal{L}_X,\mathcal{S}^B,\omega_x^B,\omega_y^B).
    $$
    It is an affine scheme over $k$, which also comes equipped with a morphism
    \begin{equation}
    \label{Betti-path}
    \begin{split}
    \pi_1(X,x,y) &\rightarrow \pi_1^{B,\mathcal{S}}(X,x,y)(k)\\
    \gamma & \mapsto \gamma^{\Beta},
    \end{split}
    \end{equation}
    which has Zariski dense image, and arises via an isomorphism of fiber functors given by the pullback along a smooth path $\gamma : [0,1] \rightarrow X(\CC)$.
    
    \item \textbf{de Rham:} Let $\mathcal{A}_X$ be the category of vector bundles on $X$, equipped with an integrable connection with regular singularities at infinity, and $\omega_x^{dR}$ the fiber functor sending a vector bundle to its fiber over the point $x \in X(k)$. Let $\mathcal{S}^{dR}$ be a full semi-simple Tannakian subcategory of $\mathcal{A}_X$. Then define the \textit{relative de Rham fundamental group}
    $$
    \pi_1^{dR,\mathcal{S}}(X,x) = \pi_1(\mathcal{A}_X,\mathcal{S}^{dR},\omega_x^{dR}).
    $$
    Similarly, define the \textit{relative de Rham fundamental groupoid}
    $$
    \pi_1^{dR,\mathcal{S}}(X,x,y) = \pi_1(\mathcal{A}_X,\mathcal{S}^{dR},\omega_x^{dR},\omega_y^{dR}).
    $$
    \item \textbf{Comparison:} The Riemann-Hilbert correspondence gives an equivalence of categories 
    $$
    \mathcal{L}_X \otimes \CC \sim \mathcal{A}_X \otimes \CC
    $$
    and thus induces a comparison isomoprhism
    \begin{equation}
    \label{grp-comparison}
    c : \pi_1^{B,\mathcal{S}}(X,x,y) \otimes \CC \xrightarrow{\sim} \pi_1^{dR,\mathcal{S}}(X,x,y) \otimes \CC,
    \end{equation}
    which in turn induces an isomoprhism of their affine rings. 
\end{itemize}
We are particularly interested in the previous definitions in the case when $X$ is a modular curve $X_{\Gamma}$ for a congruence subgroup $\Gamma$ of $\textrm{SL}_2(\ZZ)$, the category $\mathcal{S}^B$ is the category of local systems generated by $R^1f_*\QQ$ for $f:\mathcal{E} \rightarrow X_{\Gamma}$ the universal family of elliptic curves, and similarly $\mathcal{S}^{dR}$ is the category of vector bundles generated by the relative algebraic de Rham cohomology $H^1_{\dR}(\mathcal{E}/X_{\Gamma})$ with the Gauss-Manin connection. We will set one of the two base points to be a \textit{tangential base point} \cite[\S15.3 - 15.12]{DelGFD}, see also \cite[\S 4]{Brown5} at the cusp at infinity, denoted $\vec{1}_{i\infty}$. Setting the other base point to $\tau \in X_{\Gamma}(k)$ we get an object 
$$(\mathcal{O}(\pi_1^{B,\mathcal{S}}(X_{\Gamma},\vec{1}_{i\infty},\tau)), \mathcal{O}(\pi_1^{dR,\mathcal{S}}(X_{\Gamma},\vec{1}_{i\infty},\tau)) \otimes_{k} \mathcal{O}_{X_{\Gamma}}, c)
$$
in the category $\mathcal{H}(k) = \mathcal{H}(\Sp(k))$, denoted $\mathcal{O}(\pi_1^{\textrm{rel}}(X_{\Gamma},\vec{1}_{i\infty},\tau))$.
\begin{rmk}
Note that $\pi_1^{\textrm{rel}}(X_{\Gamma},x,y)$ is a right torsor over $\pi_1^{\textrm{rel}}(X_{\Gamma},x)$:
\begin{equation}
    \label{torsor-structure-r}
\pi_1^{\textrm{rel}}(X,x,y) \times \pi_1^{\textrm{rel}}(X,x) \rightarrow \pi_1^{\textrm{rel}}(X,x,y),
\end{equation}
and a left torsor over $\pi_1^{\textrm{rel}}(X,y)$:
\begin{equation}
    \label{torsor-structure-l}
\pi_1^{\textrm{rel}}(X,y) \times \pi_1^{\textrm{rel}}(X,x,y)  \rightarrow \pi_1^{\textrm{rel}}(X,x,y).
\end{equation}
\end{rmk}
\begin{rmk}
\label{families-rel-compl}
We may replace the point $\tau$ and its associated fiber functors $\omega_{\Beta}^{\tau},\omega_{\dR}^{\tau}$ with $\omega_{\Beta}^X,\omega_{\dR}^Y$,  for any $X \subset X_{\Gamma}(\CC)$ simply connected and $Y\subset X_{\Gamma}(\CC)$ such that $Y \subset U(\CC)$ for $U\subset X_{\Gamma}$ affine, sending an element $(\mathbb{V}_{\Beta},\mathcal{V}_{\dR},c)$ in $ \mathcal{H}(X_{\Gamma})$ to the sections of $\mathbb{V}_{\Beta}$ over $X$ and $\mathcal{V}_{\dR}$ over $Y$ respectively -- see \cite[\S 7.2.1]{Brown2} for details. We obtain a local system and vector bundle 
$$
\mathcal{O}(\pi_1^{B,\mathcal{S}}(X_{\Gamma},\vec{1}_{i\infty},\bullet)), \quad \mathcal{O}(\pi_1^{dR,\mathcal{S}}(X_{\Gamma},\vec{1}_{i\infty},\bullet)) \otimes_{k} \mathcal{O}_{X_{\Gamma}}.
$$
 The Betti relative completion $\mathcal{O}(\pi_1^{B,\mathcal{S}}(X_{\Gamma},\vec{1}_{i\infty},\bullet))$ is an (admissible) variation of mixed Hodge structures on $X_{\Gamma}$ \cite[Theorem 7.18]{Hain2}. Hence we get an object in the category $\mathcal{H}(X_{\Gamma})$:
$$
 (\mathcal{O}(\pi_1^{B,\mathcal{S}}(X_{\Gamma},\vec{1}_{i\infty},\bullet)), \mathcal{O}(\pi_1^{dR,\mathcal{S}}(X_{\Gamma},\vec{1}_{i\infty},\bullet)) \otimes_{k} \mathcal{O}_{X_{\Gamma}}, c),
$$
denoted $\mathcal{O}(\pi_1^{\textrm{rel}}(X_{\Gamma},\vec{1}_{i\infty},\bullet))$.
We will, however, work over a fiber at $\tau$ in what follows for simplicity, and note that all results lift to families over $X_{\Gamma}$ by replacing the fiber functors at $\tau$ with $\omega^Y_{\bullet}$, where $\bullet = \Beta,\dR$, everywhere.
\end{rmk}
\end{subsection}
\begin{subsection}{Motivic periods}
\label{motivic periods of rel compl}
Our next goal is to undestand the affine rings of the relative de Rham and Betti fundamental groupoids of $X_{\Gamma}$ and relate the associated periods, i.e. matrix coefficients of their comparison, to iterated integrals of modular forms.
\subsubsection{Splittings}
Let $\tau \in X_{\Gamma}(k)$, and consider the Tannaka group of the category $\mathcal{H}(k)$ with respect to the fiber functor $\omega^{\tau}_{\dR}$, which we denote by $G^{\dR}_{\mathcal{H}(k)}$. Its action on $\QQ(-1)$ defines a character $\chi : G^{\dR}_{\mathcal{H}(k)} \rightarrow \mathbb{G}_m$ sending $g \in G^{\dR}_{\mathcal{H}(k)}(R)$ to $\lambda_g \in R^{\times}$, for any commutative ring $R$, where
$$
g\mathbb{L}^{\dR} = (1\otimes g) \Delta \mathbb{L}^{\dR} = \lambda_g\mathbb{L}^{\dR}.
$$
We get an exact sequence
$$
1 \rightarrow G^{\dR,ker(\chi)}_{\mathcal{H}(k)} \rightarrow G^{\dR}_{\mathcal{H}(k)} \xrightarrow{\chi} \mathbb{G}_m \rightarrow 1,
$$
a splitting of which gives a splitting of the $W$-filtration for all objects in $\mathcal{H}(k)$. Since $\mathcal{O}(\pi_1^{\textrm{rel}}(X_{\Gamma},\vec{1}_{i\infty},\tau))$ is an object of the category $\mathcal{H}(k)$ we get a splitting of the weight filtration on its de Rham realization. In turn we get a "weight 0 projection" morphism in the de Rham realization:
$$
gr^{W,\dR}_0 : \mathcal{O}(\pi_1^{\dR,\mathcal{S}}(X_{\Gamma},\vec{1}_{i\infty},\tau)) \rightarrow k.
$$
Thus, choosing such a splitting is equivalent to choosing a "de Rham path" which we denote $1^{\dR}_{\tau} \in \pi_1^{\textrm{dR},\mathcal{S}}(X_{\Gamma},\vec{1}_{i\infty},\tau)(k)$, and which in turn induces an isomorphism between the de Rham realizations of the relative torsor of paths and the relative fundamental group via the torsor structure of the former over the latter \eqref{torsor-structure-r}:
\begin{equation}
\begin{split}
\pi_1^{\textrm{dR},\mathcal{S}}(X_{\Gamma},\vec{1}_{i\infty}) & \xrightarrow{\sim} \pi_1^{\textrm{dR},\mathcal{S}}(X_{\Gamma},\vec{1}_{i\infty},\tau) \\
g &\mapsto 1^{\dR}_{\tau} \cdot g.
\end{split}
\end{equation}
Similarly we have
\begin{equation}
\label{torsor-group-iso}
\begin{split}
\pi_1^{\textrm{dR},\mathcal{S}}(X_{\Gamma},\tau) & \xrightarrow{\sim} \pi_1^{\textrm{dR},\mathcal{S}}(X_{\Gamma},\vec{1}_{i\infty},\tau) \\
g &\mapsto  g \cdot (1^{\dR}_{\tau})^{-1}.
\end{split}
\end{equation}
We also need to choose a splitting of the extension
\begin{equation}
\label{splitting-relative-group}
0 \rightarrow U^{\dR}_{\tau} \rightarrow \pi_1^{\dR,\mathcal{S}}(X_{\Gamma},\tau) \rightarrow S_{\tau}^{\dR} \rightarrow 0.
\end{equation}
Choosing such splittings is possible by the argument in \cite[\S13.9]{Brown5} and \cite[Proposition 3.1]{Hain2}. Putting this together we get an isomorphism of affine rings
$$
\mathcal{O}(\pi_1^{\textrm{dR},\mathcal{S}}(X_{\Gamma},\vec{1}_{i\infty},\tau)) \cong \mathcal{O}(U_{\tau}^{\dR}) \otimes \mathcal{O}(S_{\tau}^{\dR}).
$$
\subsubsection{The unipotent and reductive paths} Let $\gamma_{\tau} \in \pi_1(X_{\gamma},\vec{1}_{i\infty},\tau)$ be the class of a path between $\vec{1}_{i\infty}$  and $\tau$, and $\gamma^{\Beta}_{\tau}$ be its image under \eqref{Betti-path}. This gives a homomorphism
$$\gamma^{\Beta}_{\tau} : \mathcal{O}(\pi_1^{\Beta,\mathcal{S}}(X_{\Gamma},\vec{1}_{i\infty},\tau)) \rightarrow k,$$ 
and we can extend scalars to $\CC$. Pre-composing $\gamma^{\Beta}_{\tau}$ with the morphism of affine rings associated to the comparison isomorphism \eqref{grp-comparison} we obtain an element 
$$
\gamma^{\Beta,\dR}_{\tau} \in \pi_1^{\dR,\mathcal{S}}(X_{\Gamma},\vec{1}_{i\infty},\tau)(\CC) = \Hom(\mathcal{O}(\pi_1^{\dR,\mathcal{S}}(X_{\Gamma},\vec{1}_{i\infty},\tau)),\CC).
$$
Now consider the composition of the (inverse of) the isomorphism \eqref{torsor-group-iso} between the relative de Rham torsor of paths and the relative de Rham group with the base point at $\tau$ and the morphism $\pi_1^{\dR,\mathcal{S}}(X_{\Gamma},\tau) \rightarrow U^{\dR}_{\tau}$ which splits \eqref{splitting-relative-group}. The homomorphism of affine rings associated to this composition is a homomorphism $\mathcal{O}(U^{\dR}_{\tau}) \rightarrow \mathcal{O}(\pi_1^{\dR,\mathcal{S}}(X_{\Gamma},\vec{1}_{i\infty},\tau))$. We can pre-compose $\gamma^{\Beta,\dR}_{\tau}$ by it and we denote the resulting homomorphism by
\begin{equation}
    \label{unipotent-path}
    \gamma_{\tau}^{\Beta,u} : \mathcal{O}(U^{\dR}_{\tau}) \rightarrow \CC
\end{equation}
We may think of $\gamma_{\tau}^{\Beta,u}$ as a formal power series the coefficients of which are periods. Choosing an element $\omega \in \mathcal{O}(U^{\dR}_{\tau})$ amounts to selecting a coefficient of this series. In this way $\gamma_{\tau}^{\Beta,u}$ can be regarded as the modular version of the Drinfeld associator (see \cite{Brown6}). 

Moreover, we shall need to consider the homomorphism obtained by pre-composing $\gamma^{\Beta,\dR}_{\tau}$ by the homomorphism of affine rings associated to $\pi_1^{\dR,\mathcal{S}}(X_{\Gamma},\vec{1}_{i\infty}, \tau) \rightarrow S_{\vec{1}_{i\infty},\tau}^{\dR}$ defined in \S\ref{tannakian definition}. Denote this composition by 
\begin{equation}
\label{reductive-path}
\gamma^{red}_{\vec{1}_{i\infty},\tau} : \mathcal{O}(S^{\dR}_{\vec{1}_{i\infty},\tau}) \rightarrow \CC.
\end{equation}

\subsubsection{$U^{\dR}_{\tau}$ and modular forms.} Next, we describe the affine ring of the unipotent part of the relative de Rham fundamental group $\mathcal{O}(U^{\dR}_{\tau})$. We consider the Tannakian category $\mathcal{W}(\mathcal{L}_X, \mathcal{S}^B)$ of local systems on $X$ equipped with a filtration with prescribed graded pieces. It has a functor to the category of local systems on $X$ which forgets the filtration. This induces a morphism of Ext groups
$$
\Ext^r_{\mathcal{W}(\mathcal{L}_{X_{\Gamma}}, \mathcal{S}^B)}(\QQ,\mathbb{V}) \rightarrow \Ext^r_{\mathcal{L}_{X_{\Gamma}}}(\QQ,\mathbb{V}),
$$
where $\mathbb{V} \in \mathcal{W}(\mathcal{L}_{X_{\Gamma}}, \mathcal{S}^B)$. By definition $\Ext^r_{\mathcal{L}_{X_{\Gamma}}}(\QQ,\mathbb{V}) = H^r(\Gamma;V)$ where $V$ is the $\Gamma$-representation associated to $\mathbb{V}$, and on the left hand side we have 
$$
\Ext^r_{\mathcal{W}(\mathcal{L}_{X_{\Gamma}}, \mathcal{S}^B)}(\QQ,\mathbb{V}) = H^r(\pi_1^{\Beta,\mathcal{S}}(X_{\Gamma},\tau); V).
$$ 
This map is an isomorphism for $r=1$, and both sides are trivial for $r \geq 2$, given that $\Gamma$ is a modular group \cite[\S 3.2, \S3.4.2]{Hain2}. It follows that $$\textrm{gr}^C\mathcal{O}(U^{\dR}_{\tau}) \cong  T^c(H^1(U^{\dR}_{\tau})),$$ 
where $T^c$ stands for the tensor coalgebra, and $\textrm{gr}^C$ stands for the graded for the `coradical filtration', otherwise known as the `filtration by unipotency degree' -- see \cite[\S2.5]{Brown2}. Moreover, from \cite[\S 6]{Brown1} we have that
\begin{equation}
\label{H1Udr}
H^1(U^{\dR}_{\tau}) = \bigoplus_{n\geq 0}\Ext^1_{\mathcal{W}(\mathcal{L}_{X_{\Gamma}}, \mathcal{S}^B)}(\mathbb{Q},\mathbb{V}_n^{\vee}) \otimes \mathcal{V}_{n,\tau},
\end{equation}
where $\mathbb{V}_n$ is the $n$th symmetric power of the local system $R^1f_*\QQ$, where $f: \mathcal{E} \rightarrow X_{\Gamma}$ is the universal family of elliptic curves, and $\mathcal{V}_{n,\tau}$ is the fiber at $\tau$ of the nth symmetric power of the relative algebraic de Rham cohomology of $\mathcal{E}$. Note that $\mathbb{V}_n$ is self-dual induced by $H^1(\mathcal{E}_{\tau}; \QQ)^{\vee} \cong H^1(\mathcal{E}_{\tau}; \QQ)(1)$. Since $\mathcal{W}(\mathcal{L}_{X_{\Gamma}}, \mathcal{S}^B)$ is a subcategory of local systems on $S(\CC)$ we have
$$
\Ext^1_{\mathcal{W}(\mathcal{L}_{X_{\Gamma}}, \mathcal{S}^B)}(\mathbb{Q},\mathbb{V}_n^{\vee}) \cong H^1(\Gamma, V_n^{\vee}) \cong H^1(X_{\Gamma}, \mathbb{V}_n^{\vee}),
$$ where in the middle we have group cohomology where $V_n^{\vee}$ denotes the fiber of $\mathbb{V}_n^{\vee}$ at $\partial/\partial q$, and on the right hand side we have cohomology with coefficients in the local system $\mathbb{V}_n^{\vee}$.

\subsubsection{Eichler-Shimura.} Grothendieck's algebraic de Rham theorem gives an isomorphism between $H^1(X_{\Gamma}, \mathbb{V}_n)$ and $M^!_{n+2}/\mathcal{D}^{n+1}M_{-n}^! \otimes \CC$ where  $M^!$ is the space of weakly holomorphic modular forms, and $\mathcal{D} = qdq/q$ (see \cite[Corollary 1.4]{BrownHain} for the level 1 case). When restricted to the space of holomorphic modular forms of weight $n+2$ we recover the Eichler-Shimura isomorphism. The Hodge structure  of $H^1(X_{\Gamma},\mathbb{V}_n)$ is known \cite{Hain2}. We have
\begin{equation}
\label{Hodge-structure-mod-forms}
H^1(X_{\Gamma},\mathbb{V}_n) \cong \bigoplus_f M_f \oplus \bigoplus_{e} \QQ(-n-1),
\end{equation}
where $f$ ranges over all cusp form of weight $n+2$ and $e$ ranges over all Eisenstein series of weight $n+2$ for the congruence group $\Gamma$. Here $M_f$ is the Hodge realization of the motive of the cusp form $f$ \cite{Scholl}. Note that the pullback of $\mathbb{V}_n$ to the upper half-plane $\mathfrak{H}$ via $\rho : \mathfrak{H} \rightarrow X_{\Gamma}$ is the trivial local system whose fiber over $z \in \mathfrak{H}$ is $H^1(E_z)$, where $E_z = \CC/(\ZZ \oplus z\ZZ)$. Denote by $X, Y$ the sections of $\mathbb{V}_1^{\vee}$ which over a point $z$ correspond to the standard basis of $H_1(E_z)$ given by classes of paths from $0$ to $1$ and $0$ to $z$, respectively. Then we have that
\begin{alignat*}{2}
&M_{n+2}(\Gamma) &&\cong F^{n+1}H^1(X_{\Gamma},\mathbb{V}_n) \\
&\quad \quad f(z) && \mapsto (2\pi i)^{n+1} f(z)( zX - Y)^n dz\, ,
\end{alignat*}
where $M_{n+2}(\Gamma)$ is the space of holomorphic modular forms of weight $n+2$ for the congruence group $\Gamma$ -- see \cite[Theorem 11.4]{Hain2}, \cite[Appendix A]{BrownHain}.

\subsubsection{Motivic periods of the unipotent part}
We write $\eta \in \mathcal{O}(U_{\tau}^{\dR})$ using the bar notation $\eta = [f_1| \cdots |f_k] \otimes [\xi_1|
\cdots|\xi_k]$, where $f_i$ is such that $(2\pi i)^{n+1}f_i(z)( zX - Y)^ndz$ corresponds to a class in $H^1(X_{\Gamma},\mathbb{V}_n)$, and $\xi_i \in \mathcal{V}_{n,\tau}$. We associate to it a motivic period of $\mathcal{O}(\pi_1^{\textrm{rel}}(X_{\Gamma},\vec{1}_{i\infty},\tau))$:
\begin{equation}
\label{unipotent-motivic-period}
 \gamma_{\tau,u}^{\fm}(\eta) = \left[\mathcal{O}(\pi_1^{\textrm{rel}}(X_{\Gamma},\vec{1}_{i\infty},\tau)), \gamma^{\Beta}_{\tau},\eta \right]^{\fm},
\end{equation}
where $\gamma^{\Beta}_{\tau}$ is viewed as an element of $\mathcal{O}(\pi_1^{\textrm{rel},\Beta}(X_{\Gamma},\vec{1}_{i\infty},\tau))^{\vee}$. Its image under the period isomorphism is given by the iterated Eichler integral of $f_1,\ldots,f_k$ along $\gamma$ from $i\infty$ to $\tau$, followed by a pairing of sections $X,Y$ and the appropriate element of $\mathcal{V}_{n,\tau}$ which is given by integration. 
\begin{ex}
Let $\eta = [f(z)]\otimes[\omega_{\tau}]$, where $f(z)$ is modular form of weight 3 over $X_1(6)$ and $[\omega_{\tau}] \in F^1H^1_{\dR}(\mathcal{E}_{\tau})$. Then we have
\begin{equation}
\begin{split}
\textrm{per}(\gamma^{\fm}_{\tau,u}(\eta)) &= (2\pi i)^2 \int_{\tau}^{i\infty} f(z)( zX - Y) dz \otimes [\omega_{\tau}] \\
& = (2\pi i)^2 \int_{\tau}^{i\infty} f(z)(z - \tau) dz,
\end{split}
\end{equation}
 Note that in the case when $f(z)$ is an Eisenstein series we have to regularize at $i\infty$ (see \cite[\S 4]{Brown5}).
\end{ex}

\subsubsection{The reductive part} Recall that $S^{\dR}_{i\infty,\tau} = \textrm{Isom}^{\otimes}_{\mathcal{S}^{\dR}}(\omega^{\dR}_{i\infty},\omega^{\dR}_{\tau})$. Since $\mathcal{S}^{\dR}$ is semi-simple its affine ring is 
$$
\mathcal{O}(S^{\dR}_{i\infty,\tau}) \cong \bigoplus_{n \geq 0} \mathcal{V}_{n,\tau} \otimes \mathcal{V}_{n,i\infty}^{\vee}.
$$
Note that $\mathcal{V}_{n,i\infty}$ is the $n$th symmetric power of the algebraic de Rham cohomology of the infinitesimal Tate elliptic curve $\mathcal{E}_{i\infty}^{\times}$. It's split Tate and we have 
$$(H^1(\mathcal{E}_{i\infty}^{\times}))^{\vee} \cong \QQ(0) \oplus \QQ(1).
$$
Denote the de Rham generators by $\drX,\drY$, where  $\drX$ spans the $\QQ(0)$ and $\drY$ the $\QQ(1)$ piece. Similarly to de Rham periods, $\mathcal{O}(S^{\dR}_{i\infty,\tau})$ is spanned, over $k$, by classes of triples, e.g.
$$
[ \mathcal{O}(S^{\dR}_{i\infty,\tau}), \nu, \mu]^{red},
$$
where $\nu$ is a word in $\{\drX,\drY\}$, and $\mu$ is a word, of the same length as $\nu$, in $\{\omega,\eta\}$ where these form a basis of $H^1_{\dR}(\mathcal{E}_{\tau})$. Recall that a path $\gamma_{i\infty,\tau}$ from $i\infty$ to $\tau$ gives a homomorphism which assigns to each such triple a complex number \eqref{reductive-path}. Let us look more closely at how it acts. Consider the map
\begin{equation}
\label{reductive-pairing}
\mathcal{V}_{n,\tau} \otimes \CC \xrightarrow{\sim} Sym^nH^1_{\Beta}(\mathcal{E}_{\tau};\QQ) \otimes \CC \xrightarrow{\sim} Sym^nH^1_{\Beta}(\mathcal{E}_{i \infty};\QQ) \otimes \CC \xrightarrow{\sim} \mathcal{V}_{n,i\infty} \otimes \CC,
\end{equation}
where the first arrow is induced by the comparison isomorphism 
for algebraic de Rham and Betti cohomology $H^1_{\dR}(\mathcal{E}_{\tau}) \otimes \CC \cong H^1_{\Beta}(\mathcal{E}_{\tau};\QQ) \otimes \CC$, the middle arrow is an isomorphism induced by the pullback along $\gamma_{i\infty,\tau}$, and the last arrow is the inverse of the comparison isomorphism for the fiber at $i\infty$. Denote the composition by $\gamma_{i\infty,\tau}^{comp}$. It's matrix representation is $P_{i\infty}^{-1}P_{\gamma_{i\infty,\tau}^*}P_{\tau}$ where $P_{\tau}$ is period matrix of $\mathcal{E}_{\tau}$, $P_{i\infty}^{-1}$ is the inverse of the matrix of period matrix of the infinitesimal Tate curve
$$
P_{i\infty}^{-1} = \begin{pmatrix}
1 & 0 \\
0 & (2\pi i)^{-1}
\end{pmatrix} \, ,
$$
and $P_{\gamma_{i\infty,\tau}^*}$ is the matrix corresponding to the middle isomorphism above.
\begin{ex}
Let $[\omega_{\tau}]$ be as in the previous example. Then we have
$$
\gamma^{red}_{\vec{1}_{i\infty},\tau}([ \mathcal{O}(S^{\dR}_{i\infty,\tau}), \drY, [\omega_{\tau}]]^{red}) = \frac{ \omega_1}{2\pi i}.
$$
Similarly,
$$
\gamma^{red}_{\vec{1}_{i\infty},\tau}([ \mathcal{O}(S^{\dR}_{i\infty,\tau}), \drX, [\omega_{\tau}]]^{red}) = \omega_2,
$$
where $\omega_1,\omega_2$ are the two periods of the elliptic curve $\mathcal{E}_{\tau}$.
\end{ex}

We also note that the ring $\mathcal{O}(S^{\dR}_{i\infty,\tau})$ is an object of the underlying category $\mathcal{H}(k)$, and hence is equipped with Hodge and weight filtrations. In what follows we shall drop the affine ring from the notation $[ \mathcal{O}(S^{\dR}_{i\infty,\tau}), \nu, \mu]^{red}$ and simply write $[\nu, \mu]^{red}$ where it is unambiguous what it refers to.
\end{subsection}
\begin{subsection}{Motivic periods which can be expressed in terms of motivic Eichler integrals}
Let $f: \mathcal{E} \rightarrow X_{\Gamma}$, $\mathbb{V}_n$ and  $\mathcal{V}_n$ be as above -- the universal family over $X_{\Gamma}$, the nth symmetric power of the relative Betti cohomology of $\mathcal{E}$ and the nth symmetric power of the relative algebraic de Rham cohomology of $\mathcal{E}$ respectively. For $\mathcal{V} = (\mathbb{V}_{\Beta},\mathcal{V}_{\dR},c) \in \mathcal{H}(X_{\Gamma})$ define $\mathcal{P}^{\fm}_{\mathcal{V}_x}$ to be the span over $k$ of $[\mathcal{V},\sigma,\omega]^{\fm}$ where $\sigma \in \mathbb{V}_{\Beta,x}^{\vee}$ and $\omega \in \mathcal{V}_{\dR,x}$, i.e. the periods of the `fiber at $x$'.
    \begin{thm}
    \label{motivic theorem}
     Let $\mathcal{V} \in \mathcal{H}(X_{\Gamma})$ be such that the local system $\mathbb{V}_{\Beta}$ has a filtration
    $$
    0 \subset \mathbb{V}_{\Beta}^1 \subset \mathbb{V}_{\Beta}^2 \subset \cdots \subset \mathbb{V}_{\Beta}^k = \mathbb{V}_{\Beta}, 
    $$
    such that 
    $$
    \mathbb{V}_{\Beta}^i/\mathbb{V}_{\Beta}^{i-1} \in \mathbb{V}_n(r),
    $$
    where $n,r \in \ZZ$ for all $0<i\leq n$. Furthermore we assume that $\mathcal{V}_{\dR}$ has an analogous filtration such that $\mathcal{V}_{\dR}^i/\mathcal{V}_{\dR}^{i-1} \in \mathcal{V}_n$. Then
    $$
        \mathcal{P}^{\fm}_{\mathcal{V}_y} \subset \mathcal{P}^{\fm}_{\mathcal{V}_x} \otimes \mathcal{O}(\pi_1^{\textrm{rel}}(X_{\Gamma},x,y)).
    $$
\end{thm}
\begin{proof}
     For a Tannakian category $\mathcal{C}$ with fiber functors $\omega_x,\omega_y$ a $k$-point of $ \textrm{Isom}^{\otimes}_{\mathcal{C}}(\omega_x,\omega_y)$ gives an isomorphism $\omega_x(M) \otimes k \xrightarrow{\sim} \omega_y(M) \otimes k$, for any $M \in \textrm{Ob}(\mathcal{C})$. By assumption of the theorem we have that $\mathbb{V}_{\Beta} \in \mathcal{W}(\mathcal{L}_{X_{\Gamma},S^{\Beta}})$, as defined in \ref{tannakian definition}. Hence each $g \in \pi_1^{\textrm{rel},\Beta}(X_{\Gamma},x,y)(k)$ induces an isomorphism $\mathbb{V}_{\Beta,x} \xrightarrow[g]{\sim} \mathbb{V}_{\Beta,y}$. Similarly, $\mathcal{V}_{\dR} \in \mathcal{W}(\mathcal{A}_{X_{\Gamma},S^{\dR}})$, and we have for each element of $\pi_1^{\textrm{rel},\dR}(X_{\Gamma},x,y)(k)$ an isomorphism of $\mathcal{V}_{\dR,x}$ and $\mathcal{V}_{\dR,y}$. Putting this together we get an action
    $$
    \pi_1^{\textrm{rel}}(X_{\Gamma},x,y) \times (\mathcal{V}_{x}) \rightarrow (\mathcal{V}_{y}).
    $$
    Dual to the action above we have the coaction
    $$
    \rho : \mathcal{V}_{y} \rightarrow \mathcal{V}_{x} \otimes_k \mathcal{O}(\pi_1^{\textrm{rel}}(X_{\Gamma},x,y)).
    $$
    This is a morphism of mixed Hodge structures \cite[Theorem 8.1]{Hain2}, hence a morphism in $\mathcal{H}(k)$. 
    
    Let $[\mathcal{V}_y,\sigma_y,\omega_y]^{\fm} \in \mathcal{P}^{\fm}_{\mathcal{V}_y}$. The topological fundamental group acts on the dual of the Betti realization of $\mathcal{V}$, i.e. we have
    $$
    \pi_1(X_{\Gamma},x,y) \times \mathbb{V}_{\Beta,x}^{\vee} \rightarrow \mathbb{V}_{\Beta,y}^{\vee}.
    $$
    In particular the class of a path $\gamma$ from $x$ to $y$ sends $[\sigma_x]$ to $[\sigma_{y}]$, where $[\sigma_{x}]$ is the class obtained by continuing $[\sigma_y]$ along the path $\gamma^{-1}$. The image of this relation under \eqref{Betti-path} is
    $$
    \rho_{\Beta}^{\vee}\left([\gamma^{\Beta}] \otimes [\sigma_x]\right) = [\sigma_y].
    $$
    Let $\omega_x \otimes \xi \in \mathcal{V}_{\dR,x} \otimes_k \mathcal{O}(\pi_1^{\dR,\mathcal{S}}(X_{\Gamma},x,y))$ be such that $\rho_{\dR}(\omega_y) = \omega_x \otimes \xi$. It follows that
    $$
    [\mathcal{V}_y,\sigma_y,\omega_y]^{\fm} = [\mathcal{V}_x,\sigma_x,\omega_x]^{\fm} \otimes [\mathcal{O}(\pi_1^{\textrm{rel}}(X_{\Gamma},x,y)),\gamma^{\Beta},\xi]^{\fm}.
    $$

\end{proof}
\end{subsection}
\begin{subsection}{The sunrise and the relative completion of $\Gamma_1(6)$}
The previous theorem implies that periods of an object in $\mathcal{H}(X_1(6))$ satisfying the assumptions can be expressed as products of iterated Eichler integrals from $x$ to $y$ and periods of the special fiber at $x$. We will now see that the sunrise can be easily related to such an object. It is known that in the equal-mass case the elliptic curve given by the vanishing locus of the second Symanzik polynomial of the sunrise is the universal family $\mathcal{E} \rightarrow X_1(6)$ -- see \cite{BV}.
\begin{cor}
 The motivic Feynman amplitude of the sunrise in the equal-mass case is a motivic period of an object $E \in \mathcal{H}(X_1(6))$ satisfying the conditions of the previous theorem. The sunrise Feynman integral can therefore be written as a $k$-linear combination of products of Eichler integrals and periods of a special fiber of $E$.
\end{cor}
\begin{proof}
Let $t = \frac{q_1^2}{m^2}$, and define $J_G^{\fm}(t) = m^2I^m_G(m,q)$. We may choose a basis of $H^1_{\dR}(\mathcal{E})$, and a basis of $H_1(\mathcal{E};\QQ)$. This fixes a choice of periods of the elliptic curve $\mathcal{E}$ denoted $\omega_1,\omega_2$. We can then write $t$ as a modular function $g(\tau)$, where $\tau = \frac{\omega_1}{\omega_2}$. Define $J_G^{\fm}(\tau) = [mot_G(\tau),\sigma_{G,\tau},\omega_{G,\tau}]^{\fm} = g^*J_G^{\fm}(t)$ to be the pull-back of the motivic period $J_G^{\fm}$ over $X_1(6)$. We can pull back the diagram \eqref{diagram-equal-masses} via $g$ to $X_1(6)$, and we denote $i : H^1(\mathcal{E})(-1)_{X_1(6)} \rightarrow H^2(P^G \setminus \mathcal{E})_{X_1(6)}$ the morphism which splits the residue. We can pull back $mot_G(\tau)$ via $i$ to an extension
\begin{equation}
\label{extension-E}
0 \rightarrow \QQ(0)_{/X_1(6)} \rightarrow E_{/X_1(6)} \rightarrow H^1(\mathcal{E})(-1)_{/X_1(6)} \rightarrow 0.
\end{equation}
We have $J_G^{\fm}(\tau) = [E_{/X_1(6)}, [\sigma_{G,\tau}],[res(\omega_{G,\tau})]]^{\fm}$. Furthermore, the object $E_{/X_1(6)} \in \textrm{Ob}(\mathcal{H}(X_1(6)))$ is equipped with a filtration
$$
0 \subset \QQ(0)_{/X_1(6)}\subset E_{/X_1(6)}
$$
such that its graded quotient is $H^1(\mathcal{E})(-1)_{/X_1(6)}$. Hence 
$$
(E_{/X_1(6)})_{\Beta} \in \textrm{Ob}(\mathcal{W}(\mathcal{L}_{X_1(6)},S^B)), \quad (E_{/X_1(6)})_{\dR} \in \textrm{Ob}(\mathcal{W}(\mathcal{A}_{X_1(6)},S^{\dR})),
$$
as required, and the previous theorem applies.
\end{proof}
\end{subsection}

\begin{rmk}
To sharpen this result one could use the Hodge and weight filtrations, i.e. the fact that $[\pi^*_G\omega_G] \in F^2W_3(mot_G)_{\dR}$, and the fact that the above morphism $\rho$ respects them as it is a morphism in $\mathcal{H}(X_1(6))$. The mixed Hodge theory of $\mathcal{O}(\pi_1^{\textrm{rel}}(X_{\Gamma},x,y))$ is known by \cite{Hain2}, and one could use it to set up an ansatz for $J^m(\tau)$ in terms of Eichler integrals, providing an alternative strategy to writing the sunrise in this form to those given in \cite{BV,AW}. Note that the method here would apply to any Feynman graph the motive of which satisfies the assumptions of theorem \ref{motivic theorem}.
\end{rmk}

\appendix
\begin{section}{Defining single-valued periods over a base}
\label{single-valued-periods definition}
We enrich the category $\mathcal{H}(S)$ with an isomorphism of local systems
$$
F_{\infty} : \mathbb{V}_{\textrm{B}} \xrightarrow{\sim} \sigma^{*}\mathbb{V}_{\textrm{B}}
$$
where $\sigma:S(\CC) \xrightarrow{\sim} S(\CC)$ is induced by complex conjugation, such that the following diagram commutes 
\begin{equation}
\xymatrix{
            \mathcal{V}_{\dR} \otimes_{\mathcal{O}_S} \mathcal{O}_{S^{an}} \ar[d]^{id\otimes \rho} \ar[r]^{c}& \mathbb{V}_{\textrm{B}} \otimes_{\mathbb{Q}} \mathcal{O}_{S^{an}} \ar[d]^{F_{\infty}\otimes \rho}\\
            \mathcal{V}_{\dR} \otimes_{\mathcal{O}_{S}} \mathcal{O}_{\overline{S}^{an}} \ar[r]^{\overline{c}} & \sigma^*\mathbb{V}_{\textrm{B}} \otimes_{\mathbb{Q}} \mathcal{O}_{\overline{S}^{an}}
            }
\end{equation}
where $\mathcal{O}_{\overline{S}}$ is the sheaf of antiholomorphic functions on $S^{an}$, $\rho: f \mapsto \overline{f}$, and $\overline{c}$ is the pullback of the comparison isomorphism by $\sigma$.

The morphism $F_{\infty}$ induces an  isomorphism of rings of motivic periods $$F_{\infty} : \mathcal{P}^{\fm,X,Y}_{\mathcal{H}(S)} \xrightarrow{\sim}  \mathcal{P}^{\fm,\overline{X},Y}_{\mathcal{H}(\overline{S})},$$ where $\mathcal{H}(\overline{S})$ is the category defined in the same way as $\mathcal{H}(S)$ except for the comparison isomorphism
$$
\overline{c} : \mathcal{V}_{\dR} \otimes_{\mathcal{O}_{S}} \mathcal{O}_{\overline{S}^{an}} \xrightarrow{\sim} \mathbb{V}_{\textrm{B}} \otimes_{\mathbb{Q}} \mathcal{O}_{\overline{S}^{an}}.
$$
The composition of morphisms $\sigma^*F_{\infty}$ defines an isomorphism  $$\sigma^*F_{\infty} : \mathcal{P}^{\fm,X,Y}_{\mathcal{H}(S)} \xrightarrow{\sim}  \mathcal{P}^{\fm,X,Y}_{\mathcal{H}(\overline{S})},$$
which sends
$$
\left[\mathcal{V},\gamma,\omega \right] \mapsto \left[\overline{\mathcal{V}},\gamma\circ\sigma^*F_{\infty},\omega \right],
$$
where $\overline{\mathcal{V}} = \left(\sigma^*\mathbb{V}_{\textrm{B}},\mathcal{V}_{\dR},\overline{c} \right)$. Recall that the period map on $\mathcal{P}^{\fm,X,Y}_{\mathcal{H}(S)}$ takes values in $M_{X,Y}(S(\CC))$ which is the ring of multivalued meromorphic functions on $S(\CC)$ with a prescribed branch on $X$. Similarly the period map on $\mathcal{P}^{\fm,X,Y}_{\mathcal{H}(\overline{S})}$ takes values in $\overline{M}_{X,Y}(S(\CC))$ which is the ring of quotients of antiholomorphic functions on $S(\CC)$. We have $\textrm{per}(\sigma^*F_{\infty}\xi) = \overline{\textrm{per}(\xi)}$, where $\xi$ is a motivic period over $S$.

Denote by $\mathcal{P}$ the ring $\mathcal{P}^{\fm,X,Y}_{\mathcal{H}(S)} \otimes_{\mathbb{Q}} \mathcal{P}^{\fm,X,Y}_{\mathcal{H}(\overline{S})}$, and by $f$ the morphism $\sigma^*F_{\infty}$ composed with $\mathcal{P}^{\fm,X,Y}_{\mathcal{H}(\overline{S})} \rightarrow \mathcal{P}$ sending $x \mapsto 1 \otimes x$. Let $\iota$ be the morphism $\mathcal{P}^{\fm,X,Y}_{\mathcal{H}(S)} \rightarrow \mathcal{P}$ sending $x \mapsto x \otimes 1$. We have two $\mathcal{P}$-points
$$
\iota, f \in \textrm{Hom}\left(\mathcal{P}^{\fm,X,Y}_{\mathcal{H}(S)},\mathcal{P}\right) = \textrm{Isom}^{\otimes}_{\mathcal{H}(S)}(\omega_{\dR}^{Y},\omega_{\textrm{B}}^X)(\mathcal{P}),
$$
of the scheme $\textrm{Isom}^{\otimes}_{\mathcal{H}(S)}(\omega_{\dR}^{Y},\omega_{\textrm{B}}^X)$, which is a torsor over $\textrm{Aut}^{\otimes}_{\mathcal{H}(S)}(\omega_{\dR}^{Y})$, hence we have
$$
\textrm{Isom}^{\otimes}_{\mathcal{H}(S)}(\omega_{\dR}^{Y},\omega_{\textrm{B}}^X)(\mathcal{P}) \times \textrm{Aut}^{\otimes}_{\mathcal{H}(S)}(\omega_{\dR}^{Y})(\mathcal{P}) \rightarrow \textrm{Isom}^{\otimes}_{\mathcal{H}(S)}(\omega_{\dR}^{Y},\omega_{\textrm{B}}^X)(\mathcal{P}).
$$
Therefore there is a unique $\mathcal{P}$-point of $G^{\mathfrak{dr},Y}_{\mathcal{H}(S)}$, i.e. a morphism $\mathcal{P}^{\fdr,Y}_{\mathcal{H}(S)} \rightarrow \mathcal{P}$, which we denote by $\textrm{s}^{\fm}$ and which satisfies
$$
f \circ \textrm{s}^{\fm} = \iota.
$$
We call it the \textit{single-valued morphism}.

The morphism $\textrm{s}^{\fm}$ is computed for each object $\mathcal{V} \in \textrm{Ob}(\mathcal{H}(S))$ by the composition
$$
\omega_{\dR}^{Y}(\mathcal{V}) \otimes \mathcal{P} \xrightarrow{c^{\fm}_{\mathcal{V}}} \omega_{\textrm{B}}^{X}(\mathcal{V}) \otimes \mathcal{P} \xrightarrow{f^{\infty}_{\mathcal{V}}}\omega_{\textrm{B}}^{X}(\mathcal{V}) \otimes \mathcal{P} \xrightarrow{(c^{\fm}_{\mathcal{V}})^{-1}} \omega_{\dR}^{Y}(\mathcal{V}) \otimes \mathcal{P},
$$
where $c^{\fm}_{\mathcal{V}}$ is the extension of scalars of the comparison isomorphism $c$ to $\mathcal{P}$, and the morpshim $f^{\infty}_{\mathcal{V}}$ sends $x \otimes y \otimes 1 \mapsto x \otimes 1 \otimes \sigma^*F_{\infty}y$, where $x \in \omega_{\dR}^{Y}(\mathcal{V})$, $y \in \mathcal{P}^{\fm,X,Y}_{\mathcal{H}(S)}$, $\sigma^*F_{\infty}y \in \mathcal{P}^{\fm,X,Y}_{\mathcal{H}(\overline{S})}$. If $C^{\fm}_{\mathcal{V}}$ is the \textit{universal period matrix}, i.e. the matrix representing $c^{\fm}_{\mathcal{V}}$, then we have $\textrm{s}^{\fm}_{\mathcal{V}}$ is represented by
$$
 \left(\sigma^*F_{\infty}C^{\fm}_{\mathcal{V}}\right)^{-1}C^{\fm}_{\mathcal{V}}.
$$
Applying the period morphism, and defining $\textrm{s} = \textrm{per}(\textrm{s}^{\fm})$, we get that $\textrm{s}$ is represented by
\begin{equation}
\label{single-valued-matrix definition}
\textrm{s}_{\mathcal{V}} = \left(\overline{C_{\mathcal{V}}}\right)^{-1}C_{\mathcal{V}},
\end{equation}
where $C_{\mathcal{V}}$ is the matrix representing the comparison isomorphism for the object $\mathcal{V} \in \textrm{Ob}(\mathcal{H}(S))$. Note that this map is indeed invariant under the change of Betti basis, as this amounts to replacing $C_{\mathcal{V}}$ with $PC_{\mathcal{V}}$ for some $P \in \textrm{GL}(\omega^X_{\textrm{B}}(\mathcal{V});\QQ)$, and complex conjugation acts trivially on the coefficients of $P$ as they are rational.

\begin{ex}
We will look at the motivic logarithm. Let $\mathcal{V} \in \textrm{Ob}(\mathcal{H}(S))$, for $S = \Pp^1 \setminus \{0,1,\infty\}$, be the object whose fibers are 
$$
V_x = (H^1_{\Beta}(\mathbb{G}_m,\{1,x\}),H^1_{\dR}(\mathbb{G}_m,\{1,x\}),c) \in \mathcal{H}(\Sp(\QQ)), \text{ for } x \in S(\QQ).
$$
We can write down the universal period matrix associated to $\mathcal{V} \in \textrm{Ob}(\mathcal{H}(S))$:
\begin{equation}
c^{\fm}_{\mathcal{V}} = 
\begin{pmatrix}
\mathbb{L}^{\fm} &  0\\
\log^{\fm}(x) & 1
\end{pmatrix}
\end{equation}
Computing $(\overline{c^{\fm}_{\mathcal{V}}})^{-1}c^{\fm}_{\mathcal{V}}$ we get
$$
\textrm{s}^{\fm}(\log^{\fdr}(x)) = \log^{\fm}(x) + \overline{\log^{\fm}}(x).
$$
Applying the period map to the right hand side we get $2\log|x|$. More generally one can define motivic multiple polylogarithms and their de Rham versions as motivic and de Rham periods of the motivic fundamental groupoid of $\mathbb{P}^1\setminus \{0,1,\infty\}$ (see \cite[\S 10.6.2]{Brown1}). An analogous computation to the one for the logarithm then leads to
$$
\textrm{s}^{\fm}(\textrm{Li}_2^{\fdr}(x)) = \textrm{Li}_2^{\fm}(x) - \overline{\textrm{Li}_2^{\fm}}(x) + (\log^{\fm}(x) + \overline{\log^{\fm}}(x))\overline{\textrm{Li}_1^{\fm}}(x) .
$$
The image of the right hand side under the period homomorphism is $2i$ times the Bloch-Wigner dilogarithm.
\end{ex}
\begin{rmk}
For some general results on the single-valued pairing (over a fixed subfield of $\CC$) see \cite{BD}. In its sequel \cite{BD2} the authors use these results to prove a conjecture of Stieberger that relates open and closed string amplitudes at tree level via the single-valued morphism, as well as to show that the KLT formula expressing closed string amplitudes as quadratic expressions in open string amplitudes follows from a variant of the single-valued formalism for cohomology with coefficients in a local system.
\end{rmk}
\end{section}

\appendix
\begin{section}{Residues of Feynman integrands of the sunrise with subdivided edges}
\label{Residues of integrands}

In proposition \ref{basis-elliptic-dr} we consider a basis 
$$
\{[df_1],\ldots,[df_5],[res(\eta_G)],[res(\omega_G)]\} \text{ of } H^1_{\dR}(\mathcal{E}, D \cap \mathcal{E}) \, ,
$$
and we write the residues of forms $\nu_1,\nu_2,\nu_3$, defined in the statement of the proposition, in this basis. The following are the coefficients for each residue form.

\paragraph{Coefficients of $res(\nu_1)$ :}
\begin{equation*}
\begin{split}
    a_{1,1} &= \frac{1}{2d_1}((m_1^2 - m_3^2)(m_1^2 - m_2^2 + m_3^2 + q_1^2)),\\
    a_{1,2}&= \frac{-1}{d_1}((m_1^2 - m_2^2 + m_3^2 + q_1^2)(-m_1^2 + m_2^2 + m_3^2 + q_1^2)),\\
    a_{1,3}&=\frac{-1}{2m_2^2q_1^2}, \quad a_{1,4} = a_{1,1} - a_{1,3}, \quad a_{1,5} = a_{1,2} + a_{1,3}, \\
    a_{1,6} &=\frac{-1}{d_1}m_1^2(m_1^6 - 3m_1^4m_2^2 - 3m_1^4m_3^2 + 3m_1^4q_1^2 + 3m_1^2m_2^4 + 2m_1^2m_2^2m_3^2 - 2m_1^2m_2^2q_1^2 + \\
    &+3m_1^2m_3^4 - 2m_1^2m_3^2q_1^2 + 3m_1^2q_1^4 - m_2^6 + m_2^4m_3^2 - m_2^4q_1^2 + m_2^2m_3^4 + 10m_2^2m_3^2q_1^2 + \\
    &+ m_2^2q_1^4 - m_3^6 - m_3^4q_1^2 + m_3^2q_1^4 + q_1^6),\\
    a_{1,7} &=\frac{1}{d_1}(m_1^6 - 2m_1^4m_2^2 - 2m_1^4m_3^2 + 2m_1^4q_1^2 + m_1^2m_2^4 + 2m_1^2m_2^2m_3^2 - 2m_1^2m_2^2q_1^2 + \\
    & +m_1^2m_3^4 - 2m_1^2m_3^2q_1^2 + m_1^2q_1^4 + 4m_2^2m_3^2q_1^2), \text{ where }\\
    d_1 &= 2m_2^2q_1^2(3m_1^4 - 2m_1^2m_2^2 - 2m_1^2m_3^2 + 2m_1^2q_1^2 - m_2^4 + 2m_2^2m_3^2 - 2m_2^2q_1^2 - m_3^4 \\ 
    & - 2m_3^2q_1^2 - q_1^4)
\end{split}
\end{equation*}

\paragraph{Coefficients of $res(\nu_2)$ :}
\begin{equation*}
\begin{split}
    a_{2,1} &= \frac{1}{2d_2}(2m_1^2 - 2m_3^2) \\
    a_{2,2}&= \frac{1}{m_1^2d_2}(-m_1^4 - 4m_1^2q_1^2 + m_2^4 - 2m_2^2m_3^2 + 2m_2^2q_1^2 + m_3^4 + 2m_3^2q_1^2 + q_1^4)\\
    a_{2,3}&=\frac{-1}{2m_1^2q_1^2}, \quad a_{2,4} = a_{2,1} - a_{2,3}, \quad a_{2,5} = a_{2,2} + a_{2,3}, \\
    a_{2,6} &=\frac{1}{d_2}(m_1^6 - m_1^4m_2^2 - 3m_1^4m_3^2 + m_1^4q_1^2 - m_1^2m_2^4 - 2m_1^2m_2^2m_3^2 - 10m_1^2m_2^2q_1^2 + \\
    & + 3m_1^2m_3^4 + 2m_1^2m_3^2q_1^2 - m_1^2q_1^4 + m_2^6 - 3m_2^4m_3^2 + m_2^4q_1^2 + 3m_2^2m_3^4 + 2m_2^2m_3^2q_1^2 - \\
    & - m_2^2q_1^4 - m_3^6 - 3m_3^4q_1^2 - 3m_3^2q_1^4 - q_1^6)\\
    a_{2,7} &=\frac{-1}{d_2}(m_1^4 - 2m_1^2m_3^2 - m_2^4 - 2m_2^2q_1^2 + m_3^4 - q_1^4), \text{ where } \\
    d_2 & = 2q_1^2(3m_1^4 - 2m_1^2m_2^2 - 2m_1^2m_3^2 + 2m_1^2q_1^2 - m_2^4 + 2m_2^2m_3^2 - 2m_2^2q_1^2 - m_3^4 - 2m_3^2q_1^2 - q_1^4)
\end{split}
\end{equation*}

\paragraph{Coefficients of $res(\nu_3)$ :}
\begin{equation*}
\begin{split}
    a_{3,1} &= \frac{1}{2d_3}(2m_1^4 - m_1^2m_2^2 - m_1^2m_3^2 + 3m_1^2q_1^2 - m_2^4 + 2m_2^2m_3^2 - 2m_2^2q_1^2 - m_3^4 - 2m_3^2q_1^2 - q_1^4)\\
    a_{3,2}&= \frac{1}{d_3}(-m_1^4 + 2m_1^2m_2^2 - 2m_1^2m_3^2 + 2m_1^2q_1^2 - m_2^4 + 2m_2^2m_3^2 - 2m_2^2q_1^2 - m_3^4 - 2m_3^2q_1^2 - q_1^4)\\
    a_{3,3}&=\frac{-1}{2m_3^2q_1^2}, \quad a_{3,4} = a_{3,1} - a_{3,3}, \quad a_{3,5} = a_{3,2} + a_{3,3}, \\
    a_{3,6} &=\frac{1}{d_3}(m_1^2(m_1^6 - 3m_1^4m_2^2 - m_1^4m_3^2 + m_1^4q_1^2 + 3m_1^2m_2^4 - 2m_1^2m_2^2m_3^2 + 2m_1^2m_2^2q_1^2 - \\
    & - m_1^2m_3^4 - 10m_1^2m_3^2q_1^2 - m_1^2q_1^4 - m_2^6 + 3m_2^4m_3^2 - 3m_2^4q_1^2 - \\
    & - 3m_2^2m_3^4 + 2m_2^2m_3^2q_1^2 - 3m_2^2q_1^4 + m_3^6 + m_3^4q_1^2 - m_3^2q_1^4 - q_1^6))\\
    a_{3,7} &=\frac{-1}{d_3}(m_1^2(m_1^4 - 2m_1^2m_2^2 + m_2^4 - m_3^4 - 2m_3^2q_1^2 - q_1^4)), \text{ where } \\
    d_3 &= 2m_3^2q_1^2(3m_1^4 - 2m_1^2m_2^2 - 2m_1^2m_3^2 + 2m_1^2q_1^2 - m_2^4 + 2m_2^2m_3^2 - 2m_2^2q_1^2 - m_3^4 \\
    & - 2m_3^2q_1^2 - q_1^4)
\end{split}
\end{equation*}
\end{section}

\bibliographystyle{abbrv}
\bibliography{bibliography}
\end{document}